\documentclass[review,onefignum,onetabnum]{siamart190516}

\usepackage{fullpage}
\usepackage{amsfonts,physics}
\usepackage{epstopdf}

\usepackage{lineno,amsmath}
\usepackage{comment}
\includecomment{confidential}

\newcommand{\Og}{\Omega}

\newcommand{\Dt}{\Delta}
\newcommand{\be}{\begin{equation}}
\newcommand{\ee}{\end{equation}}
\newcommand{\ba}{\begin{array}}
\newcommand{\ea}{\end{array}}
\newcommand{\bea}{\begin{eqnarray}}
\newcommand{\eea}{\end{eqnarray}}
\newcommand{\beas}{\begin{eqnarray*}}
\newcommand{\eeas}{\end{eqnarray*}}
\usepackage{graphics,color,cancel}
\usepackage{amsmath,amscd,amssymb,bm}
\usepackage{cite}
\usepackage{fancyhdr,fancybox}
\usepackage{listings}
\usepackage{url} 
\usepackage{enumerate,epsfig}
\includecomment{confidential}
\excludecomment{confidential}
\usepackage{array}
\usepackage{multirow}
\usepackage{subfig,float}
\ifpdf
  \DeclareGraphicsExtensions{.eps,.pdf,.png,.jpg}
\else
  \DeclareGraphicsExtensions{.eps}
\fi

\newsiamremark{remark}{Remark}
\newsiamremark{hypothesis}{Hypothesis}
\crefname{hypothesis}{Hypothesis}{Hypotheses}
\newsiamthm{claim}{Claim}

\headers{Variable Time Step Method of DLN for A Corrected Smagorinsky Model }{F. Siddiqua, and W. Pei}

\title{Variable Time Step Method of Dahlquist, Liniger and Nevanlinna (DLN) for A Corrected Smagorinsky Model  \thanks{
The research was partially supported by NSF grant DMS-2110379.}}
\author{Farjana Siddiqua\thanks{Department of Mathematics, University of Pittsburgh, Pittsburgh, PA 15260, USA. 
  Email: fas41@pitt.edu.}
\and Wenlong Pei\thanks{Department of Mathematics, The Ohio State University, Columbus, OH 43210, USA. 
  Email: pei.176@osu.edu.}
}

\usepackage{amsopn}

\ifpdf
\hypersetup{
  pdftitle={Variable Time Step Method of DAHLQUIST, LINIGER AND
NEVANLINNA (DLN) for a Corrected Smagorinsky Model },
  pdfauthor={Farjana Siddiqua and Wenlong Pei}
}
\fi
\begin{document}
\nolinenumbers
\maketitle
\begin{abstract}
 Turbulent flows  strain resources, both memory and CPU speed. The DLN method has greater accuracy and allows larger time steps, requiring less memory and fewer FLOPS. The DLN method can also be implemented adaptively. The classical Smagorinsky model, as an effective way to approximate a (resolved) mean velocity, has recently been corrected to represent a flow of energy from unresolved fluctuations to the (resolved) mean velocity. In this paper, we apply a family of second-order, $G$-stable time-stepping methods proposed by Dahlquist, Liniger, and Nevanlinna (the DLN method) to one corrected Smagorinsky model and provide the detailed numerical analysis of the stability and consistency. We prove that the numerical solutions under any arbitrary time step sequences are unconditionally stable in the long term and converge at second order. We also provide error estimate under certain time step condition. Numerical tests are given to confirm the rate of convergence and also to show that the adaptive DLN algorithm helps to control numerical dissipation so that backscatter is visible.
\end{abstract}
\section{Introduction}
Eddy viscosity (EV) models are the most common approaches to depict the average of turbulent flow of Navier-Stokes equations (NSE). Various eddy viscosity models in practical setting are proposed for analytical and numerical study \cite{BFR80_13FPCon,FP99_Springer,MR1753115,IJLMT03_IJCFD,IL98}. Unfortunately, most EV models have difficulties in simulating backscatter or complex turbulent flow not at statistical equilibrium due to neglection of the intermittent energy flow from fluctuations back to means. To overcome this defect, Jiang and Layton \cite{jiang2016ev} calibrated the standard eddy viscosity model by fitting the turbulent viscosity coefficient to flow data. Rong, Layton, and Zhao \cite{rong2019extension} extended the usual Baldwin-Lomax model so that the new model can account for statistical backscatter without artificial negative viscosities. Recently, Siddiqua and Xie \cite{siddiqua2022numerical} have corrected the classical Smagorinsky model \cite{Smag63} with no new fitting parameters to reflect a flow of energy from unresolved fluctuations to means in the corrected Smagorinsky model (CSM henceforth). Most recently, Dai, Liu, Liu, Jiang, and Chen \cite{dai2023development} proposed a new dynamic Smagorinsky model by an artificial neural network for prediction of outdoor airflow and pollutant dispersion. Herein we give an analysis of the method of Dahlquist, Liniger, and Nevanlinna \cite{DLN83_SIAM_JNA} (the DLN method henceforth) for the CSM \cite{siddiqua2022numerical} with variable time steps. Let $f(x,t)$ be the prescribed body force, $\nu$ be the kinematic viscosity in the regular and bounded flow domain $\Omega\subset \mathbb{R}^d\ (d=2,\ 3)$. We analyze the variable step, DLN time discretization for the CSM: $\div{ w}=0$ and
\begin{equation}\label{csm0}
\begin{aligned}
& w_t-C_s^4\delta^2\mu^{-2}\Dt w_t+ w\cdot\grad w-\nu\Dt w+\grad q-\div\Big((C_s\delta)^2|\grad w|\grad w\Big)=f, \   x\in\Omega, \ 0\leq t\leq T.
\end{aligned}
\end{equation}
Here $\mu$ is a constant from Kolmogorov-Prandtl relation  \cite{kolmogrov,prandtl} and $( w,q)$ approximate an ensemble average pair of velocity and pressure of Navier-Stokes solutions, $(\overline{u},\overline{p})$. This is an eddy viscosity model with turbulent viscosity, $\nu_T=(C_s\delta)^2|\grad w|$, where $C_s\approx 0.1$, (suggested by Lilly \cite{lilly}), $\delta$ is a length scale (or grid-scale). In \cite{siddiqua2022numerical}, the CSM model derivation and some basic properties of the CSM are developed and two algorithms for its numerical simulation are proposed. However, the significant backscatter of model dissipation is not observed in specific examples except for Linearized Crank-Nicolson time discretization. Besides that, constant time discretization in their algorithms (backward Euler and Linearized Crank-Nicolson time-stepping schemes) excludes the use of time adaptivity since the solution pattern (in terms of stability and convergence) under extreme time step ratios is hard to expect. Dahlquist, Liniger, and Nevanlinna designed a one-parameter family of one-leg, second-order method for evolutionary equations \cite{DLN83_SIAM_JNA}. This time-stepping method (the DLN method) is proved to be $G$-stable (non-linear stable) under any arbitrary time grids \cite{Dah76_Tech_RIT,Dah78_BIT,Dah78_AP_NYL} and hence ideal choice for time discretization of fluid models \footnote[1]{To our knowledge, the DLN method is the \textbf{only} variable multi-step method which are both non-linear stable and second-order accurate.}. Herein we apply the fully-discrete DLN algorithm (finite element space discretization) for the CSM in \eqref{csm0} and present a complete numerical analysis of the algorithm. We prove that the numerical solutions on arbitrary time grids are unconditionally, long-term stable, and converge to exact solutions at second order with moderate time step restrictions. Let $\{ t_{n} \}_{n=0}^{N}$ be the time grids on interval $[0,T]$ 
and $k_{n} = t_{n+1} - t_{n}$ is the local time step. Let $w_{n}^{h}$ and $q_{n}^{h}$ be the numerical approximations of velocity and pressure at time $t_{n}$ of the CSM in \eqref{csm0} respectively on certain finite element space with the diameter $h$, the fully discrete DLN algorithm (with parameter $\theta \in [0,1]$) for the CSM in \eqref{csm0} is written as follows: $\nabla \cdot w_{n+1}^{h} = 0$ and
\begin{equation}\label{eq:DLN-CSM-Alg}
	\begin{aligned}
		&\frac{\alpha_2 w_{n+1}^h+\alpha_1 w_{n}^h+\alpha_0 w_{n-1}^h}{\alpha_{2}k_{n} - \alpha_{0}k_{n-1}} 
		- \frac{C_s^4\delta^2}{\mu^2} \Delta \Big(
		\frac{\alpha_2 w_{n+1}^h + \alpha_1 w_{n}^h + \alpha_0 w_{n-1}^h}{\alpha_{2}k_{n} - \alpha_{0}k_{n-1}} \Big)
		 \\&
		+ \Big( \sum_{\ell=0}^{2} \beta_{\ell}^{(n)} w_{n-\ell}^h \Big) \cdot 
		\nabla \Big( \sum_{\ell=0}^{2} \beta_{\ell}^{(n)} w_{n-\ell}^h \Big) 
		- \nu \Delta \Big( \sum_{\ell=0}^{2} \beta_{\ell}^{(n)} w_{n-\ell}^h \Big) 
		+ \nabla \Big( \sum_{\ell=0}^{2} \beta_{\ell}^{(n)} q_{n-\ell}^h \Big)
		 \\&
		+\nabla \cdot \Big((C_s\delta)^2
		\Big| \nabla  \Big( \sum_{\ell=0}^{2} \beta_{\ell}^{(n)} w_{n-\ell}^h \Big) \Big| 
		\nabla \Big( \sum_{\ell=0}^{2} \beta_{\ell}^{(n)} w_{n-\ell}^h \Big)  \Big)
		= f \Big( \sum_{\ell=0}^{2} \beta_{\ell}^{(n)} t_{n-\ell}^h \Big),  \quad \text{for } 1 \leq n \leq N-1,  
	\end{aligned}
 \end{equation}
	where $| \cdot |$ is the Euclidian norm on $\mathbb{R}^{d}$ and the coefficients in \eqref{eq:DLN-CSM-Alg} are
	\begin{gather*}
		\begin{bmatrix}
			\alpha _{2} \vspace{0.2cm} \\
			\alpha _{1} \vspace{0.2cm} \\
			\alpha _{0} 
		\end{bmatrix}
		= 
		\begin{bmatrix}
			\frac{1}{2}(\theta +1) \vspace{0.2cm} \\
			-\theta \vspace{0.2cm} \\
			\frac{1}{2}(\theta -1)
		\end{bmatrix}, \ \ \ 
		\begin{bmatrix}
			\beta _{2}^{(n)}  \vspace{0.2cm} \\
			\beta _{1}^{(n)}  \vspace{0.2cm} \\
			\beta _{0}^{(n)}
		\end{bmatrix}
		= 
		\begin{bmatrix}
			\frac{1}{4}\Big(1+\frac{1-{\theta }^{2}}{(1+{%
					\varepsilon _{n}}{\theta })^{2}}+\varepsilon _{n}^{2}\frac{\theta (1-{%
					\theta }^{2})}{(1+{\varepsilon _{n}}{\theta })^{2}}+\theta \Big)\vspace{0.2cm%
			} \\
			\frac{1}{2}\Big(1-\frac{1-{\theta }^{2}}{(1+{\varepsilon _{n}}{%
					\theta })^{2}}\Big)\vspace{0.2cm} \\
			\frac{1}{4}\Big(1+\frac{1-{\theta }^{2}}{(1+{%
					\varepsilon _{n}}{\theta })^{2}}-\varepsilon _{n}^{2}\frac{\theta (1-{%
					\theta }^{2})}{(1+{\varepsilon _{n}}{\theta })^{2}}-\theta \Big)%
		\end{bmatrix}.
	\end{gather*}
	The step variability $\varepsilon _{n} = (k_n - k_{n-1})/(k_n + k_{n-1})$ is the function of two step sizes and $\varepsilon_n\in(-1,1)$. 

The main result of this article is the complete numerical analysis of the DLN method and computational tests showing backscatter phenomena for the CSM model \eqref{csm0}. The paper is organized as follows. We provide some necessary notations and preliminaries for numerical analysis in \Cref{sec:np}. We present the fully discrete variational formulation in \Cref{sec:fd}. We show that the DLN solutions are long-term, unconditional stable in \Cref{thm:Stab-Uncond} of \Cref{sec:Stability} and perform the variable step error analysis with the moderate time step restriction in \Cref{thm:numerial-error} of \Cref{sec:ne}. Furthermore, in \Cref{sec:firstexample}, we present the test problem with exact solutions \cite{Victor2017artificial} to confirm the fully discrete DLN algorithm is second-order in time and in \Cref{sec:secondexample}, we present the test problem about flow between offset cylinders \cite{JL14_IJUQ} to check the unconditional stability and the efficiency of the time adaptivity of the DLN algorithm.
\subsection{Related Work}
	Due to the fine properties about the stability and consistency, the whole DLN family calls great attention in the simulation of evolutionary equations and fluid models. The DLN method with $\theta = \frac{2}{3}$ is suggested in \cite{DLN83_SIAM_JNA} to relieve the conflict between error and stability. Kulikov and Shindin find that the DLN method with $\theta = \frac{2}{\sqrt{5}}$ have the best stability at infinity \cite{kulikov2005stable}. The midpoint rule (the DLN method with $\theta = 1$), conserving all quadratic Hamiltonians, has been thoroughly studied and widely used in computational fluid dynamics \cite{AP98_SIAM,HNW93I_Springer,BT20_AML,LLMNR09_CMAME,BST21_JMFM,BT21_CMAME,BPT22_IJNAM}. Recently, the whole DLN family is applied to some time-dependent fluid model and shows its outstanding performance in some specific examples \cite{LPQT21_NMPDE,QHPL21_JCAM,QCWLL23_ANM}. In addition, the DLN implementation has been simplified by the re-factorization process (adding time-filters on backward Euler method) for wide application \cite{LPT21_AML}. Time adaptivity of the DLN method (by the local truncation error criterion) is proposed to solve stiff differential systems for both efficiency and accuracy \cite{LPT22_Tech}.
\section{Notations and preliminary results}\label{sec:np}
${\empty}$
In this section, we introduce some of the notations and results used in this paper.
Recall that $\Omega\subset \mathbb{R}^d\ (d=2,\ 3)$ is the bounded domain of the CSM in \cref{csm0}. Banach space $L^p(\Omega)\ (p\geq 1)$ contains all Lebesgue measurable function $f$ such that $|f|^p$ is integrable. For $r\in\{0\}\cup \mathbb{N}$, Sobolev space $W^{m,p}(\Omega)$ with norm $\|.\|_{m,p}$ contains all functions whose weak derivatives up to $m$-th belong to $L^{p}(\Omega)$. Thus $W^{m,p}(\Omega)$ is exactly $L^{p}$ when $m=0$. We use $H^{m}$ with norm $\| \cdot \|_{m}$ and semi-norm $| \cdot |_{m}$ to denote the inner product space $W^{m,2}(\Omega)$. $\|\cdot\|$ and $(\cdot,\cdot)$ denote the $L^2(\Omega)$ norm and inner product, respectively.
The solution spaces $X$ for the velocity and $Q$ for the pressure are defined as:
\begin{gather*}
		X= \Big\{ v\in \big(L^3(\Omega) \big)^{d}: \nabla  v \in \big(L^3(\Omega)\big)^{d \times d}, \ v \big|_{\partial\Omega} = 0 \Big\}, \ \ \ 
		Q = \Big\{ q\in L^2(\Omega): \int_\Omega q\ d x=0 \Big\},
	\end{gather*}
	and the divergence-free velocity space is 
	\begin{gather*}
		V = \big\{ v \in X: (q,\nabla \cdot  v)=0,\ \forall q \in Q \big\}.
\end{gather*}
$X'$ is the dual norm of $X$ with the dual norm 
	\begin{gather*}
		\|f\|_{-1}=\sup_{0 \neq v \in X}\frac{(f, v)}{\|\nabla  v\|}, \qquad \forall f \in X'.
	\end{gather*}
\begin{definition}\label{trilinear0}
(Trilinear Form)
Define the trilinear form $b^*:X\times{X}\times{X}\rightarrow \mathbb{R}$ as follows 
$$b^*( u, v, w):=\frac{1}{2}( u\cdot \grad  v, w)-\frac{1}{2}( u\cdot \grad  w, v),\qquad
		\forall u,v,w \in X.$$
\end{definition}
\begin{lemma} \label{trilinear1}
The nonlinear term $b^*(\cdot,\cdot,\cdot)$ is continuous on $X\times X\times X$ (and thus on $V\times V\times V$) which has the following skew-symmetry property, 
\begin{gather}
		\label{eq:Skew-Symmetry}
		b^{\ast}( u, v, w)=-b^{\ast}( u, w, v), \qquad b^*( u, v, v)=0.
	\end{gather}
As a consequence, we get
\begin{align*}
b^*( u, v, w)=( u\cdot \grad  v, w),\ \quad &\forall\  u\in V \text{and} \  v,\  w\in X, \\
b^*( u, v, v)=0,\quad &\forall\  u,\  v\in X.
\end{align*}
\end{lemma}
\begin{proof}
Proof of this lemma is standard, see $p.114$ of Girault and Raviart \cite{GandR}.
\end{proof}
\begin{lemma}
		\label{lemma:trilinear-ineq}
		For any $ u,\  v,\  w\in X$ 
		\begin{align}
			\label{eq:b-bound}
			b^{\ast}(u, v, w) &\leq C(\Omega) \|\nabla u \| \|\nabla v \|\|\nabla w \|, 
			\notag \\
			b^{\ast}(u, v, w) &\leq C(\Omega) \| u \|^{1/2} \|\nabla u \|^{1/2} \|\nabla v\| \|\nabla w \|. 
		\end{align}
	\end{lemma}
	\begin{proof}
		By H$\rm{\ddot{o}}$lder's inequality, Poincar\'{e}-Friedrichs's inequality and Ladyzhenskaya's inequality.
	\end{proof}
Next is a Discrete Gr\"{o}nwall Lemma, see \cite[Lemma 5.1, p.369]{heywood1990}.
 \begin{lemma}\label{lemma:Gronwall}
    	Let $\Delta t,\ B$ be non-negative real numbers and $\{ a_n \}_{n=0}^{\infty}$, $\{ b_n \}_{n=0}^{\infty}$,
		$\{ c_n \}_{n=0}^{\infty}$, $\{ d_n \}_{n=0}^{\infty}$ be non-negative sequences of real numbers such that 
		\begin{gather*}
    		a_\ell + \Delta t \sum_{n=0}^\ell b_n \leq \Delta t \sum_{n=0}^{\ell} d_n a_n 
    		+ \Delta t \sum_{n=0}^\ell c_n + B, \qquad \forall \ell \in \mathbb{N},
    	\end{gather*}
    	and $\Delta t d_n <1$ for all $n$, then 
    	\begin{gather*}
    		a_\ell + \Delta t \sum_{n=0}^\ell b_n 
    		\leq \exp \Big( \Delta t \sum_{n=0}^{\ell} \frac{d_n}{1 - \Delta t d_n} \Big)
    		\Big( \Delta t \sum_{n=0}^\ell c_n+B\Big), \qquad \forall \ell \in \mathbb{N}.
    	\end{gather*}
\end{lemma}
\begin{proof}
		See \cite[p.369]{heywood1990}.
\end{proof}
\begin{lemma}\label{lemma:Monotone-LLC}
({\bf Strong Monotonicity (SM) and Local Lipschitz Continuity (LLC)})\\There exists $C_1, \ C_2>0$ such that for all $ u,\  v,\  w\in W^{1,3}(\Omega),$
\begin{align}
    \textbf{(SM)}&\quad  (|\nabla  u| \nabla  u \!-\! |\nabla  w| \nabla  w, \nabla( u- w))
			\!\geq \! C_1 \|\nabla( u- w) \|_{0,3}^3,\label{eq:Strong-Monotone} \\ 
    \textbf{(LLC)}&\quad  (|\nabla  u| \nabla  u \!-\! |\nabla  w|\nabla  w,\nabla  v)
			\!\leq \! C_2 \big(\! \max\{\|\nabla  u\|_{0,3},\|\nabla  w\|_{0,3}\}  \!\big)
			\|\nabla( u - w)\|_{0,3} \| \nabla  v\|_{0,3}.\label{eq:LLC}
\end{align}
\end{lemma}
\begin{proof}
		We refer \cite{Layton2002error,Lay96_SIAM_JSC,lady1991} for proof.
	\end{proof}
Let $\mathcal{T}_{h}$ be the edge-to-edge triangulation of the domain $\Omega$ with diameter $h>0$. $X^{h} \subset X$ and $Q^{h} \subset Q$ are certain finite element spaces of velocity and pressure respectively. The divergence-free subspace of $X^{h}$ is 
\begin{gather*}
		V^{h}:= \Big\{ v^{h} \in X^{h}: (p^{h}, \nabla \cdot v^{h} ) = 0, \quad \forall p^{h} \in Q^{h} \Big\}.
\end{gather*}
Given $(w,q) \in X \times Q$, the finite element pair $(X^{h},Q^{h})$ satisfies the approximation 
theorem (See \cite{Cia78_NHPC,BS08_SpringerNY}): for any $r,s \in \{ 0 \} \cup \mathbb{N}$ and $\ell \in \{0, 1\}$,
\begin{align}
\label{eq:approx-thm}
		\inf_{v^{h} \in X^{h}} \| w - v^{h} \|_{\ell} &\leq C h^{r+\ell-1} \| w \|_{r+1}, 
		\qquad \text{for } \ u \in (H^{r+1})^{d} \cap X,  \notag \\
		\inf_{p^{h} \in Q^{h}} \| q - p^{h} \| &\leq C h^{s+1} \| q \|_{s+1}, 
		\qquad \quad \text{for } \  q \in H^{s+1} \cap Q,   
\end{align}
where $r$ and $s$ are highest degree of polynomials for $X^{h}$ and $Q^{h}$ respectively. We need the $L^{p}-L^{2}$-type inverse inequality \cite{Lay96_SIAM_JSC}.
\begin{theorem}
\label{thm:Lp-L2}
Let $\Theta$ be the minimum angle in the triangulation of domain $\Omega \subset \mathbb{R}^{d}$ ($d = 2,3$) 
and $X^{h}$ be the finite element space with highest polynomial degree $r$. For any $v^{h} \in X^{h}$ and $2 \leq p < \infty$, there is a constant $C = C(\Theta,p,r)>0$ such that
\begin{gather}
			\label{eq:Lp-L2}
			\| \nabla^{h} v^{h} \|_{0,p} \leq C h^{\frac{d}{2}(\frac{2-p}{p})} \| \nabla^{h} v^{h} \|,
\end{gather}
where $\nabla^{h}$ is the element-wise defined gradient operator.
\end{theorem}
\begin{proof}
See \cite[p.349-350]{Lay96_SIAM_JSC} for proof.
\end{proof}
We assume that $(X^h, Q^h)$ satisfies the discrete inf-sup condition:
\begin{gather*}
		\inf_{p^h\in Q^h}\sup_{ v^h\in X^h}\frac{(p^h,\nabla \cdot  v^h)}{\|p^h\|\|\nabla  v^h\|}\geq C>0,
\end{gather*}
where $C>0$ is some constant independent of $h$. We define the Stokes projection $\Pi : V \times Q \rightarrow V^{h} \times Q^{h}$ as follows: given the pair $(w,q) \in V \times Q$, the Stokes projection $\Pi(w,q) = \big( \mathcal{W}, \mathcal{Q} \big)$ satisfies 
\begin{equation}\label{eq:Stokes-def}
\begin{aligned}
		&\nu (\nabla w, \nabla v^{h}) - (q, \nabla \cdot v^{h}) 
		= \nu (\nabla \mathcal{W}, \nabla v^{h}) - (\mathcal{Q}, \nabla \cdot v^{h}),  \\&
		(p^{h}, \nabla \cdot \mathcal{W}) = 0, \qquad \forall (v^{h}, p^{h}) \in X^{h} \times Q^{h}.
\end{aligned}
\end{equation}
The above Stokes projection on $V^{h} \times Q^{h}$ is well defined since $X \subset (H_{0}^{1}(\Omega))^{d}$ if the domain $\Omega$ is bounded. We need the following approximation of Stokes projection (see \cite{GR86_Springer,Joh16_Springer} for proof)
\begin{equation}
		\label{eq:Stoke-Approx}
  \begin{aligned}
		&\| w - \mathcal{W} \| \leq C h \big( \nu^{-1} \inf_{q^{h} \in Q^{h}} \| q - p^{h} \| 
		+ \inf_{v^{h} \in X^{h}} | w - v^{h} |_{1} \big),   \\&
		\| w - \mathcal{W} \|_{1} \leq C \big( \nu^{-1} \inf_{q^{h} \in Q^{h}} \| q - p^{h} \| 
		+ \inf_{v^{h} \in X^{h}} | w - v^{h} |_{1} \big). 
\end{aligned}
\end{equation}
\section{The variable step DLN method for CSM}\label{sec:fd}
We denote $w(t_{n})$ by $w_{n}$ and $q(t_{n})$ by $q_{n}$ in 
the CSM in \eqref{csm0}. $w_{n}^{h} \in X^{h}$ and $q_{n}^{h} \in Q^{h}$ represent the DLN solutions of $w_{n}$ and $q_{n}$ respectively. 
 	For convenience, we denote 
	\begin{gather*}
		t_{n,\beta} = \sum_{\ell=0}^{2} \beta_{\ell}^{(n)} t_{n-1+\ell}, \qquad
		w_{n,\beta} = \sum_{\ell=0}^{2} \beta_{\ell}^{(n)} w(t_{n-1+\ell}), \qquad
		w_{n,\beta}^{h} = \sum_{\ell=0}^{2} \beta_{\ell}^{(n)} w_{n-1+\ell}^{h}, \\
		q_{n,\beta} = \sum_{\ell=0}^{2} \beta_{\ell}^{(n)} q(t_{n-1+\ell}), \qquad
		q_{n,\beta}^{h} = \sum_{\ell=0}^{2} \beta_{\ell}^{(n)} q_{n-1+\ell}^{h}, \qquad
		f_{n,\beta} = \sum_{\ell=0}^{2} \beta_{\ell}^{(n)} f(t_{n-1+\ell}),
	\end{gather*}
and represent the average time step $\alpha_{2}k_{n} - \alpha_{0}k_{n-1}$ by $\widehat{k}_{n}$. The variational formulation of the variable time-stepping DLN scheme (with grad-div stabilizer \cite{CELR11_SIAMNA}) in \eqref{eq:DLN-CSM-Alg} is: given $ w_n^h,\  w_{n-1}^h\in X^h$ and $q_{n}^h,\ q_{n-1}^h\in Q^h$, find $ w_{n+1}^h$ and $q_{n+1}^h$ satisfying
\begin{equation}\label{eq:weak-DLN-CSM-Alg}
\begin{aligned}
    &\Big(\frac{\alpha_2 w_{n+1}^h+\alpha_1 w_{n}^h+\alpha_0 w_{n-1}^h}{\widehat{k_n}}, v^h\Big)+\frac{C_s^4\delta^2}{\mu^2} \Big(\frac{\alpha_2 \nabla w_{n+1}^h + \alpha_1 \nabla w_{n}^h +\alpha_0 \nabla w_{n-1}^h}{\widehat{k_n}}, \nabla v^h \Big)
    \\& +\nu(\nabla w_{n,\beta}^h, \nabla v^h) +b^*( w_{n,\beta}^h, w_{n,\beta}^h, v^h)
		+ \gamma (\nabla \cdot w_{n,\beta}^{h}, \nabla \cdot v^{h})-(q_{n,\beta}^h,\nabla \cdot {v^h})
  \\& +\Big((C_s\delta)^2 |\nabla w_{n,\beta}^h |\nabla w_{n,\beta}^h, \nabla v^h\Big) 
		=(f_{n,\beta}, v^h), \qquad \forall  v^h\in X^h, 
  \\& (\nabla \cdot {w_{n,\beta}^h},p^h)=0, \qquad  \forall p^h\in Q^h,
 \end{aligned}
\end{equation}
where constant $\gamma > 0$ needs to be decided by specific problems.	
\begin{confidential}
    \color{blue}
Under the discrete inf-sup condition, (\ref{v}) is equivalent to the following:
\begin{align}\label{v}
 &\Bigg(\frac{\alpha_2 w_{n+1}^h+\alpha_1 w_{n}^h+\alpha_0 w_{n-1}^h}{\widehat{k_n}}, v^h\Bigg)+\nu(\grad w_{n,\beta}^h,\grad  v^h)\notag
 \\&+\frac{C_s^4\delta^2}{\mu^2}\Bigg(\frac{\alpha_2\grad w_{n+1}^h+\alpha_1\grad w_{n}^h+\alpha_0\grad w_{n-1}^h}{\widehat{k_n}},\grad  v^h\Bigg)
 +b^*( w_{n,\beta}^h, w_{n,\beta}^h, v^h)\notag
 \\&+\Big((C_s\delta)^2|\grad w_{n,\beta}^h|\grad w_{n,\beta}^h,\grad  v^h\Big)=(f(t_{n,\beta}), v^h),\ \forall  v^h\in X^h.
\end{align}
\normalcolor
\end{confidential}
Let $\widetilde{w_{n}^h}$ denote the standard (second order) linear extrapolation\cite{layton2020doubly} of $w_{n}^h$
$$\widetilde{w_{n}^h}=\beta_2^{(n)}\Bigg\{ \bigg(1+\frac{k_{n}}{k_{n-1}}\bigg)w_n^h-\bigg(\frac{k_{n}}{k_{n-1}}\bigg)w_{n-1}^h\Bigg\}+\beta_1^{(n)} w_{n}^h+\beta_0^{(n)} w_{n-1}^h.$$
After applying the linearly implicit DLN scheme for time discretization, we get the following  discretization:
\begin{equation}\label{v2}
\begin{aligned}
 &\Bigg(\frac{\alpha_2 w_{n+1}^h+\alpha_1 w_{n}^h+\alpha_0 w_{n-1}^h}{\widehat{k_n}}, v^h\Bigg)+\nu(\grad w_{n,\beta}^h,\grad  v^h)
\\&+\frac{C_s^4\delta^2}{\mu^2}\Bigg(\frac{\alpha_2\grad w_{n+1}^h+\alpha_1\grad w_{n}^h+\alpha_0\grad w_{n-1}^h}{\widehat{k_n}},\grad  v^h\Bigg)
 +b^*( \widetilde{w_{n}^h}, w_{n,\beta}^h, v^h)+\gamma (\nabla \cdot w_{n,\beta}^{h}, \nabla \cdot v^{h})
 \\&-(q_{n,\beta}^h,\div{ v^h})+\Big((C_s\delta)^2|\grad \widetilde{w_{n}^h}|\grad w_{n,\beta}^h,\grad  v^h\Big)=(f(t_{n,\beta}), v^h),\ \forall  v^h\in X^h, \\&
 (\div{ w_{n,\beta}^h},p^h)=0, \  \forall p^h\in Q^h.
\end{aligned}
\end{equation}
\section{Numerical Analysis}
	\label{sec:Numerical-Analysis}
	We define the discrete Bochner space with time grids $\{ t_{n} \}_{n=0}^{N}$ on time interval $[0, T]$,
	\begin{align*}
		\ell^{\infty} \big(0,N;(W^{m,p})^{d} \big) 
		:=& \big\{ f(\cdot, t) \in (W^{m,p})^{d}: \| |f| \|_{\infty,m,p} < \infty \big\}, \\
		\ell^{p_{1},\beta} \big(0,N;(W^{m,p_{2}})^{d} \big)
		:=& \big\{ f(\cdot, t) \in (W^{m,p_{2}})^{d}: \| |f| \|_{p_{1},m,p_{2},\beta} < \infty \big\}, 
	\end{align*} 
	where the corresponding discrete norms are 
	\begin{gather*}
		\| |f| \|_{\infty,m,p}:= \max_{0 \leq n \leq N} \| f(\cdot , t_{n}) \|_{m,p}, \quad 
		\| |f| \|_{p_{1},m,p_{2},\beta}:= 
		\Big( \sum_{n=1}^{N} (k_{n} + k_{n-1}) \|f(\cdot, t_{n,\beta}) \|_{m,p_{2}}^{p_{1}} \Big)^{1/p_{1}}.
	\end{gather*}
\begin{definition}\label{g}
For $0\leq \theta\leq 1$, define the semi-positive symmetric definite matrix $G(\theta)$ by
\begin{equation*}
G(\theta) = 
\begin{bmatrix}
\frac{1}{4}(1+\theta)\mathbb{I}_d & 0 \\
0 & \frac{1}{4}(1-\theta)\mathbb{I}_d 
\end{bmatrix}.
\end{equation*}
\end{definition}
We present two Lemmas about stability and consistency of the DLN method.
	\begin{lemma}
		\label{lemma:G-stable}
		Let $\{ y_{n} \}_{n=0}^{N}$ be any sequence in $L^{2}(\Omega)$. For any $\theta \in [0,1]$ 
		and $n \in \{1, 2, \cdots, N-1 \}$, we have 
		\begin{gather}
			\label{eq:G-identity}
			\Big( \sum_{\ell=0}^{2} \alpha_{\ell} y_{n-1+\ell}, 
			\sum_{\ell=0}^{2} \beta_{\ell}^{(n)} y_{n-1+\ell} \Big)
			= 
			\begin{Vmatrix}
				y_{n+1} \\ y_{n}
			\end{Vmatrix}_{G(\theta)}
			-  
			\begin{Vmatrix}
				y_{n} \\ y_{n-1}
			\end{Vmatrix}_{G(\theta)}
			+ \Big\| \sum_{\ell=0}^{2} \lambda_{\ell}^{(n)} y_{n-1+\ell} \Big\|^{2},
		\end{gather}
		where the $\| \cdot \|_{G(\theta)}$-norm is 
		\begin{gather}
			\label{eq:G-norm-def}
			\begin{Vmatrix}
				u \\ v
			\end{Vmatrix}_{G(\theta)}
			= \frac{1}{4}(1+\theta) \| u \|^{2} + \frac{1}{4}(1-\theta) \| v \|^{2}, 
			\qquad \forall u,v \in L^{2}(\Omega),
		\end{gather}
		and the coefficients $\{ \lambda_{\ell}^{(n)} \}_{\ell=0}^{2}$ are
		\begin{gather}
			\label{eq:G-coefficients}
			\lambda_{1}^{(n)} = - \frac{\theta (1 - \theta^{2})}{\sqrt{2} (1 + \varepsilon_{n} \theta)},
			\qquad \lambda_{2}^{(n)} = - \frac{1 - \varepsilon_{n}}{2} \lambda_{1}^{(n)}, 
			\qquad \lambda_{0}^{(n)} = - \frac{1 + \varepsilon_{n}}{2} \lambda_{1}^{(n)}.
		\end{gather}
	\end{lemma} 
	\begin{proof}
		The proof of identity in \eqref{eq:G-identity} is just algebraic calculation.
	\end{proof}
	\begin{remark}
		If we replace $L^{2}(\Omega)$ by Euclidian space $\mathbb{R}^{d}$, the identity in \eqref{eq:G-identity} still holds and implies $G$-stability of the DLN method. (See \cite[p.2]{Dah76_Tech_RIT} for the definition of $G$-stability.)
	\end{remark}

	\begin{lemma}
		\label{lemma:DLN-consistency}
		Let $Y$ be any Banach space over $\mathbb{R}$ with norm $\| \cdot \|_{Y}$, $\{ t_{n} \}_{n=0}^{N}$ be time grids on time interval $[0,T]$ and $u$ be the mapping from $[0,T]$ to $Y$. 
		We set 
		\begin{gather*}
			k_{\rm{max}} = \max_{0 \leq n \leq N-1} \{ k_{n} \},
		\end{gather*} 
		and assume that the mapping $u(t)$ is smooth enough about the variable $t$, 
		then for any $\theta \in [0,1]$,
		\begin{align}
			\label{eq:DLN-consistency}
			\Big\| \sum_{\ell=0}^{2} \beta_{\ell}^{(n)} u(t_{n-1+\ell}) - u(t_{n,\beta}) \Big\|_{Y}^{2}
			\leq& C(\theta) k_{\rm{max}}^{3} \int_{t_{n-1}}^{t_{n+1}} \| u_{tt} \|_{Y}^{2} dt, \notag \\
			\Big\| \frac{1}{\widehat{k}_{n}} \sum_{\ell=0}^{2} \alpha_{\ell} u(t_{n-1+\ell}) - u_{t}(t_{n,\beta})  \Big\|_{Y}^{2} \leq& C(\theta) k_{\rm{max}}^{3} \int_{t_{n-1}}^{t_{n+1}} \| u_{ttt} \|_{Y}^{2} dt.
		\end{align}
	\end{lemma}
	\begin{proof}
		We use Taylor's Theorem and expand $u(t_{n+1})$, $u(t_{n})$, $u(t_{n-1})$ at $t_{n,\beta}$.
		By H$\ddot{\rm{o}}$lder's inequality, we obtain \eqref{eq:DLN-consistency}.
	\end{proof}
\subsection{Stability of the DLN scheme for the CSM}\label{sec:Stability}
The DLN method is a one parameter family of $A$-stable, 2 step, $G$-stable methods $(0\leq \theta \leq 1)$. It reduces to one-step midpoint scheme if $\theta=1$. The important property of it is $G$-stability matrix does not depend on the time step ratio but on $\theta$ in \cref{lemma:G-stable}. In this section, we prove the unconditional, long time, variable time step energy-stability of \eqref{eq:weak-DLN-CSM-Alg} by using G-stability property (\cref{lemma:G-stable}) of the method.
\begin{theorem}\label{thm:Stab-Uncond}
The one-leg variable time step DLN scheme by \eqref{eq:weak-DLN-CSM-Alg} is unconditionally, long-time stable, i.e. for any integer $N>1$,
\begin{equation}\label{eq:DLN-Stability}
   \begin{aligned}
    		&\frac{1\!+\!\theta}{4} \big( \| w_{N}^h \|^2 \!+\! \frac{C_s^4 \! \delta^2}{\mu^2} \| \nabla w_{N}^h \|^2 \big) 
			+ \frac{1\!-\!\theta}{4} \big( \| w_{N-1}^h \|^2 \!+\! \frac{C_s^4 \! \delta^2}{\mu^2} \| \nabla w_{N-1}^h \|^2 \big)  \\&
    		\!+\! \sum_{n\!=\!1}^{N\!-\!1} \! \Big( \! \big\|\! \sum_{\ell\!=\!0}^2 \lambda_{\ell}^{(n)} w_{n\!-\!1\!+\!\ell}^h \!\big\|^2
			\!+\! \frac{C_s^4 \! \delta^2}{\mu^2} \big\|\! \sum_{\ell\!=\!0}^2 \!\lambda_{\ell}^{(n)} \nabla w_{n\!-\!1\!+\!\ell}^h \!\big\|^2 \! \Big) 
			\!+\!  \sum_{n\!=\!1}^{N\!-\!1} \widehat{k}_{n} \big( \!\frac{\nu}{2} \|\! \nabla \! w_{n\!,\!\beta}^{h} \!\|^{2} \!+\! \gamma \! \| \! \nabla \!\cdot \!w_{n\!,\!\beta}^{h} \! \|^{2} \!\big)
			 \\&
    		+\! \sum_{n\!=\!1}^{N\!-\!1} \widehat{k}_{n} \int_{\Omega} \big[(C_s\delta)^2 |\nabla w^h_{n,\beta}| \big] |\nabla w_{n,\beta}^{h}|^2 d x
			\!\leq \!\frac{C(\theta) k_{\rm{max}}^{4}}{\nu} \| f_{tt} \|_{L^{2}(0,T;X')}^{2}
		    \!+\! \frac{1}{\nu} \| |f| \|_{2,-1,2,\beta}^{2} 
			 \\&
    		+\frac{1\!+\!\theta}{4} \big( \| w_{1}^h\|^2 \!+\! \frac{C_s^4 \! \delta^2}{\mu^2}\|\nabla w_{1}^h\|^2 \big)
			+\frac{1\!-\!\theta}{4} \big( \| w_{0}^h\|^2 \!+\! \frac{C_s^4 \! \delta^2}{\mu^2}\|\nabla w_{0}^h\|^2 \big). 
        \end{aligned}
		\end{equation}
	\end{theorem}
\begin{proof}
We set $ v^h= w^h_{n,\beta},\ p^h=q_{n,\beta}^h$ in \eqref{eq:weak-DLN-CSM-Alg}. By \cref{trilinear1} and identity \eqref{eq:G-identity} in \cref{lemma:G-stable}, we obtain
			\begin{gather*}
				\Big(\sum_{\ell=0}^{2} \alpha_{\ell} w_{n-1+\ell}^{h}, w_{n,\beta}^h \Big)
				+ \frac{C_{s}^4 \delta^2}{\mu^2} \Big(\sum_{\ell=0}^{2} \alpha_{\ell} \nabla w_{n-1+\ell}^{h},\nabla w_{n,\beta}^{h} \Big) 
				+ \gamma \widehat{k}_{n} \| \nabla \cdot w_{n,\beta}^{h} \|^{2} \\
				+ \widehat{k}_{n} \int_{\Omega} \big( \nu + (C_s\delta)^2 |\nabla  w_{n,\beta}^{h} | \big) 
				|\nabla  w_{n,\beta}^{h}|^2 \ dx 
				= \widehat{k}_{n} ( f_{n,\beta}, w_{n,\beta}^{h} )
				\leq \widehat{k}_{n} \| f_{n,\beta} \|_{-1} \| \nabla w_{n,\beta}^{h} \|.
			\end{gather*}
  The $G$-stability relation \eqref{lemma:G-stable} implies
\begin{equation}\label{eq:stable-eq1}
\begin{aligned}
			&\begin{Vmatrix} w_{n\!+\!1}^h \\  w_{n}^h \end{Vmatrix}_{G(\theta)}^2
			\!-\! \begin{Vmatrix} w_{n}^{h} \\  w_{n\!-\!1}^{h} \end{Vmatrix}_{G(\theta)}^2
			\!+\! \Big\| \sum_{\ell\!=\!0}^{2} \lambda_{\ell}^{(n)} w_{n\!-\!1\!+\!\ell}^{h} \Big\|^2   
			\!+\! \widehat{k}_{n} \int_{\Omega} \big( \frac{\nu}{2} \!+\! (C_s\delta)^2 | \nabla  w_{n,\beta}^{h}| \big)
			|\nabla  w_{n,\beta}^{h} |^2 d x
			 \\
			&+\! \gamma \widehat{k}_{n} \| \nabla \cdot w_{n,\beta}^{h} \|^{2} 
			+\! \frac{C_s^4\delta^2}{\mu^2} \Big(
			\begin{Vmatrix} \nabla w_{n\!+\!1}^{h} \\ \nabla w_{n}^{h} \end{Vmatrix}_{G(\theta)}^2
			\!-\! \begin{Vmatrix} \nabla w_{n}^{h} \\ \nabla w_{n\!-\!1}^{h} \\ \end{Vmatrix}_{G(\theta)}^2
			\!+\! \Big\| \sum_{\ell\!=\!0}^2 \lambda_{\ell}^{(n)} \nabla w_{n\!-\!1\!+\!\ell}^{h} \Big\|^2 \Big) 
			\\
			&\leq  \frac{\widehat{k}_{n}}{2 \nu} \| f_{n,\beta} \|_{-1}^{2}.
\end{aligned}
\end{equation}
		By triangle inequality and \eqref{eq:DLN-consistency} in \cref{lemma:DLN-consistency}, we get,
		\begin{align*}
			\frac{\widehat{k_n}}{2 \nu} \| f_{n,\beta} \|_{-1}^{2}
			\leq& \frac{\widehat{k_n}}{\nu} \| f_{n,\beta} - f(t_{n,\beta}) \|_{-1}^{2} 
			+ \frac{\widehat{k_n}}{\nu} \| f(t_{n,\beta}) \|_{-1}^{2} \\
			\leq& \frac{C(\theta) k_{\rm{max}}^{4}}{\nu} \int_{t_{n-1}}^{t_{n+1}} \| f_{tt} \|_{-1}^{2} dt
			+ \frac{(k_{n} + k_{n-1})}{\nu} \| f(t_{n,\beta}) \|_{-1}^{2}.
		\end{align*}
Summing \eqref{eq:stable-eq1} over $n$ from $1$ to $N-1$, we have desired result \eqref{eq:DLN-Stability}.
\end{proof}
\begin{remark}
		We identify the following quantities from the energy equality in \eqref{eq:DLN-Stability}:
		\begin{enumerate}
			\item Model kinetic energy,
				$$ \mathcal{E}_{N}^{\tt EN} 
				= \frac{1+\theta}{4} \big( \| w_{N}^{h} \|^{2} + \frac{C_s^4 \delta^2}{\mu^{2}} 
				\|\nabla w_{N}^{h} \|^{2} \big)
				+ \frac{1-\theta}{4} \big( \| w_{N-1}^{h} \|^{2} + \frac{C_s^4 \delta^2}{\mu^{2}} 
				\|\nabla w_{N-1}^{h} \|^{2} \big).$$
			\item Energy dissipation due to viscous force, 
			$$ \mathcal{E}_{N}^{\tt VD} = \nu \| \nabla w_{N-1,\beta}^{h} \|^{2}.$$
			\item Eddy viscosity dissipation, 
			$$ \mathcal{E}_{N}^{\tt EVD} = \int_{\Omega} \big[(C_s\delta)^2 |\nabla w_{N-1,\beta}^{h}| \big]
			|\nabla  w_{N-1,\beta}^{h}|^2\ dx.$$   
			\item Numerical dissipation,
			$$ \mathcal{E}_{N}^{\tt ND}=\bigg\|\frac{\sum_{l=0}^{2}\lambda_l^{N-1} w_{N-2+l}^h}{\sqrt{\widehat{k_{N-1}}}}\bigg\|^2+\frac{C_s^4\delta^2}{\mu^2}\bigg\|\frac{\sum_{l=0}^{2}\lambda_l^{N-1} \grad w_{N-2+l}^h}{\sqrt{\widehat{k_{N-1}}}}\bigg\|^2.$$
			$\mathcal{E}_{N}^{\tt ND}$ vanishes if and only if  $\theta\in\{0,1\}.$
			\item The model dissipation originating from the CSM in \eqref{csm0},
			\begin{align*} 
				\mathcal{E}_{N}^{\tt MD} 
				=& \frac{C_s^4 \delta^{2}}{\widehat{k_{N-1}} \mu^{2}} \Big(
				\begin{Vmatrix} \nabla w_{N}^h \\ \nabla w_{N}^h \end{Vmatrix}_{G(\theta)}^2
				-\begin{Vmatrix} \nabla w_{N-1} \\ \nabla w_{N-1} \end{Vmatrix}_{G(\theta)}^2
			    + \Big\| \sum_{\ell=0}^{2} \lambda_{\ell}^{(N-1)} \nabla w_{N-2+\ell}^{h} \Big\|^{2} \Big) \\
			    &+ \int_{\Omega} \big[ (C_s\delta)^2 |\nabla w_{N-1,\beta}^{h}| \big] |\nabla w_{N-1,\beta}^{h} |^2 dx.
			\end{align*}
           Model dissipation in this paper can be positive or negative. When it is positive, it aggregates energy from mean to fluctuations. And when it is negative, energy is being transferred from fluctuations back to mean.
		\end{enumerate}
	\end{remark}
\begin{remark}
The one-leg linearly implicit DLN method by (\ref{v2}) is unconditionally, long-time stable, i.e. for any integer $N>1$,
\begin{equation}\label{eq:DLN-Stability-lagged}
\begin{aligned}
    		&\frac{1\!+\!\theta}{4} \big( \| w_{N}^h \|^2 \!+\! \frac{C_s^4 \! \delta^2}{\mu^2} \| \nabla w_{N}^h \|^2 \big) 
			+ \frac{1\!-\!\theta}{4} \big( \| w_{N-1}^h \|^2 \!+\! \frac{C_s^4 \! \delta^2}{\mu^2} \| \nabla w_{N-1}^h \|^2 \big)  \\& 
    		\!+\! \sum_{n\!=\!1}^{N\!-\!1} \! \Big( \! \big\|\! \sum_{\ell\!=\!0}^2 \lambda_{\ell}^{(n)} w_{n\!-\!1\!+\!\ell}^h \!\big\|^2
			\!+\! \frac{C_s^4 \! \delta^2}{\mu^2} \big\|\! \sum_{\ell\!=\!0}^2 \!\lambda_{\ell}^{(n)} \nabla w_{n\!-\!1\!+\!\ell}^h \!\big\|^2 \! \Big) 
			\!+\!  \sum_{n\!=\!1}^{N\!-\!1} \widehat{k}_{n} \big( \!\frac{\nu}{2} \|\! \nabla \! w_{n\!,\!\beta}^{h} \!\|^{2} \!+\! \gamma \! \| \! \nabla \!\cdot \!w_{n\!,\!\beta}^{h} \! \|^{2} \!\big)
			 \\&
    		+\! \sum_{n\!=\!1}^{N\!-\!1} \widehat{k}_{n} \int_{\Omega} \big[(C_s\delta)^2 |\widetilde{w^h_{n,\beta}}| \big] |\nabla w_{n,\beta}^{h}|^2 d x
			\!\leq \!\frac{C(\theta) k_{\rm{max}}^{4}}{\nu} \| f_{tt} \|_{L^{2}(0,T;X')}^{2}
		    \!+\! \frac{1}{\nu} \| |f| \|_{2,-1,\beta}^{2} 
			 \\&
    		+\frac{1\!+\!\theta}{4} \big( \| w_{1}^h\|^2 \!+\! \frac{C_s^4 \! \delta^2}{\mu^2}\|\nabla w_{1}^h\|^2 \big)
			+\frac{1\!-\!\theta}{4} \big( \| w_{0}^h\|^2 \!+\! \frac{C_s^4 \! \delta^2}{\mu^2}\|\nabla w_{0}^h\|^2 \big). 
\end{aligned}		
\end{equation}
\end{remark}
\subsection{Error Analysis of the DLN Scheme for the CSM}\label{sec:ne}
In this section, we analyze the error between the semi-discrete solution and the fully discrete solution to \eqref{csm0} in \cref{thm:numerial-error} under the following time step condition \footnote[1]{To our best knowledge, time step condition like $\Delta t < \mathcal{O}(\nu^{-3})$ cannot be avoided for fully-implicit schemes in error analysis.}:
\begin{gather} 
		\label{eq:time-condi}
		\frac{C(\theta)}{\nu^{3}} \big( \widehat{k}_{n+1} \| \nabla w_{n+1,\beta} \|^{4}
			+ \widehat{k}_{n} \| \nabla w_{n,\beta} \|^{4} + \widehat{k}_{n-1} \| \nabla w_{n-1,\beta} \|^{4}  \big) < 1, \qquad \forall n.
    \end{gather}
\begin{theorem}\label{thm:numerial-error}
Let $(w(t), q(t))$ be sufficiently smooth, strong solution of the CSM. We assume that the velocity $w \in X$, pressure $q \in Q$, body force $f$ of the CSM in \eqref{csm0} satisfy 
		\begin{gather*}
			w \in \ell^{2,\beta}(0,N;(H^{r+1})^{d}) \cap \ell^{4,\beta}(0,N;(H^{r+1})^{d}) \cap \ell^{3,\beta}(0,N;(H^{r+1})^{d}) 
			\cap \ell^{3,\beta}(0,N;(W^{1,3})^{d}), \\
			w_{t} \in L^{2}\big( 0,T;(H^{r+1})^{d} \big), \qquad w_{ttt} \in L^{2} \big(0,T;(H^{1})^{d} \big),     \\
			w_{tt} \in L^{2}\big(0,T;(H^{r+1})^{d} \big) \cap L^{3}\big(0,T;(W^{1,3})^{d} \big) \cap L^{3}\big(0,T;(H^{r+1})^{d} \big)
			\cap L^{4}\big(0,T; (H^{r+1})^{d} \big), \\
			q \in \ell^{2,\beta}(0,N;H^{s+1}), \qquad
			f_{tt} \in L^{2}(0,T;X')
		\end{gather*}
Under the time step condition in \eqref{eq:time-condi}, the variable time-stepping DLN scheme (with $\theta \in [0,1]$) for the CSM in \eqref{eq:weak-DLN-CSM-Alg} satisfies: for $r,\ s, \in\{0\}\cup \mathbb{N}$ and any integer $N \geq 2$
\begin{equation}\label{eq:error-conclusion}
\begin{aligned}
&\max_{0 \leq n \leq N} \| w_{n}^{h} - w_{n} \| 
			+ C(\theta) \sqrt{\nu} \Big( \sum_{n=1}^{N-1} \widehat{k}_{n} 
			\| \nabla (w_{n,\beta}^{h} - w_{n,\beta}) \|^{2} \Big)^{1/2}  \\&
			\leq \mathcal{O} \big( k_{\max}^{2}, h^{r}, h^{s+1}, \delta h^{\frac{3r}{4}-\frac{d}{8}}, 
			\delta k_{\max}^{3/2} \big).
\end{aligned}
\end{equation}
\end{theorem}
\begin{remark}
	Since $\delta$ has the same dimension as $h$, the spatial convergence rate in \eqref{eq:error-conclusion} is $\min\{r,s+1\}$ as long as the highest polynomial degree for velocity $r \in \{1,2\}$. Thus the DLN algorithm in \eqref{eq:DLN-CSM-Alg} is second order accurate in both time and space if we choose Taylor-Hood $\mathbb{P}2-\mathbb{P}1$ finite element space and set the time step $\Delta t \approx h$. 
\end{remark}
\begin{proof}
We start with the CSM at time $t_{n,\beta}\ (1\leq n\leq N-1)$. For any $v^h\in V^h$, the variational formulation becomes
\begin{gather*}
    			( w_t(t_{n,\beta}), v^h) + \frac{C_s^4 \delta^2 }{\mu^2}(\nabla w_t(t_{n,\beta}), \nabla v^h) 
				+ b^{\ast} \big( w(t_{n,\beta}), w(t_{n,\beta}), v^h \big) - ( q(t_{n,\beta}), \nabla \cdot  v^h)
    +\nu(\nabla w(t_{n,\beta}), \nabla v^h) 
				\\
    			+ \Big((C_s\delta)^2 |\nabla w(t_{n,\beta})|\nabla w(t_{n,\beta}), \nabla v^h\Big)
				= \big( f(t_{n,\beta}), v^h \big), \quad \forall v^{h} \in V^{h}.
			\end{gather*}
Equivalently,
\begin{equation}\label{eq:weak-exact}
\begin{aligned}			
 			&\Big(\frac{\alpha_2 w_{n+1} + \alpha_1 w_{n} + \alpha_0 w_{n-1}}{\widehat{k}_{n}}, v^{h} \Big)
			+ \frac{C_s^4\delta^2}{\mu^2} \Big(\frac{\alpha_2 \nabla w_{n+1} + \alpha_1 \nabla w_{n} + \alpha_0 \nabla w_{n-1}}{\widehat{k}_n}, \nabla v^{h} \Big)  \\&
 			+ b^{\ast}(w_{n,\beta}, w_{n,\beta}, v^h) - \big( q(t_{n,\beta}), \nabla \cdot v^{h} \big)
			+ \nu (\nabla w_{n,\beta}, \nabla v^h) + \gamma (\nabla \cdot w_{n,\beta}, \nabla \cdot v^h) 
			 \\&
			+ \Big( (C_s\delta)^2 |\nabla w_{n,\beta}| \nabla w_{n,\beta}, \nabla v^{h} \Big) 
			= (f_{n,\beta}, v^{h}) + \tau_{n}(v^{h}),
		\end{aligned}
\end{equation}
where the truncation error is
\begin{align*}
 			\tau_{n}(v^{h})
			&= \Big(\frac{\alpha_2 w_{n+1}+\alpha_1 w_{n}+\alpha_0 w_{n-1}}{\widehat{k}_n}-w_t(t_{n,\beta}), v^h \Big) + \nu \big( \nabla \big( w_{n,\beta} - w(t_{n,\beta}) \big), \nabla v^h \big)   
			\\
			& +\frac{C_s^4\delta^2}{\mu^2} \Big(\frac{\alpha_2 \nabla w_{n+1} + \alpha_1 \nabla w_{n} + \alpha_0 \nabla w_{n-1}}{\widehat{k}_{n}} - \nabla w_t(t_{n,\beta}), \nabla v^{h} \Big) \\
			& + b^{\ast} ( w_{n,\beta}, w_{n,\beta}, v^h) - b^{\ast}( w(t_{n,\beta}), w(t_{n,\beta}), v^h) 
			+ \big( f(t_{n,\beta}) - f_{n,\beta}, v^{h} \big)
			\\
			& + \Big( (C_s\delta)^2 \big( |\nabla w_{n,\beta}|\nabla w_{n,\beta}
			-|\nabla w(t_{n,\beta})|\nabla w(t_{n,\beta}) \big), \nabla v^{h} \Big).
		\end{align*}
Let $W_{n}$ be velocity component of Stokes projection of $(w_{n}, 0)$ onto $V^{h} \times Q^{h}$. We set 
		\begin{gather}
			\label{eq:error-decompose}
			e_{n} = w_{n} - w_{n}^{h}, \qquad
			\eta_{n} = w_{n}-W_{n}, \qquad \phi_{n}^{h} = w_{n}^{h} - W_{n}, \\
			e_{n,\beta} = \sum_{\ell=0}^{2} \beta_{\ell}^{(n)} e_{n-1+\ell}, \qquad 
			\eta_{n,\beta} = \sum_{\ell=0}^{2} \beta_{\ell}^{(n)} \eta_{n-1+\ell}, \qquad 
			\phi_{n,\beta}^{h} = \sum_{\ell=0}^{2} \beta_{\ell}^{(n)} \phi_{n-1+\ell}^{h}.    \notag
		\end{gather}
		Then the velocity error $e_{n}$ can be decomposed as $e_{n} = \eta_{n} - \phi_{n}^{h}$.
		We subtract the DLN scheme in \eqref{eq:weak-DLN-CSM-Alg} from \eqref{eq:weak-exact} to get the following,
			\begin{gather*}
				\Big(\frac{\alpha_2 e_{n+1} + \alpha_1 e_{n} + \alpha_0 e_{n-1}}{\widehat{k}_{n}}, v^{h} \Big)
				+\frac{C_s^4\delta^2}{\mu^2} \Big(\frac{\alpha_2 \nabla e_{n+1} + \alpha_1 \nabla e_{n} 
			    +\alpha_0 \nabla e_{n-1}}{\widehat{k_n}}, \nabla v^{h} \Big) \\
				+ b^{\ast} ( w_{n,\beta}, w_{n,\beta}, v^{h}) - b^*( w_{n,\beta}^h, w_{n,\beta}^h, v^{h})
				+ \nu (\nabla e_{n,\beta}, \nabla v^h) + \gamma (\nabla \cdot e_{n,\beta}, \nabla \cdot v^h)
				\\
			    +\Big((C_s\delta)^2 (|\nabla w_{n,\beta}|\nabla w_{n,\beta} - |\nabla w_{n,\beta}^h| 
			    \nabla w_{n,\beta}^h), \nabla v^{h} \Big)
			    = \big( q(t_{n,\beta}),\nabla \cdot v^{h} \big) + \tau_{n}(v^{h}), \quad \forall v^h\in V^h.
		    \end{gather*}
	Notice that,
   \begin{align*}
				&b^{\ast}( w_{n,\beta},  w_{n,\beta}, v^h) - b^{\ast}( w_{n,\beta}^h,  w_{n,\beta}^h, v^h) \\
				=& b^{\ast}( w_{n,\beta},  w_{n,\beta}, v^h) - b^{\ast}( w_{n,\beta}^h,  w_{n,\beta}, v^h)
				  +b^{\ast}( w_{n,\beta}^h,  w_{n,\beta}, v^h) - b^{\ast}( w_{n,\beta}^h,  w_{n,\beta}^h, v^h), \\
				=& b^{\ast}({e_{n,\beta}},  w_{n,\beta}, v^h) + b^{\ast}( w_{n,\beta}^h, {e_{n,\beta}}, v^h),
			\end{align*}
	and		\begin{align*}
				&\int_{\Omega}(|\nabla w_{n,\beta}| \nabla w_{n,\beta} - |\nabla w_{n,\beta}^h| 
				\nabla w_{n,\beta}^h) : \nabla v^h d x \\
				= & \int_{\Omega}(|\nabla w_{n,\beta}| \nabla w_{n,\beta} - |\nabla W_{n,\beta}| 
				\nabla W_{n,\beta} + |\nabla W_{n,\beta}| \nabla W_{n,\beta} - |\nabla w_{n,\beta}^h|
				\nabla w_{n,\beta}^h) : \nabla v^h  d x.
			 \end{align*}
Hence,
\begin{equation}\label{eq:error-eq1}
		\begin{aligned}		
			&\Big( \frac{\alpha_2 \phi_{n+1}^{h} + \alpha_1 \phi_{n}^{h} + \alpha_0 \phi_{n-1}^{h}}{\widehat{k}_{n}}, v^{h} \Big)
			+ \frac{C_s^4\delta^2}{\mu^2} \Big(\frac{\alpha_2 \nabla \phi_{n+1}^{h} + \alpha_1 \nabla \phi_{n}^{h}+ \alpha_0 \nabla \phi_{n-1}^{h}}{\widehat{k}_{n}}, \nabla  v^{h} \Big)  \\
			&- b^{\ast}({e_{n,\beta}},  w_{n,\beta}, v^{h}) - b^{\ast} (w_{n,\beta}^h, {e_{n,\beta}}, v^{h}) 
			+ \nu(\nabla \phi_{n,\beta}^{h}, \nabla v^{h}) 
			+ \gamma (\nabla \cdot \phi_{n,\beta}^{h}, \nabla \cdot v^{h})  \\
			&+ (C_s\delta)^2 \int_{\Omega}(|\nabla w_{n,\beta}^h| \nabla w_{n,\beta}^h 
			- |\nabla W_{n,\beta}| \nabla W_{n,\beta}):(\nabla v^{h}) d x  \\
			=& \Big( \frac{\alpha_2 \eta_{n+1}+\alpha_1 \eta_{n}+\alpha_0 \eta_{n-1}}{\widehat{k}_n}, v^{h} \Big)
			+ \frac{C_s^4\delta^2}{\mu^2} \Big(\frac{\alpha_2 \nabla \eta_{n+1} + \alpha_1 \nabla \eta_{n}+ \alpha_0 \nabla \eta_{n-1}}{\widehat{k}_{n}}, \nabla v^{h} \Big)  \\
			& + (C_s\delta)^2 \int_{\Omega}(|\nabla w_{n,\beta}| \nabla w_{n,\beta} 
			- |\nabla W_{n,\beta}| \nabla W_{n,\beta}) : \nabla v^{h} dx  \\
			& + \nu(\nabla \eta_{n,\beta}, \nabla v^{h})  
			+ \gamma (\nabla \cdot \eta_{n,\beta}, \nabla \cdot v^{h})
			- \big( q(t_{n,\beta}), \nabla \cdot v^{h} \big) - \tau_{n}(v^h).
		\end{aligned}
  \end{equation}
		We set $v^{h} = \phi_{n,\beta}^{h}$ in \eqref{eq:error-eq1} and use \eqref{eq:G-identity} 
		in Lemma \ref{lemma:G-stable},
		\begin{align}
			\label{eq:error-eq2}
			&\begin{Vmatrix} \phi_{n\!+\!1}^h \\  \phi_{n}^h \end{Vmatrix}_{G(\theta)}^2
			\!-\! \begin{Vmatrix} \phi_{n}^{h} \\  \phi_{n\!-\!1}^{h} \end{Vmatrix}_{G(\theta)}^2
			\!+\! \Big\| \sum_{\ell\!=\!0}^{2} \lambda_{\ell}^{(n)} \phi_{n\!-\!1\!+\!\ell}^{h} \Big\|^2 
			+ \nu \widehat{k}_{n} \| \nabla \phi_{n,\beta}^{h} \|^{2} 
			+ \gamma \widehat{k}_{n} \| \nabla \cdot \phi_{n,\beta}^{h} \|^{2} \notag \\
			&+ \frac{C_s^4\delta^2}{\mu^{2}} \Big( \begin{Vmatrix} \nabla \phi_{n\!+\!1}^h \\ \nabla  \phi_{n}^h \end{Vmatrix}_{G(\theta)}^2
			\!-\! \begin{Vmatrix} \nabla \phi_{n}^{h} \\ \nabla  \phi_{n\!-\!1}^{h} \end{Vmatrix}_{G(\theta)}^2
			\!+\! \Big\| \sum_{\ell\!=\!0}^{2} \lambda_{\ell}^{(n)} \nabla \phi_{n\!-\!1\!+\!\ell}^{h} \Big\|^2 \Big) \\
			&+\!  (C_s\delta)^2 \widehat{k}_{n} \int_{\Omega}(|\nabla w_{n,\beta}^h| \nabla w_{n,\beta}^h 
			\!-\! |\nabla W_{n,\beta}| \nabla W_{n,\beta}) \!:\!(\nabla \phi_{n,\beta}^{h}) d x \notag \\
			=& \Big( \sum_{\ell=0}^{2} \alpha_{\ell} \eta_{n-1+\ell}, \phi_{n,\beta}^{h} \Big) 
			+ \frac{C_s^4\delta^2}{\mu^2} \Big( \sum_{\ell=0}^{2} \alpha_{\ell} \nabla \eta_{n-1+\ell}, \nabla \phi_{n,\beta}^{h} \Big)
			+ \nu \widehat{k}_{n} (\nabla \eta_{n,\beta}, \nabla \phi_{n,\beta}^{h})  \notag \\
			&+ \gamma \widehat{k}_{n} (\nabla \cdot \eta_{n,\beta}, \nabla \cdot \phi_{n,\beta}^{h}) 
			+ \widehat{k}_{n} b^{\ast}({e_{n,\beta}},  w_{n,\beta}, \phi_{n,\beta}^{h}) 
			+ \widehat{k}_{n} b^{\ast} (w_{n,\beta}^h, {e_{n,\beta}}, \phi_{n,\beta}^{h}) \notag \\
			&+ \! (\!C_s\delta \!)^2 \! \widehat{k}_{n} \!\! \int_{\Omega}\!(\!|\nabla w_{n,\beta}| \nabla w_{n,\beta} 
			\! - \! |\nabla W_{n,\beta}| \nabla W_{n,\beta} \!) \!:\! \nabla \phi_{n,\beta}^{h} dx 
			\!-\! \widehat{k}_{n} \!\big(\! q(\!t_{n,\beta}\!), \!\nabla \!\cdot \! \phi_{n,\beta}^{h} \!\big)
			\!-\! \widehat{k}_{n} \!\tau_{n}(\! \phi_{n,\beta}^{h} \!). \notag 
		\end{align}
		By strong monotonicity property \eqref{eq:Strong-Monotone} in  Lemma \ref{lemma:Monotone-LLC}, 
		\begin{gather}
			\label{eq:error-term-LHS}
			(C_s\delta)^2 \! \widehat{k}_{n} \!\int_{\Omega}(|\nabla w_{n,\beta}^h|\nabla w_{n,\beta}^h \!-\! |\nabla W_{n,\beta}|\nabla W_{n,\beta})
			\!:\!(\nabla \phi_{n,\beta}^h)  d x
			\!\geq \!C_1(C_s\delta)^2 \widehat{k}_{n} \|\nabla \phi_{n,\beta}^h\|_{0,3}^3.
		\end{gather}
		By Cauchy Schwarz inequality, Poincar$\rm{\acute{e}}$ inequality and Young's inequality, we obtain
		\begin{align}
			\label{eq:error-term1}
			\Big( \sum_{\ell=0}^{2} \alpha_{\ell} \eta_{n-1+\ell}, \phi_{n,\beta}^{h} \Big)
			\leq \frac{1}{\nu \widehat{k}_{n}} \Big\| \sum_{\ell=0}^{2} \alpha_{\ell} \eta_{n-1+\ell} \Big\|^{2}
			+ \frac{\nu \widehat{k}_{n}}{32} \| \nabla \phi_{n,\beta}^{h} \|^{2}
		\end{align}
		We use the approximation of Stokes projection in \eqref{eq:Stoke-Approx} and H$\rm{\ddot{o}}$lder's inequality 
		\begin{align}
			\label{eq:error-term1-term1}
			\Big\| \sum_{\ell=0}^{2} \alpha_{\ell} \eta_{n-1+\ell} \Big\|^{2}
			\leq& C h^{2r+2} \Big\| \sum_{\ell=0}^{2} \alpha_{\ell} w_{n-1+\ell} \Big\|_{r+1}^{2} \\
			\leq& C(\theta)h^{2r+2} \big( \| w_{n+1} - w_{n} \|_{r+1}^{2} + \| w_{n+1} - w_{n-1} \|_{r+1}^{2}   \big) \notag \\
			\leq& C(\theta)h^{2r+2} (k_{n}+k_{n-1}) \int_{t_{n-1}}^{t_{n+1}} \| w_{t} \|_{r+1}^{2} dt. \notag 
		\end{align}
		By \eqref{eq:error-term1-term1}, \eqref{eq:error-term1} becomes
		\begin{gather}
			\label{eq:error-term1-final}
			\Big( \sum_{\ell=0}^{2} \alpha_{\ell} \eta_{n-1+\ell}, \phi_{n,\beta}^{h} \Big)
			\leq C(\theta)h^{2r+2} \int_{t_{n-1}}^{t_{n+1}} \| w_{t} \|_{r+1}^{2} dt 
			+ \frac{\nu \widehat{k}_{n}}{32} \| \nabla \phi_{n,\beta}^{h} \|^{2}. 
		\end{gather}
		Similarly, we have 
		\begin{gather}
			\label{eq:error-term2-final}
			\frac{C_s^4\delta^2}{\mu^2} \Big(\sum_{\ell\!=\!0}^{2} \alpha_\ell \nabla \eta_{n\!-\!1\!+\!\ell},
			\nabla \phi_{n,\beta}^h \Big)
			\!\leq \! C(\theta)h^{2r} \Big( \frac{C_s^4\delta^2}{\mu^2} \Big)^{2}
			\int_{t_{n\!-\!1}}^{t_{n\!+\!1}} \| w_{t} \|_{r\!+\!1}^{2} dt 
			\!+\! \frac{\nu \widehat{k}_{n}}{32} \| \nabla \phi_{n,\beta}^{h} \|^{2}.
		\end{gather}
		By the definition of Stokes projection in \eqref{eq:Stokes-def}, 
		$(\nabla \eta_{n,\beta}, \nabla \phi_{n,\beta}^{h}) = 0$. 
		By Cauchy Schwarz inequality, Poincar$\rm{\acute{e}}$ inequality and Young's inequality,
		\begin{gather}
			\label{eq:error-term4}
			\gamma \widehat{k}_{n} (\nabla \cdot \eta_{n,\beta}, \nabla \cdot \phi_{n,\beta}^{h})
			\leq \gamma d \widehat{k}_{n} \| \nabla \eta_{n,\beta} \| \| \nabla \phi_{n,\beta}^{h} \|
			\leq \frac{C \gamma^{2} \widehat{k}_{n}}{\nu} \| \nabla \eta_{n,\beta} \|^{2} 
			+ \frac{\nu \widehat{k}_{n}}{32} \| \nabla \phi_{n,\beta}^{h} \|^{2}.
		\end{gather}
		By the approximation of Stokes projection in \eqref{eq:Stoke-Approx}, triangle inequality and \eqref{eq:DLN-consistency} in Lemma \ref{lemma:DLN-consistency}
		\begin{align}
			\label{eq:error-term4-term1}
			\| \nabla \eta_{n,\beta} \|^{2} 
			\leq& C h^{2r} \big( \| w_{n,\beta} - w(t_{n,\beta}) \|_{r+1}^{2} + \| w(t_{n,\beta}) \|_{r+1}^{2} \big) \\
			\leq& C h^{2r} \Big( k_{\rm{max}}^{3} \int_{t_{n-1}}^{t_{n+1}} \| w_{tt} \|_{r+1}^{2} dt
			+ \| w(t_{n,\beta}) \|_{r+1}^{2} \Big). \notag 
		\end{align}
		Hence \eqref{eq:error-term4} becomes
		\begin{align}
			\label{eq:error-term4-final}
			&\gamma \widehat{k}_{n} (\nabla \cdot \eta_{n,\beta}, \nabla \cdot \phi_{n,\beta}^{h}) \\
			\leq& \frac{C \gamma^{2} h^{2r}}{\nu} \Big( k_{\rm{max}}^{4} \int_{t_{n-1}}^{t_{n+1}} \| w_{tt} \|_{r+1}^{2} dt + (k_{n} + k_{n-1}) \| w(t_{n,\beta}) \|_{r+1}^{2} \Big)
			+ \frac{\nu \widehat{k}_{n}}{32} \| \nabla \phi_{n,\beta}^{h} \|^{2}. \notag 
		\end{align}
		By \eqref{eq:b-bound} in Lemma \ref{lemma:trilinear-ineq}, Young's inequality and approximation of Stokes projection in \ref{eq:approx-thm}
		\begin{align}
		\label{eq:error-term5}
			&\widehat{k}_{n} b^{\ast}({e_{n,\beta}},  w_{n,\beta}, \phi_{n,\beta}^{h})   \\
			=& \widehat{k}_{n} b^{\ast}(\eta_{n,\beta}, w_{n,\beta}, \phi_{n,\beta}^h)
			- \widehat{k}_{n} b^{\ast}(\phi_{n,\beta}^h, w_{n,\beta}, \phi_{n,\beta}^h) \notag \\
			\leq& C \widehat{k}_{n} \| \nabla \eta_{n,\beta} \| \| \nabla w_{n,\beta} \| \| \nabla \phi_{n,\beta}^h \|
			+ C \widehat{k}_{n} \| \phi_{n,\beta}^h \|^{1/2} \| \nabla w_{n,\beta} \| \| \nabla \phi_{n,\beta}^h \|^{3/2} \notag \\
			\leq& \frac{C \widehat{k}_{n}}{\nu} \| \nabla \eta_{n,\beta} \|^{2} \| \nabla w_{n,\beta} \|^{2}
			+ \frac{C \widehat{k}_{n}}{\nu^{3}} \| \nabla w_{n,\beta} \|^{4} \| \phi_{n,\beta}^h \|^{2}
			+ \frac{\nu \widehat{k}_{n}}{32} \| \nabla \phi_{n,\beta}^h \|^{2} \notag \\ 
			\leq& \frac{C \widehat{k}_{n}h^{2r}}{\nu} \big( \| w_{n,\beta} \|_{r+1}^{4} + \| \nabla w_{n,\beta} \|^{4}  \big)
			+ \frac{C \widehat{k}_{n}}{\nu^{3}} \| \nabla w_{n,\beta} \|^{4} \| \phi_{n,\beta}^h \|^{2}
			+ \frac{\nu \widehat{k}_{n}}{32} \| \nabla \phi_{n,\beta}^h \|^{2} \notag
		\end{align}
		We use triangle inequality, \eqref{eq:DLN-consistency} in Lemma \ref{lemma:DLN-consistency} and H$\rm{\ddot{o}}$lder's inequality 
		\begin{align*}
			\| w_{n,\beta} \|_{r+1}^{4}
			\leq& C \big( \| w_{n,\beta} - w(t_{n,\beta})  \|_{r+1}^{4} + \| w(t_{n,\beta}) \|_{r+1}^{4} \big)  \\
			\leq& C \Big[ \big( C k_{\rm{max}}^{3} \int_{t_{n-1}}^{t_{n+1}} 1 \cdot \| w_{tt} \|_{r+1}^{2} dt \big)^{2} 
			+ \| w(t_{n,\beta}) \|_{r+1}^{4} \Big] \\
			\leq& C(\theta) \Big( k_{\rm{max}}^{7} \int_{t_{n-1}}^{t_{n+1}} \| w_{tt} \|_{r+1}^{4} dt  + \| w(t_{n,\beta}) \|_{r+1}^{4}  \Big), \\
			\| \nabla w_{n,\beta} \|^{4}
			\leq& C(\theta) \Big( k_{\rm{max}}^{7} \int_{t_{n-1}}^{t_{n+1}} \| \nabla w_{tt} \|^{4} dt  + \| \nabla w(t_{n,\beta}) \|^{4}  \Big)
		\end{align*}
		Thus \eqref{eq:error-term5} becomes
		\begin{confidential}
			\color{darkblue}
			\begin{align*}
				\| \nabla w_{n,\beta} \|^{4}
				\leq& C \big( \| \nabla w_{n,\beta} - \nabla w(t_{n,\beta})  \|^{4} + \| \nabla w(t_{n,\beta}) \|^{4} \big)  \\
				\leq& C \Big[ \big( C k_{\rm{max}}^{3} \int_{t_{n-1}}^{t_{n+1}} 1 \cdot \| \nabla w_{tt} \|^{2} dt \big)^{2} 
				                      + \| \nabla w(t_{n,\beta}) \|^{4} \Big] \\
				\leq& C \Big[ C k_{\rm{max}}^{6} \big( \int_{t_{n-1}}^{t_{n+1}} 1^2 dt \big) 
				             \big( \int_{t_{n-1}}^{t_{n+1}} \| \nabla w_{tt} \|^{4} dt \big) + \| \nabla w(t_{n,\beta}) \|^{4} \Big] \\
				\leq& C \Big( k_{\rm{max}}^{7} \int_{t_{n-1}}^{t_{n+1}} \| \nabla w_{tt} \|^{4} dt  + \| \nabla w(t_{n,\beta}) \|^{4}  \Big)
			\end{align*}
			\normalcolor
		\end{confidential}
\begin{equation}		
  \begin{aligned}
			&\widehat{k}_{n} b^{\ast}({e_{n,\beta}},  w_{n,\beta}, \phi_{n,\beta}^{h})  \\
			\leq& \frac{C(\theta) h^{2r}}{\nu} \Big( k_{\rm{max}}^{8} \int_{t_{n-1}}^{t_{n+1}} \| w_{tt} \|_{r+1}^{4} dt
			+ k_{\rm{max}}^{8} \int_{t_{n-1}}^{t_{n+1}} \| \nabla w_{tt} \|^{4} dt   \\
			&\qquad \qquad + (k_{n}+k_{n-1}) \| w(t_{n,\beta}) \|_{r+1}^{4} + (k_{n}+k_{n-1}) \| \nabla w(t_{n,\beta}) \|^{4} \Big)  \\
			&+ \frac{C \widehat{k}_{n}}{\nu^{3}} \| \nabla w_{n,\beta} \|^{4}
			\| \phi_{n,\beta}^h \|^{2} + \frac{\nu \widehat{k}_{n}}{32} \| \nabla \phi_{n,\beta}^h \|^{2}.
\end{aligned}
\end{equation}
		By \eqref{eq:Skew-Symmetry} and approximation of Stokes projection in \eqref{eq:Stoke-Approx}
		\begin{align}
		\label{eq:error-term6}
			\widehat{k}_{n} b^{\ast} (w_{n,\beta}^h, {e_{n,\beta}}, \phi_{n,\beta}^{h})
			\leq& \frac{C \widehat{k}_{n}}{\nu} \| \nabla w_{n,\beta}^{h} \|^{2} \| \nabla \eta_{n,\beta} \|^{2}
			+ \frac{\nu \widehat{k}_{n}}{32} \| \nabla \phi_{n,\beta}^h \|^{2} \notag \\
			\leq&  \frac{C h^{r}\widehat{k}_{n}}{\nu} \| |w|  \|_{\infty,r+1,2}^{2} \| \nabla w_{n,\beta}^{h} \|^{2}
			+ \frac{\nu \widehat{k}_{n}}{32} \| \nabla \phi_{n,\beta}^h \|^{2}.
		\end{align}
		By Local Lipschitz continuity \eqref{eq:LLC} in Lemma \ref{lemma:Monotone-LLC}	
  \begin{align}
		\label{eq:error-term7}
		&(C_s\delta)^2 \! \widehat{k}_{n}\! \int_{\Omega}(|\nabla w_{n,\beta}|\nabla w_{n,\beta} \!-\! |\nabla W_{n,\beta}|\nabla W_{n,\beta})
		\!:\! (\nabla \phi_{n,\beta}^h) d x \\
		\!\leq& \! (C_s\delta)^2\widehat{k}_{n} C_2 \mathcal{R}_{n} 
		\|\nabla{\eta_{n,\beta}}\|_{0,3} \|\nabla \phi_{n,\beta}^h\|_{0,3} \notag \\
		\!\leq& \frac{C (C_s\delta)^2 C_{2}^{3/2}\widehat{k}_{n}}{\sqrt{C_{1}}} \mathcal{R}_{n}^{3/2}  
		\| \nabla \eta_{n,\beta} \|_{0,3}^{3/2}
		+ \! \frac{C_{1}(C_s\delta)^2 \widehat{k}_{n}}{3} \|\nabla \phi_{n,\beta}^h\|_{0,3}^{3},  \notag 
		\end{align}
		where $\mathcal{R}_{n}=\max\{ \|\nabla w_{n,\beta}\|_{0,3}, \|\nabla W_{n,\beta}\|_{0,3} \}$.
		By triangle inequality, 
		\begin{align*}
		\mathcal{R}_{n}  \leq \max\big\{ \|\nabla w_{n,\beta}\|_{0,3}, \big\| \nabla \big(W_{n,\beta} - w_{n,\beta} \big) \big\|_{0,3} 
		+ \|\nabla w_{n,\beta}\|_{0,3}  \big\} 
		= \| \nabla \eta_{n,\beta} \|_{0,3} + \|\nabla w_{n,\beta}\|_{0,3}. 
		\end{align*}
		We use \eqref{eq:Lp-L2} in Theorem \ref{thm:Lp-L2}, triangle inequality and approximation theorem of interpolation in \eqref{eq:approx-thm}
		\begin{gather*}
			\| \nabla \eta_{n,\beta} \|_{0,3}
			\leq C h^{-d/6} \| \nabla \eta_{n,\beta} \| \leq C h^{r-d/6} \| w_{n,\beta} \|_{r+1}
		\end{gather*}
		By the fact: for any $a,b,c \in \mathbb{R}$ with $c>1$, 
		\begin{gather}
		\label{eq:fact1}
			(|a| + |b|)^{c} \leq 2^{c-1}(|a|^{c} + |b|^{c}),
		\end{gather}
		we have
		\begin{align}
		\label{eq:error-term7-term1}
			\mathcal{R}_{n}^{3/2} \| \nabla \eta_{n,\beta} \|_{0,3}^{3/2}
			\leq& C \big( \| \nabla \eta_{n,\beta} \|_{0,3}^{3} 
			+ \| \nabla \eta_{n,\beta} \|_{0,3}^{3/2} \| \nabla w_{n,\beta} \|_{0,3}^{3/2} \big) \\
			\leq& C h^{3r-d/2} \| w_{n,\beta} \|_{r+1}^{3} 
			+ C h^{3/2r-d/4} \| w_{n,\beta} \|_{r+1}^{3/2} \| \nabla w_{n,\beta} \|_{0,3}^{3/2} \notag \\
			\leq&  C h^{3r-d/2} \| w_{n,\beta} \|_{r+1}^{3} 
			+ C h^{3/2r-d/4} \big( \| w_{n,\beta} \|_{r+1}^{3} + \| \nabla w_{n,\beta} \|_{0,3}^{3} \big). \notag 
		\end{align}
		By triangle inequality, the fact in \eqref{eq:fact1}, \eqref{eq:DLN-consistency} in \Cref{lemma:DLN-consistency} and H$\rm{\ddot{o}}$lder's inequality,
		\begin{align}
		\label{eq:error-term7-term2}
			\| w_{n,\beta} \|_{r+1}^{3} 
			\leq& C \| w_{n,\beta} - w(t_{n,\beta}) \|_{r+1}^{3} + C \| w(t_{n,\beta}) \|_{r+1}^{3}  \\
			\leq& C \Big( k_{\rm{max}}^{3} \int_{t_{n-1}}^{t_{n+1}} \| w_{tt} \|_{r+1}^{2} dt \Big)^{3/2}
			+ C \| w(t_{n,\beta}) \|_{r+1}^{3} \notag \\
			\leq& C k_{\rm{max}}^{5} \int_{t_{n-1}}^{t_{n+1}} \| w_{tt} \|_{r+1}^{3} dt  
			+ C \| w(t_{n,\beta}) \|_{r+1}^{3} \notag.
		\end{align}
		By \eqref{eq:error-term7-term1} and \eqref{eq:error-term7-term2}, \eqref{eq:error-term7} becomes
		\begin{align}
		\label{eq:error-term7-final}
			&(C_s\delta)^2 \! \widehat{k}_{n}\! \int_{\Omega}(|\nabla w_{n,\beta}|\nabla w_{n,\beta} \!-\! |\nabla W_{n,\beta}|\nabla W_{n,\beta})
			\!:\! (\nabla \phi_{n,\beta}^h) d x \\
			\leq& \frac{C (C_s\delta)^2 C_{2}^{3/2}}{\sqrt{C_{1}}}
			\Big[ (1+h^{\frac{3r}{2}-\frac{d}{4}} ) h^{\frac{3r}{2}-\frac{d}{4}} \big(k_{\rm{max}}^{6}  
			\int_{t_{n-1}}^{t_{n+1}} \| w_{tt} \|_{r+1}^{3} dt + (k_{n} + k_{n-1}) \| w(t_{n,\beta}) \|_{r+1}^{3}  \big) \notag \\
			& \qquad \qquad \qquad +h^{\frac{3r}{2}-\frac{d}{4}} \big(k_{\rm{max}}^{6}  
			\int_{t_{n-1}}^{t_{n+1}} \| \nabla w_{tt} \|_{0,3}^{3} dt + (k_{n} + k_{n-1}) \| \nabla w(t_{n,\beta}) \|_{0,3}^{3}  \big) \Big].
			\notag 
		\end{align}
		We choose $p^{h}$ to be $L^2$-projection of $q(\!t_{n,\beta}\!)$ onto $Q^{h}$, then 
		\begin{align}
		\label{eq:error-term8}
			\widehat{k}_{n} \big( q(t_{n,\beta}), \nabla \!\cdot\!  \phi_{n,\beta}^{h} \big)
			\!=\!	\widehat{k}_{n} \big( q(t_{n,\beta}) \!-\! p^{h}, \nabla \!\cdot\!  \phi_{n,\beta}^{h} \big) 
			\!\leq\! \sqrt{d} \widehat{k}_{n} \| q(t_{n,\beta}) \!-\! p^{h} \| \| \nabla \phi_{n,\beta}^{h} \|.
		\end{align}
		By Young's inequality and approximation of pressure in \eqref{eq:approx-thm}, \eqref{eq:error-term8} becomes 
		\begin{align}
		\label{eq:error-term8-final}
			\widehat{k}_{n} \big( q(t_{n,\beta}), \nabla \!\cdot\!  \phi_{n,\beta}^{h} \big)
			\leq \frac{C h^{2s+2}}{\nu} (k_{n}+k_{n-1}) \| q(t_{n,\beta}) \|_{s+1}^{2} 
			+ \frac{\nu \widehat{k}_{n}}{32} \| \nabla \phi_{n,\beta}^{h} \|^{2}
		\end{align}
		Now we deal $\widehat{k}_{n}\tau_{n}(\phi_{n,\beta}^h)$: by Cauchy Schwarz inequality, 
		Poincar$\rm{\acute{e}}$ inequality and \eqref{eq:DLN-consistency} in Lemma \ref{lemma:DLN-consistency}, the first three terms become 
		\begin{align}
			\label{eq:tau-term1}
			& \widehat{k}_{n} \Big(\frac{\alpha_2 w_{n+1}+\alpha_1 w_{n}+\alpha_0 w_{n-1}}{\widehat{k}_{n}}
			-w_t(t_{n,\beta}), \phi_{n,\beta}^h \Big) \notag \\
			\leq& C \widehat{k}_{n} \Big\| \frac{\alpha_2 w_{n+1}+\alpha_1 w_{n}+\alpha_0 w_{n-1}}{\widehat{k}_{n}} - w_t(t_{n,\beta}) \Big\| \| \nabla \phi_{n,\beta}^h \|  \notag \\
			\leq& \frac{C\widehat{k}_{n}}{\nu} \Big\| \frac{\alpha_2 w_{n+1}+\alpha_1 w_{n}+\alpha_0 w_{n-1}}
			{\widehat{k}_{n}}-w_t(t_{n,\beta}) \Big\|^{2} 
			+ \frac{\nu \widehat{k}_{n}}{32} \| \nabla \phi_{n,\beta}^h \|^{2} \notag \\
			\leq& \frac{C(\theta) k_{\rm{max}}^{4}}{\nu} \int_{t_{n-1}}^{t_{n+1}} \| w_{ttt} \|^{2} dt
			+ \frac{\nu \widehat{k}_{n}}{32} \| \nabla \phi_{n,\beta}^h \|^{2}, 
		\end{align}
		\begin{align}
			\label{eq:tau-term2}
			\nu \widehat{k_n} \big( \nabla (w_{n,\beta} - w(t_{n,\beta})),\nabla  \phi_{n,\beta}^h \big)
			\leq& \frac{C \widehat{k}_{n}}{\nu} \big\| \nabla w_{n,\beta} - \nabla w(t_{n,\beta}) \big\|
			+ \frac{\nu \widehat{k}_{n}}{32} \| \nabla \phi_{n,\beta}^h \|^{2}  \notag \\
			\leq& \frac{C k_{\rm{max}}^{4}}{\nu} \int_{t_{n-1}}^{t_{n+1}} \| \nabla w_{tt} \|^{2} dt
			+ \frac{\nu \widehat{k}_{n}}{32} \| \nabla \phi_{n,\beta}^h \|^{2}, 
		\end{align}
	and	\begin{align}
			\label{eq:tau-term3}
			& \frac{C_s^4 \delta^2 \widehat{k}_{n}}{\mu^2} \Big(\frac{\alpha_2\nabla w_{n+1} 
			+ \alpha_1 \nabla w_{n} + \alpha_0 \nabla w_{n-1}}{\widehat{k}_{n}} 
			- \nabla w_t(t_{n,\beta}), \nabla  \phi_{n,\beta}^h \Big) \notag \\
			\leq& \frac{C \widehat{k_n}}{\nu} \Big( \frac{C_s^4\delta^2}{\mu^2} \Big)^{2}
			\Big\| \nabla \Big( \frac{\alpha_2 w_{n+1} + \alpha_1 w_{n} + \alpha_0 w_{n-1}}{\widehat{k}_{n}} 
			- w_t(t_{n,\beta}) \Big) \Big\|^{2}
			+ \frac{\nu \widehat{k}_{n}}{32} \| \nabla \phi_{n,\beta}^h \|^{2} \notag \\
			\leq& \frac{C(\theta)k_{\rm{max}}^{4}}{\nu} \Big( \frac{C_s^4\delta^2}{\mu^2} \Big)^{2} 
			\int_{t_{n-1}}^{t_{n+1}} \|\nabla w_{ttt} \|^{2} dt
			+ \frac{\nu \widehat{k}_{n}}{32} \| \nabla \phi_{n,\beta}^h \|^{2}.
		\end{align}
		By \eqref{eq:b-bound} in Lemma \ref{lemma:trilinear-ineq} and triangle inequality, two non-linear terms become 
		\begin{align*}
			& \widehat{k}_{n} b^{\ast} \big( w_{n,\beta}, w_{n,\beta}, \phi_{n,\beta}^h \big) 
			- \widehat{k}_{n} b^{\ast} \big( w(t_{n,\beta}), w(t_{n,\beta}), \phi_{n,\beta}^h \big) \\
			=& \widehat{k}_{n} b^{\ast} \big( w_{n,\beta} - w(t_{n,\beta}), w_{n,\beta},\phi_{n,\beta}^h \big)
			+\widehat{k}_{n} b^{\ast} \big( w(t_{n,\beta}), w_{n,\beta} - w(t_{n,\beta}), \phi_{n,\beta}^h \big) \\
			\leq& \frac{C\widehat{k}_{n}}{\nu} \big\| \nabla (w_{n,\beta} - w(t_{n,\beta}) ) \big\|^{2} 
			\big( \| \nabla w_{n,\beta} \|^{2} + \| \nabla w(t_{n,\beta}) \|^{2} \big) 
			+ \frac{\nu \widehat{k}_{n}}{32} \| \nabla \phi_{n,\beta}^h \|^{2} \\
			\leq& \frac{C\widehat{k}_{n}}{\nu} \big\| \nabla (w_{n,\beta} - w(t_{n,\beta}) ) \big\|^{2} 
			\big( \big\| \nabla (w_{n,\beta} - w(t_{n,\beta}) ) \big\|^{2} + 2\| \nabla w(t_{n,\beta}) \|^{2} \big) + \frac{\nu \widehat{k}_{n}}{32} \| \nabla \phi_{n,\beta}^h \|^{2}.
		\end{align*}
		By \eqref{eq:DLN-consistency} in Lemma \ref{lemma:DLN-consistency} and H$\rm{\ddot{o}}$lder's inequality,
    	\begin{align*}
    		\big\| \nabla (w_{n,\beta} - w(t_{n,\beta}) ) \big\|^{4}
    		\leq & C(\theta) k_{\rm_{max}}^{7} \int_{t_{n-1}}^{t_{n+1}} \| \nabla w_{tt} \|^{4} dt.
    	\end{align*}
    	\begin{align*}
    		\big\| \nabla (w_{n,\beta} - w(t_{n,\beta}) ) \big\|^{2} \| \nabla w(t_{n,\beta}) \|^{2}
    		\leq & C(\theta) k_{\rm_{max}}^{3} 
    		\int_{t_{n-1}}^{t_{n+1}} \| \nabla w(t_{n,\beta}) \|^{2} \| \nabla w_{tt} \|^{2} dt 
    		\notag \\
    	 	\leq & C(\theta) k_{\rm_{max}}^{3}
    		\int_{t_{n-1}}^{t_{n+1}} \big( \| \nabla w(t_{n,\beta}) \|^{4} + \| \nabla w_{tt} \|^{4} \big) dt 
    		\notag \\
    		\leq & C(\theta) k_{\rm_{max}}^{3} \int_{t_{n-1}}^{t_{n+1}} \| \nabla w_{tt} \|^{4} dt  
    	    + C(\theta) k_{\rm_{max}}^{4} \| \nabla w(t_{n,\beta}) \|^{4}.
    	\end{align*}
		\begin{align}
			\label{eq:tau-term45}
			& \widehat{k}_{n} b^{\ast} \big( w_{n,\beta}, w_{n,\beta}, \phi_{n,\beta}^h \big) 
			- \widehat{k}_{n} b^{\ast} \big( w(t_{n,\beta}), w(t_{n,\beta}), \phi_{n,\beta}^h \big) \\
			\leq& \frac{C(\theta) k_{\rm_{max}}^{4}}{\nu} \Big[ \big( 1 + k_{\rm_{max}}^{4} \big) 
			\int_{t_{n-1}}^{t_{n+1}} \| \nabla w_{tt} \|^{4} dt 
			+ (k_{n} + k_{n-1}) \| \nabla w(t_{n,\beta}) \|^{4} \Big]
			+ \frac{\nu \widehat{k}_{n}}{32} \| \nabla \phi_{n,\beta}^h \|^{2}. \notag 
		\end{align}
		\begin{align} 
			\label{eq:tau-term6}
			\widehat{k}_{n} \big( f(t_{n,\beta}) - f_{n,\beta}, \phi_{n,\beta}^h \big)
			\leq& \frac{C \widehat{k_n}}{\nu} \| f(t_{n,\beta}) - f_{n,\beta} \|_{-1}^{2}
			+\frac{\nu \widehat{k}_{n}}{32} \| \nabla \phi_{n,\beta}^h \|^{2} 
			\notag \\
			\leq& \frac{C \widehat{k}_{n}}{\nu} k_{\rm_{max}}^{3} \int_{t_{n-1}}^{t_{n+1}} \| f_{tt} \|_{-1}^{2} dt + \frac{\nu \widehat{k}_{n}}{32} \| \nabla \phi_{n,\beta}^h \|^{2}.
		\end{align}
		By \eqref{eq:LLC} in Lemma \ref{lemma:Monotone-LLC} and Young's inequality,
		\begin{align}
			\label{eq:tau-term7}
			& \widehat{k}_{n} \Big((C_s\delta)^2(|\nabla w_{n,\beta}|\nabla w_{n,\beta}-|\nabla w(t_{n,\beta})|\nabla w(t_{n,\beta})), 
			\nabla \phi_{n,\beta}^h \Big) \\
			\leq& \widehat{k}_{n} (C_s\delta)^2C_2 \mathcal{S}_{n} \big\| \nabla (w_{n,\beta} - w(t_{n,\beta}) )\big\|_{0,3} 
			\|\nabla \phi_{n,\beta}^h\|_{0,3} \notag \notag \\
			\leq& \frac{C (C_s\delta)^2 C_{2}^{3/2}\widehat{k}_{n}}{\sqrt{C_{1}}} \mathcal{S}_{n}^{3/2} 
			\big\| \nabla (w_{n,\beta} - w(t_{n,\beta}) ) \big\|_{0,3}^{3/2}
			+ \frac{C_{1}(C_s\delta)^2 \widehat{k}_{n}}{4} \|\nabla \phi_{n,\beta}^h\|_{0,3}^{3}, \notag  
		\end{align}
		where $\mathcal{S}_{n} = \max \big\{ \|\nabla w_{n,\beta}\|_{0,3}, \|\nabla w(t_{n,\beta})\|_{0,3} \big\}$.
		\begin{confidential}
			\color{darkblue}
			\begin{align*}
				\Big( k_{\rm{max}}^{3} \int_{t_{n-1}}^{t_{n+1}} \| \nabla w_{tt} \|_{0,3}^{2} dt \Big)^{3/2}
				=& \Big( k_{\rm{max}}^{3} \int_{t_{n-1}}^{t_{n+1}} 1 \cdot \| \nabla w_{tt} \|_{0,3}^{2} dt \Big)^{3/2} \\
				\leq& \Big[ k_{\rm{max}}^{3} \big(\int_{t_{n-1}}^{t_{n+1}} 1^{3} dt \big)^{1/3}
				\big(\int_{t_{n-1}}^{t_{n+1}} (\| \nabla w_{tt} \|_{L^{3}}^{2})^{3/2} dt \big)^{2/3} \Big]^{3/2} \\
				=& \Big[ C k_{\rm{max}}^{10/3} 
				\big(\int_{t_{n-1}}^{t_{n+1}} \| \nabla w_{tt} \|_{0,3}^{3} dt \big)^{2/3} \Big]^{3/2} \\
				=& C k_{\rm{max}}^{5} \int_{t_{n-1}}^{t_{n+1}} \| \nabla w_{tt} \|_{0,3}^{3} dt
			\end{align*}
			\normalcolor
		\end{confidential}
		By \eqref{eq:DLN-consistency} in Lemma \ref{lemma:DLN-consistency} and Young's inequality,
		\begin{align}
			\label{eq:tau-term7-term1}
			&\mathcal{S}_{n}^{3/2} \big\| \nabla \big( w(t_{n,\beta}) - w_{n,\beta} \big) \big\|_{0,3}^{3/2} \\
			\leq& C \Big( \big\| \nabla \big( w(t_{n,\beta}) - w_{n,\beta} \big) \big\|_{0,3}^{3}
			+ \big\| \nabla \big( w(t_{n,\beta}) - w_{n,\beta} \big) \big\|_{0,3}^{3/2} 
			\| \nabla w_{n,\beta} \|_{0,3}^{3/2} \Big) \notag \\ 
			\leq& C(\theta) \Big( k_{\rm{max}}^{3} \int_{t_{n-1}}^{t_{n+1}} \| \nabla w_{tt} \|_{0,3}^{2} dt \Big)^{3/2} 
			+ C(\theta) \Big( k_{\rm{max}}^{3} \int_{t_{n-1}}^{t_{n+1}} \| \nabla w_{tt} \|_{0,3}^{2} dt \Big)^{3/4} \| \nabla w_{n,\beta} \|_{0,3}^{3/2} \notag \\
			\leq& C(\theta) k_{\rm{max}}^{9/2} \Big( \int_{t_{n-1}}^{t_{n+1}} \| \nabla w_{tt} \|_{0,3}^{2} dt \Big)^{3/2} + C(\theta) k_{\rm{max}}^{3/2} \Big( \int_{t_{n-1}}^{t_{n+1}} \| \nabla w_{tt} \|_{0,3}^{2} dt \Big)^{3/2} + C(\theta) k_{\rm{max}}^{3} \| \nabla w_{n,\beta} \|_{0,3}^{3} \notag \\
			\leq& C(\theta) k_{\rm{max}}^{9/2} \Big( \int_{t_{n-1}}^{t_{n+1}} \| \nabla w_{tt} \|_{0,3}^{2} dt \Big)^{3/2} + C(\theta) k_{\rm{max}}^{3/2} \Big( \int_{t_{n-1}}^{t_{n+1}} \| \nabla w_{tt} \|_{0,3}^{2} dt \Big)^{3/2} \notag \\
			& + C(\theta) k_{\rm{max}}^{3} \Big( \big\| \nabla \big( w(t_{n,\beta}) - w_{n,\beta} \big) \big\|_{0,3}^{3} + \| \nabla w(t_{n,\beta}) \|_{0,3}^{3} \Big). \notag \\
			\leq& C(\theta) \big( k_{\rm{max}}^{15/2} + k_{\rm{max}}^{9/2} + k_{\rm{max}}^{3/2} \big)
			\Big( \int_{t_{n-1}}^{t_{n+1}} \| \nabla w_{tt} \|_{0,3}^{2} dt \Big)^{3/2} 
			+ C(\theta) k_{\rm{max}}^{3} \| \nabla w(t_{n,\beta}) \|_{0,3}^{3}. \notag 
		\end{align}
		By H$\rm{\ddot{o}}$lder's inequality, 
		\begin{align}
			\label{eq:tau-term7-term2}
			\Big( \int_{t_{n-1}}^{t_{n+1}} \| \nabla w_{tt} \|_{0,3}^{2} dt \Big)^{3/2}
			\leq  C k_{\rm{max}}^{1/2} \int_{t_{n-1}}^{t_{n+1}} \| \nabla w_{tt} \|_{0,3}^{3} dt.
		\end{align}
		By \eqref{eq:tau-term7-term1} and \eqref{eq:tau-term7-term2}, \eqref{eq:tau-term7} becomes 
		\begin{align}
			\label{eq:tau-term7-final}
			& \widehat{k}_{n} \Big((C_s\delta)^2(|\nabla w_{n,\beta}|\nabla w_{n,\beta}-|\nabla w(t_{n,\beta})|\nabla w(t_{n,\beta})), 
			\nabla \phi_{n,\beta}^h \Big) \\
			\leq& \frac{C(\theta) k_{\rm{max}}^{3} (C_s\delta)^2 C_{2}^{3/2}}{\sqrt{C_{1}}} \Big[ \big( k_{\rm{max}}^{6} + k_{\rm{max}}^{3} + 1 \big) \int_{t_{n-1}}^{t_{n+1}} \| \nabla w_{tt} \|_{0,3}^{3} dt
			+ (k_{n} + k_{n-1}) \| \nabla w(t_{n,\beta}) \|_{0,3}^{3} \Big]. \notag 
		\end{align}
		We combine \eqref{eq:error-term-LHS}, \eqref{eq:error-term1-final}, \eqref{eq:error-term2-final}, \eqref{eq:error-term4-final}, \eqref{eq:error-term5}, \eqref{eq:error-term6}, \eqref{eq:error-term7-final}, \eqref{eq:error-term8-final}-\eqref{eq:tau-term6}, \eqref{eq:tau-term7-final} in \eqref{eq:error-eq2} and then sum \eqref{eq:error-eq2} over $n$ from $1$ to $N-1$ to obtain
\begin{equation}\label{eq:error-eq-final}
\begin{aligned}
			&\begin{Vmatrix} \phi_{N}^h \\  \phi_{N-1}^h \end{Vmatrix}_{G(\theta)}^2
			+ \frac{C_s^4\delta^2}{\mu^{2}} \begin{Vmatrix} \nabla \phi_{N}^h \\ \nabla  \phi_{N-1}^h \end{Vmatrix}_{G(\theta)}^2
			+ \sum_{n=1}^{N-1} \Big( \Big\| \sum_{\ell\!=\!0}^{2} \lambda_{\ell}^{(n)} \phi_{n\!-\!1\!+\!\ell}^{h} \Big\|^2 
			+ \Big\| \sum_{\ell\!=\!0}^{2} \lambda_{\ell}^{(n)} \nabla \phi_{n\!-\!1\!+\!\ell}^{h} \Big\|^2 \Big)  \\
			&+ \sum_{n=1}^{N-1} \frac{\nu}{2} \widehat{k}_{n} \| \nabla \phi_{n,\beta}^{h} \|^{2} 
			+ \sum_{n=1}^{N-1} \gamma \widehat{k}_{n} \| \nabla \cdot \phi_{n,\beta}^{h} \|^{2}
			+ \sum_{n=1}^{N-1} C_1(C_s\delta)^2 \widehat{k}_{n} \|\nabla \phi_{n,\beta}^h\|_{0,3}^3 \\
			\leq& \sum_{n=1}^{N-1} \frac{C(\theta) \widehat{k}_{n} \| \nabla w_{n,\beta} \|^{4}}{\nu^{3}} 
			\big( \| \phi_{n+1}^h \|^{2} + \| \phi_{n}^h \|^{2} + \| \phi_{n-1}^h \|^{2} \big)
			+ \frac{C h^{r}}{\nu^2} \| |w|  \|_{\infty,r+1,2}^{2} \Big( \sum_{n=1}^{N-1} \nu \widehat{k}_{n} \| \nabla w_{n,\beta}^{h} \|^{2} \Big) \\
			&+ F(\theta,k_{\rm{max}},h,\delta) 
			+ \begin{Vmatrix} \phi_{1}^h \\  \phi_{0}^h \end{Vmatrix}_{G(\theta)}^2
			+ \frac{C_s^4\delta^2}{\mu^{2}} \begin{Vmatrix} \nabla \phi_{1}^h \\ \nabla  \phi_{0}^h \end{Vmatrix}_{G(\theta)}^2 , 
\end{aligned}
\end{equation}
		where 
		\begin{align*}
			F(\theta,k_{\rm{max}},h,\delta)
			&= C(\theta) h^{2r+2} \| w_{t} \|_{L^{2}(0,T;H^{r+1})}^{2} 
			+ C(\theta) \Big( \frac{C_s^4\delta^2}{\mu^2} \Big)^{2} h^{2r} \| w_{t} \|_{L^{2}(0,T;H^{r+1})}^{2} \\
			& + \frac{C \gamma^{2} h^{2r}}{\nu} \Big( k_{\rm{max}}^{4} \|w_{tt} \|_{L^{2}(0,T;H^{r+1})}^{2} 
			+ \| |w| \|_{2,r+1,2,\beta}^{2} \Big) \notag \\
			& + \frac{C(\theta) h^{2r}}{\nu} \Big( k_{\rm{max}}^{8} \| w_{tt} \|_{L^{4}(0,T;H^{r+1})}^{4} 
			+ k_{\rm{max}}^{8} \| \nabla w_{tt} \|_{L^{4}(0,T;L^{2})}^{4}  
			+ \| |w| \|_{4,r+1,2,\beta}^{4} + \| |\nabla w| \|_{4,0,2,\beta}^{4} \Big) \\
			&+ \frac{C (C_s\delta)^2 C_{2}^{3/2}}{\sqrt{C_{1}}}
			\Big[ (1+h^{\frac{3r}{2}-\frac{d}{4}} ) h^{\frac{3r}{2}-\frac{d}{4}} \big(k_{\rm{max}}^{6}  \| w_{tt} \|_{L^{3}(0,T;H^{r+1})}^{3}
			+ \| |w| \|_{3,r+1,2,\beta}^{3}  \big) \notag \\
			& \qquad \qquad \qquad +h^{\frac{3r}{2}-\frac{d}{4}} \big(k_{\rm{max}}^{6}  \| \nabla w_{tt} \|_{L^{3}(0,T;L^{3})}^{3}
			+ \| |\nabla w| \|_{3,0,3,\beta}^{3} \big) \Big]
			 \notag \\
			&+ \frac{C h^{2s+2}}{\nu} \| |q|  \|_{2,s+1,2,\beta}^{2}
			+ \frac{C(\theta) k_{\rm{max}}^{4}}{\nu} \| w_{ttt} \|_{L^{2}(0,T;L^{2})}^{2}
			+ \frac{C k_{\rm{max}}^{4}}{\nu} \| \nabla w_{tt} \|_{L^{2}(0,T;L^{2})}^{2} \\
			&+ \frac{C(\theta)k_{\rm{max}}^{4}}{\nu} \Big( \frac{C_s^4\delta^2}{\mu^2} \Big)^{2} \| \nabla w_{ttt} \|_{L^{2}(0,T;L^{2})}^{2}
			\notag \\
			&+ \frac{C(\theta) k_{\rm_{max}}^{4}}{\nu} \Big[ \big( 1 + k_{\rm_{max}}^{4} \big) \| \nabla w_{tt} \|_{L^{4}(0,T;L^{2})}^{4}
			+ \| |\nabla w| \|_{4,0,2,\beta}^{4} \Big] 
			+ \frac{C k_{\rm_{max}}^{4}}{\nu} \| f_{tt} \|_{L^{2}(0,T;X')}^{2} \notag \\
			&+ \frac{C(\theta) k_{\rm{max}}^{3} (C_s\delta)^2 C_{2}^{3/2}}{\sqrt{C_{1}}} \Big[ \big( k_{\rm{max}}^{6} + k_{\rm{max}}^{3} + 1 \big) \| \nabla w_{tt} \|_{L^{3}(0,T;L^{3})}^{3} + \| |\nabla w| \|_{3,0,3,\beta}^{3} \Big].
		\end{align*}
		By \eqref{eq:DLN-Stability} in Theorem \ref{thm:Stab-Uncond}, 
		\begin{gather*}
			\sum_{n=1}^{N-1} \nu \widehat{k}_{n} \| \nabla w_{n,\beta}^{h} \|^{2} < C(\theta).
		\end{gather*}
		We set 
		\begin{gather*}
			D_{n} = \frac{C(\theta) \widehat{k}_{n} \| \nabla w_{n,\beta} \|^{4}}{k_{\rm{max}} \nu^{3}}, 
			\qquad 1 \leq n \leq N-1,
		\end{gather*}
		and 
		\begin{gather}
			\label{eq:dn-def}
			d_{n}=
			\begin{cases} 
      			D_{1} & n = 0 \\
      			D_{1} + D_{2} & n = 1 \\
      			D_{n-1} + D_{n} + D_{n+1} & 2 \leq n \leq N-2  \\
				D_{N-2} + D_{N-1} & n = N-1 \\
				D_{N-1} & n = N
   			\end{cases}.
		\end{gather}
		By the time step restriction in \eqref{eq:time-condi}, we have $k_{\rm{max}} d_{n} < 1$ for all $n$.
		Then we use the definition of $G(\theta)$-norm in \eqref{eq:G-norm-def} and apply 
		Gr$\rm{\ddot{o}}$nwall's inequality in Lemma \ref{lemma:Gronwall} to \eqref{eq:error-eq-final} (with $d_{n}$ defined in \eqref{eq:dn-def} 
		and $\Delta t = k_{\rm{max}}$)
		\begin{align}
			&\| \phi_{N}^{h} \|^{2} 
			+ C(\theta) \sum_{n=1}^{N-1} \frac{\nu}{2} \widehat{k}_{n} \| \nabla \phi_{n,\beta}^{h} \|^{2}  \\
			\leq& \exp \Big( \sum_{n=1}^{N-1} \frac{k_{\rm{max}} d_{n}}{1 - k_{\rm{max}} d_{n}} \Big)
			\Big[ \frac{C(\theta) h^{r}}{\nu^2} \| |w|  \|_{\infty,r+1,2}^{2} 
			+ F(\theta,k_{\rm{max}},h,\delta) \notag \\
			&\qquad \qquad \qquad \qquad \qquad 
			+ C(\theta) \big( \| \phi_{1}^{h} \|^{2} + \| \phi_{0}^{h} \| \big)
			+ \frac{C(\theta) C_s^4\delta^2}{\mu^{2}} \big( \| \nabla \phi_{1}^{h} \|^{2} 
			+ \| \nabla \phi_{0}^{h} \| \big) \Big]. \notag 
		\end{align}
		By triangle inequality and approximation of Stokes projection in \eqref{eq:approx-thm}, we have 
		\eqref{eq:error-conclusion}.
	\end{proof}
	\begin{remark}
		The Semi-implicit DLN algorithm has been applied to the Navier Stokes equation \cite{Pei2023_arXiv} and outperforms the corresponding fully implicit algorithm in two aspects: removing the time step restriction like \eqref{eq:time-condi} as well as avoiding the non-linear solver at each time step. For error analysis of the semi-implicit DLN algorithm for CSM \eqref{v2},
		the SM \eqref{eq:Strong-Monotone} and LLC \eqref{eq:LLC} conclusions should be adjusted and are left as open problem. To do so, one can follow the work in \cite{ingram2013new,Ing13_IJNAM} where a new linear extrapolation of the convecting velocity for CNLE is proposed that ensures energetic stability without a time step restriction.
	\end{remark}

\section{Numerical Tests}\label{sec:nt}
In this section, we perform two numerical tests. In the first test, we show the numerical error and the rate of convergence of the DLN scheme. In the second test, we show whether DLN exhibits intermittent backscatter for both constant time step and variable time step. In both tests, we consider DLN algorithm with three particular values of the parameter $\theta=\ 2/3,\ 2/\sqrt{5},\ 1$. In order to minimize the error constant and maintain strong stability qualities, the value $\theta=\ 2/3$ was proposed in \cite{DLN83_SIAM_JNA}. In \cite{kulikov2005one,kulikov2005stable}, the value $\theta=\ 2/\sqrt{5}$ was suggested to guarantee the best stability at infinity, i.e. for this value the scheme has good performance in long time simulation. In the case when $\theta=1$, the DLN method reduces to the symplectic midpoint rule, having the smallest error constant \cite{LPT22_Tech} and conserving all quadratic Hamiltonians. For test 2, we also consider $\theta=\ 0.95,\ 0.98$ so that we can check how DLN scheme behaves near $\theta=1$.
\subsection{A test with exact solution} (We choose the test problem from DeCaria, Layton and McLaughlin \cite{Victor2017artificial}). \label{sec:firstexample}
The first experiment tests the accuracy of the DLN algorithm and convergence rate of \eqref{v2} with constant time-step. It confirms the second-order convergence of the DLN method.
The following test has an exact solution for the 2D Navier Stokes problem. Let the domain be $\Omega=(-1,1)\times(-1,1)$. The exact solution is as follows:
\begin{align*}
    & u(x,y,t)=\pi\sin t(\sin 2\pi y\sin ^2 \pi x,-\sin 2\pi x\sin^2\pi y).\\
    &p(x,y,t)=\sin t\cos \pi x\sin\pi y.
\end{align*}
This is inserted into the CSM and the body force $f(x,t)$ is calculated. 
Taylor-Hood elements (P2-P1) were used in this test for the spatial discretization. We simulate the test up to $T=10.$ and take $C_s=0.1, \ \mu=0.4,\ \delta$ is taken to be the shortest edge of all triangles.
We test the constant step DLN with $\theta =2/3$. We define the error for velocity and pressure to be
    \begin{gather*}
    	 e_{n}^{w} = w(t_{n}) - w_{n}^{h}, \ \ \ e_{n}^{p} = p(t_{n}) - p_{n}^{h}. 
    \end{gather*}
    Let $k$ be the constant time step. We define two discrete norms for the errors as follows
    \begin{gather*}
    	\| | e^{w} | \|_{\infty,0} := \max_{0 \leq n \leq N} \| e_{n}^{w} \|_{L^{2}(\Omega)}, \ \ \ 
    	\| | e^{w} | \|_{0,0}:= \Big( \sum_{0 \leq n \leq N} k \| e_{n}^{w} \|_{L^{2}(\Omega)}^{2} \Big)^{1/2}, \\
    	\| | e^{p} | \|_{\infty,0} := \max_{0 \leq n \leq N} \| e_{n}^{p} \|_{L^{2}(\Omega)}, \ \ \ 
    	\| | e^{p} | \|_{0,0}:= \Big( \sum_{0 \leq n \leq N} k \| e_{n}^{p} \|_{L^{2}(\Omega)}^{2} \Big)^{1/2}.
    \end{gather*}
    \begin{table}[H]
    	\centering
    	\caption{Errors by $\| \cdot \|_{\infty,0}$-norm and Convergence Rate for the constant DLN with $\theta = 2/3$}
    	\begin{tabular}{cccccccc}
    		\hline
    		\hline
    		Time step $k$ & Mesh size $h$ & $\| | e^{w} | \|_{\infty,0}$ & Rate 
    		& $\| | \nabla e^{w} | \|_{\infty,0}$ & Rate  
    		& $\| | e^{p} | \|_{\infty,0}$   & Rate
    		\\
    		\hline
    		\hline
    		0.08  & 0.08571  & 6.0302       & -        & 56.8481   & -       & 10.8576    & -
    		\\
    		0.04  & 0.04221  & 0.0498844    & 6.9175   & 1.35745   & 5.3881  & 0.079143   & 7.1000
    		\\
    		0.02  & 0.02095  & 0.0119835    & 2.0575   & 0.399758  & 1.7637  & 0.0192928  & 2.0364
    		\\
    		0.01  & 0.01048  & 0.00297779   & 2.0087   & 0.10394   & 1.9434  & 0.00490525 & 1.9757
    		\\
    		\hline
    	\end{tabular}
     \label{tab:norm1}
    \end{table}
	\begin{table}[H]
		\centering
		\caption{Errors by $\| \cdot \|_{0,0}$-norm and Convergence Rate for the constant DLN with $\theta = 2/3$}
		\begin{tabular}{cccccccc}
			\hline
			\hline
			Time step $k$ & Mesh size $h$ & $\| | e^{w} | \|_{0,0}$ & Rate 
			& $\| | \nabla e^{w} | \|_{0,0}$ & Rate  
			& $\| | e^{p} | \|_{0,0}$   & Rate
			\\
			\hline
			\hline
			0.08  & 0.08571  & 7.8961       & -        & 79.3971   & -       & 12.3373  & -
			\\
			0.04  & 0.04221  & 0.107395     & 6.2001   & 3.06024   & 4.6974  & 0.143315   & 6.4277
			\\
			0.02  & 0.02095  & 0.024972     & 2.1045   & 0.900864  & 1.7643  & 0.0345612  & 2.0520
			\\
			0.01  & 0.01048  & 0.00617647   & 2.0155   & 0.234349  & 1.9427  & 0.00877951 & 1.9769
			\\
			\hline
		\end{tabular}
  \label{tab:norm2}
	\end{table}
In this test, \cref{tab:norm1} and \cref{tab:norm2} show that for true solution, we get predicted rate of convergence. We also see the same behavior for  $\theta= \ 2/\sqrt{5},\ 1$ (See Appendix). This is also evident that DLN allows larger time step to get the desired accuracy.
\subsection{Test2. Flow between offset cylinder} (We choose the  test problem from Jiang and Layton \cite [2D Test Problem]{jiang2016ev}).\label{sec:secondexample} This flow problem is tested to show whether or not the transfer of energy from fluctuations back to means in the turbulent flow using the Corrected Smagorinksy Model \eqref{csm0} happens. 
The domain is a disk with a smaller off center obstacle inside. Let $r_1=1,r_2=0.1,c=(c_1,c_2)=(1/2,0),$ then the domain is given by
\beas
\Omega=\{(x,y):x^2+y^2< r_1^2 \;\text{and}\; (x-c_1)^2+(y-c_2)^2> r_2^2\}.
\eeas
The flow is driven by a counterclockwise rotational body force
\beas
f(x,y,t)=(-4y*(1-x^2-y^2),4x*(1-x^2-y^2))^T,
\eeas
with no-slip boundary conditions on both circles. We discretize in space using Taylor-Hood elements. There are 400 mesh points around the outer circle and 100 mesh points around the inner circle. The flow is driven by a counterclockwise force (f=0 on the outer circle). Thus, the flow rotates about the origin and interacts with the immersed circle. We start the initial condition by solving the Stokes problem. We compute up to the final time $T_{final}=10$. Take $C_s=0.1, \mu=0.4, \delta$ is taken to be the shortest edge of all triangles $\approx 0.0112927$, Re=10,000.
For the DLN algorithm in \eqref{v2}, we compute the following quantities:
\begin{align*}
&\text{Model dissipation,}\mathcal{E}_{N}^{\tt MD}\!=\!\int_\Omega \!\Big( \!\frac{C_s^4 \!\delta^2}{\mu^2}\!\frac{\alpha_2\grad w_{N}^h\!+\!\alpha_1\grad w_{N\!-\!1}^h\!+\!\alpha_0\grad w_{N\!-\!2}^h}{\widehat{k_{N-1}}} \!\cdot \!\grad  w^h_{N\!-\!1,\!\beta}
\!+\!(C_s \!\delta)^2 \!|\grad \widetilde{w_{N\!-\!1,\beta}^h}||\grad w_{N\!-\!1,\!\beta}^h|^2 \!\Big) d x.\\
&\text{Effect of new term from CSM, }\mathcal{E}_{N}^{\tt CSMD}=\int_\Omega\Bigg(\frac{C_s^4\delta^2}{\mu^2}\frac{\alpha_2\grad w_{N}^h+\alpha_1\grad w_{N-1}^h+\alpha_0\grad w_{N-2}^h}{\widehat{k_{N-1}}}\cdot\grad  w^h_{N-1,\beta}
\Bigg)\,d x. \\
&\text{Numerical dissipation, }\mathcal{E}_{N}^{\tt ND}=\bigg\|\frac{\sum_{l=0}^{2}\lambda_l^{N-1} w_{N-2+l}^h}{\sqrt{\widehat{k_{N-1}}}}\bigg\|^2+\frac{C_s^4\delta^2}{\mu^2}\bigg\|\frac{\sum_{l=0}^{2}\lambda_l^{N-1} \grad w_{N-2+l}^h}{\sqrt{\widehat{k_{N-1}}}}\bigg\|^2.\\
&\text{Viscous dissipation, }\mathcal{E}_{N}^{\tt VD}=\nu\|\grad  w^h_{N-1,\beta}\|^2.\\
&\text{Kinetic Energy,}\ KE=\frac{1}{2}\|w^h_{N}\|^2.
\end{align*}
For constant time step $k=0.001$, we first run the test for $\theta=1$ , which reduces the method to a midpoint rule. We saw the results  were very close to the results in \cite{siddiqua2022numerical} (See Appendix). This indicates the accuracy of our implementation. However, it is unclear that the oscillations in the model dissipation are the effect of the new modeling term or normal ringing \cite{burkardt2022stress} which is seen in the standard Midpoint rule.
Hence, we need to investigate further. This suggests that approximation would be improved with a bit of numerical dissipation in the method provided that numerical dissipation vanishes as $k$ goes to $0$ 
and it does not dominate the 
physical dissipation.
\begin{figure}[H]
\subfloat[$\mathcal{E}_{N}^{\tt MD}$ ]{\includegraphics[width=8.2cm]{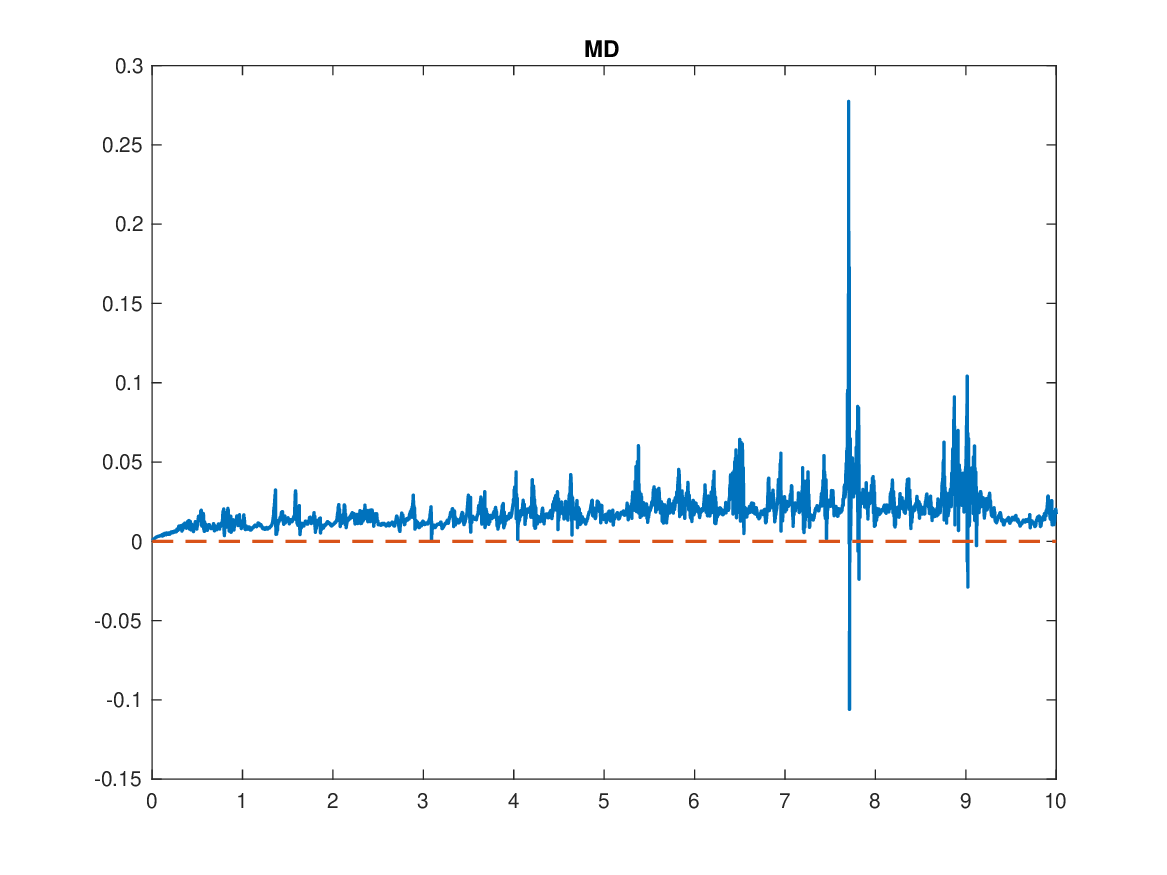}}
\subfloat[$\mathcal{E}_{N}^{\tt CSMD}$ ]{\includegraphics[width=8.2cm]{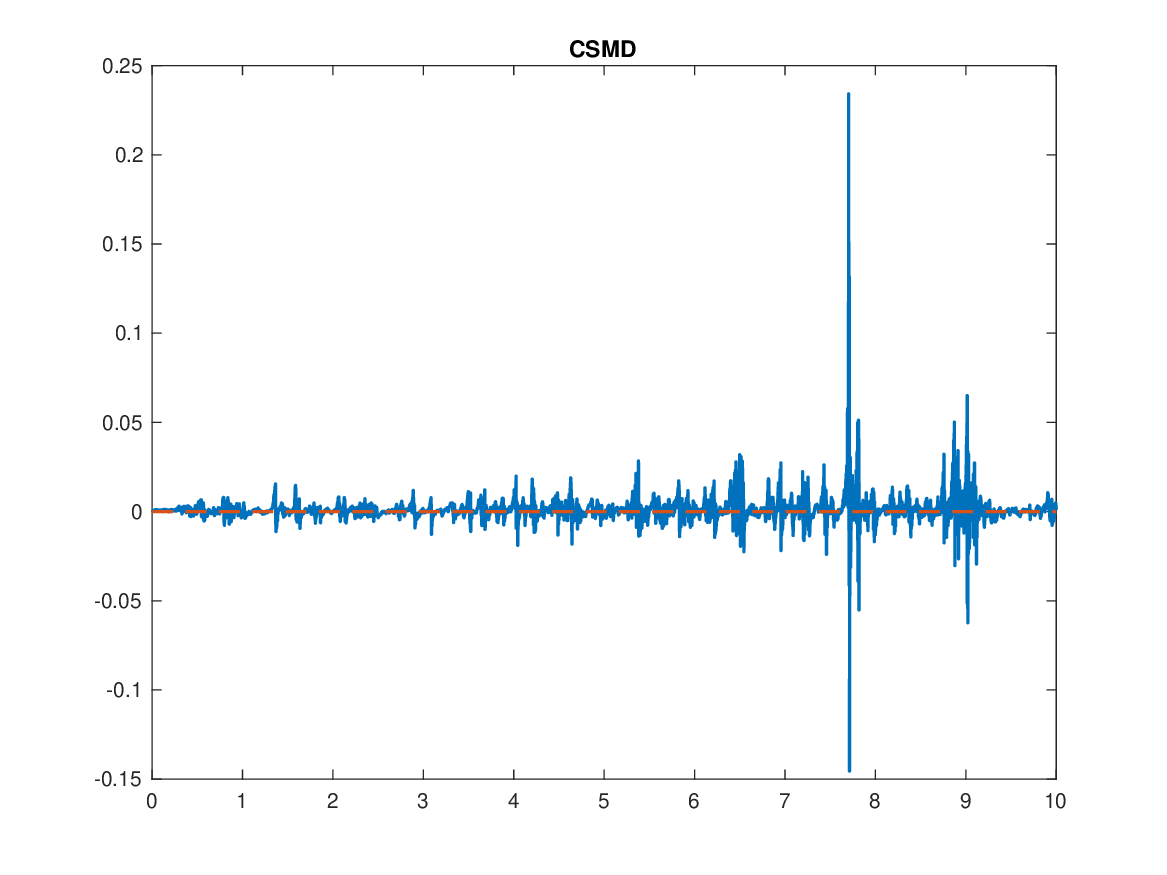}}\\
\subfloat[$\mathcal{E}_{N}^{\tt ND}$ ]{\includegraphics[width=8.2cm]{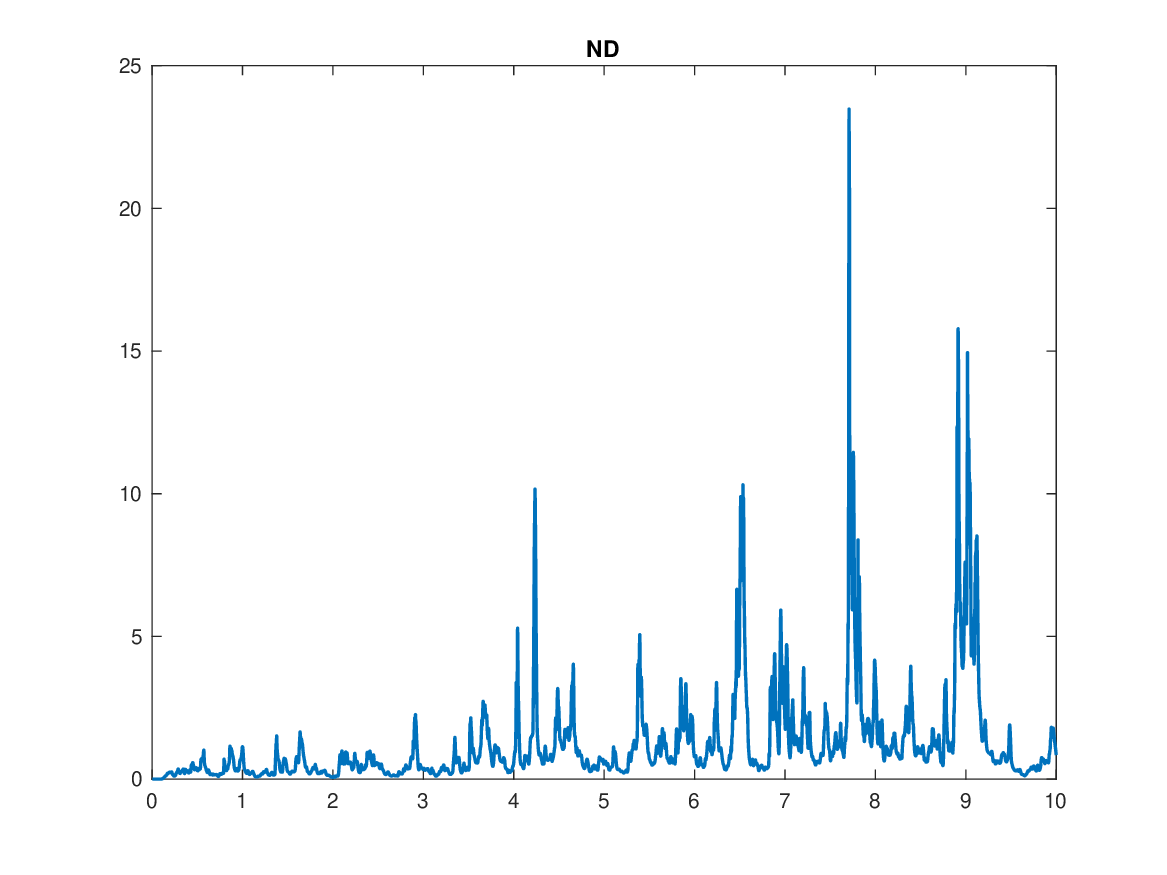}}
\subfloat[$\mathcal{E}_{N}^{\tt VD}$ ]{\includegraphics[width=8.2cm]{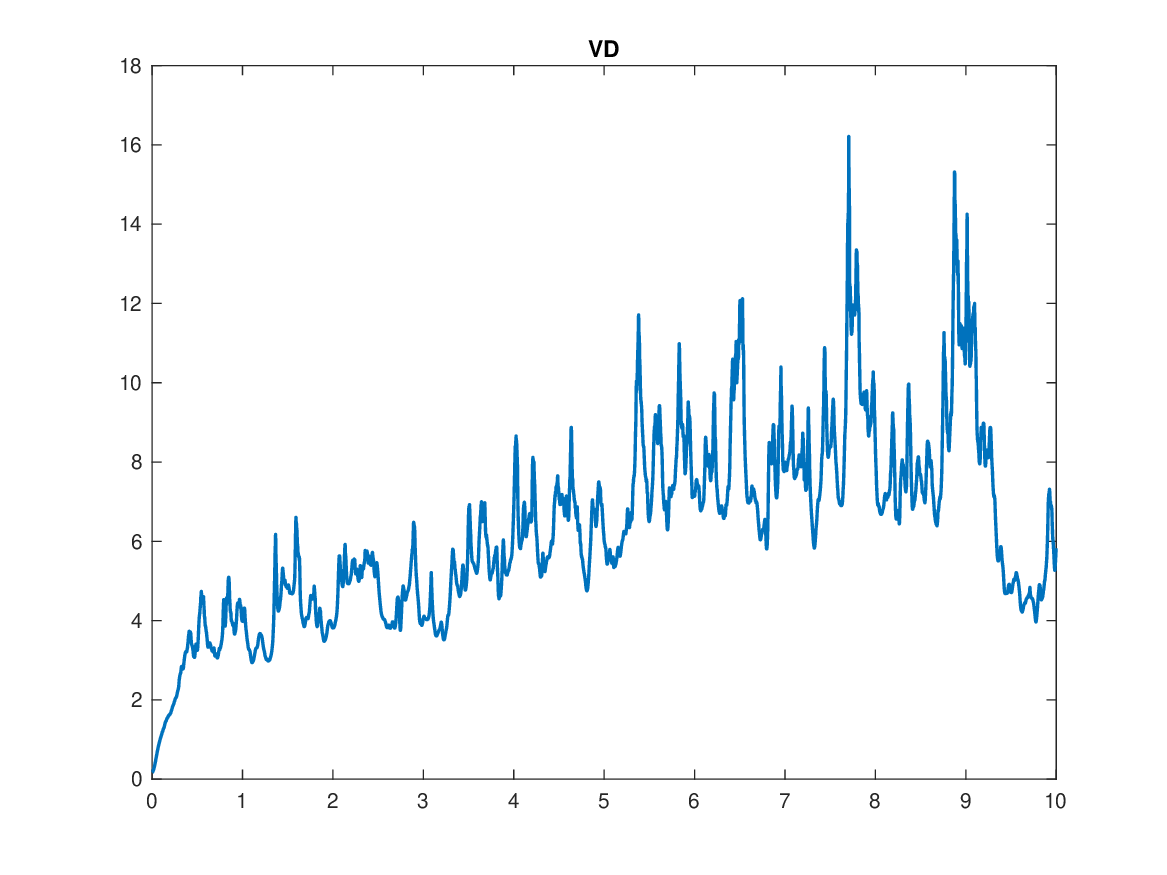}}
\caption{Constant time step DLN \eqref{v2} with $k =0.001,\ Re=10,000, \ \theta=0.98, \ C_s=0.1, \ \mu=0.4$. We do not see backscatter in $\mathcal{E}_{N}^{\tt MD}$ .}
\label{fig:plotdlnctheta0.98dt0.001}
\end{figure}
\begin{figure}[H]
\subfloat[$\mathcal{E}_{N}^{\tt MD}$ ]{\includegraphics[width=8.2cm]{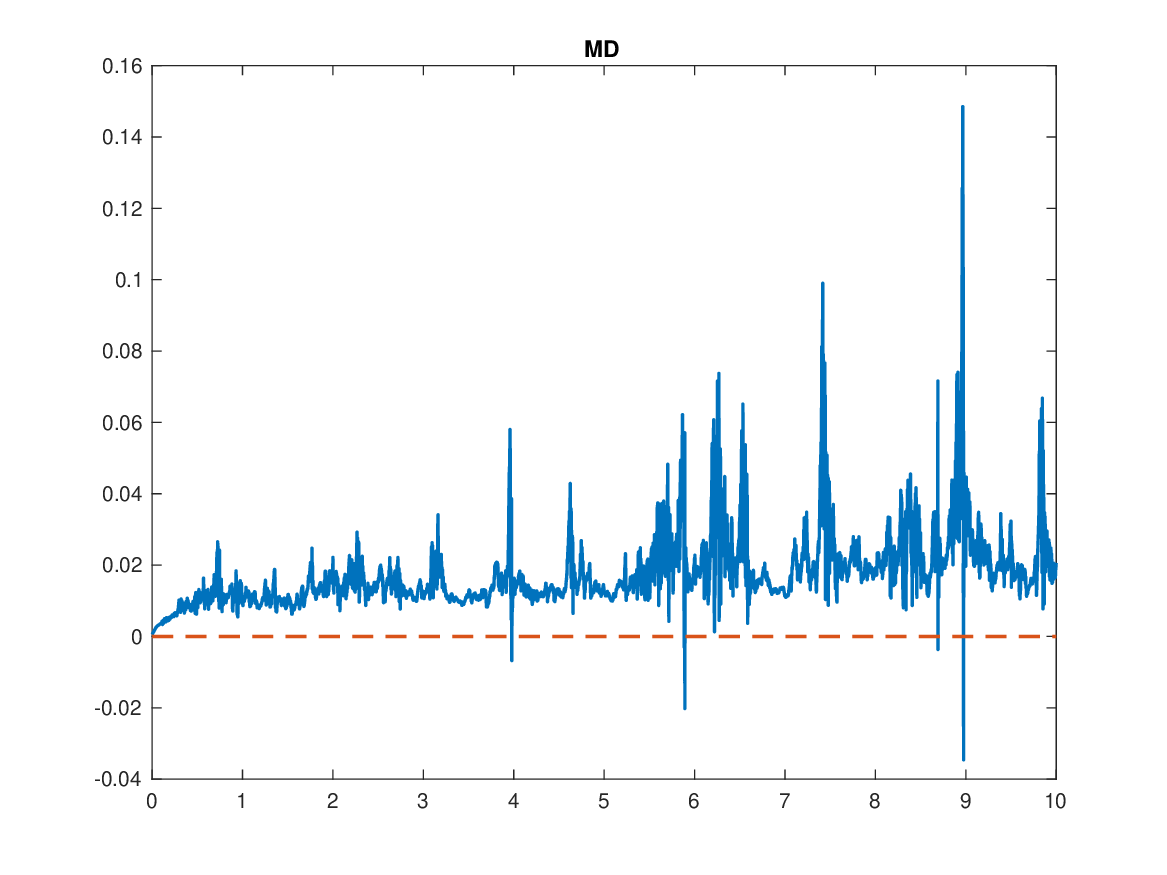}}
\subfloat[$\mathcal{E}_{N}^{\tt CSMD}$ ]{\includegraphics[width=8.2cm]{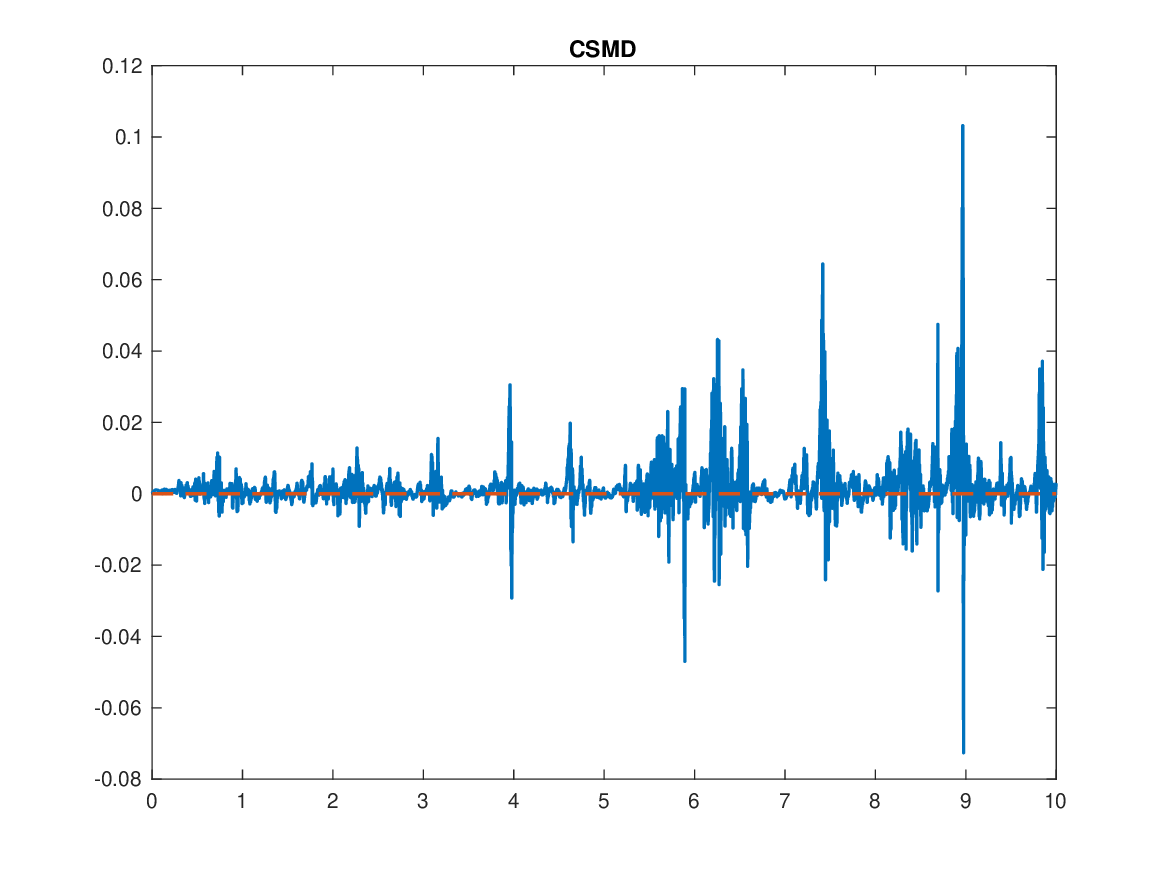}}\\
\subfloat[$\mathcal{E}_{N}^{\tt ND}$ ]{\includegraphics[width=8.2cm]{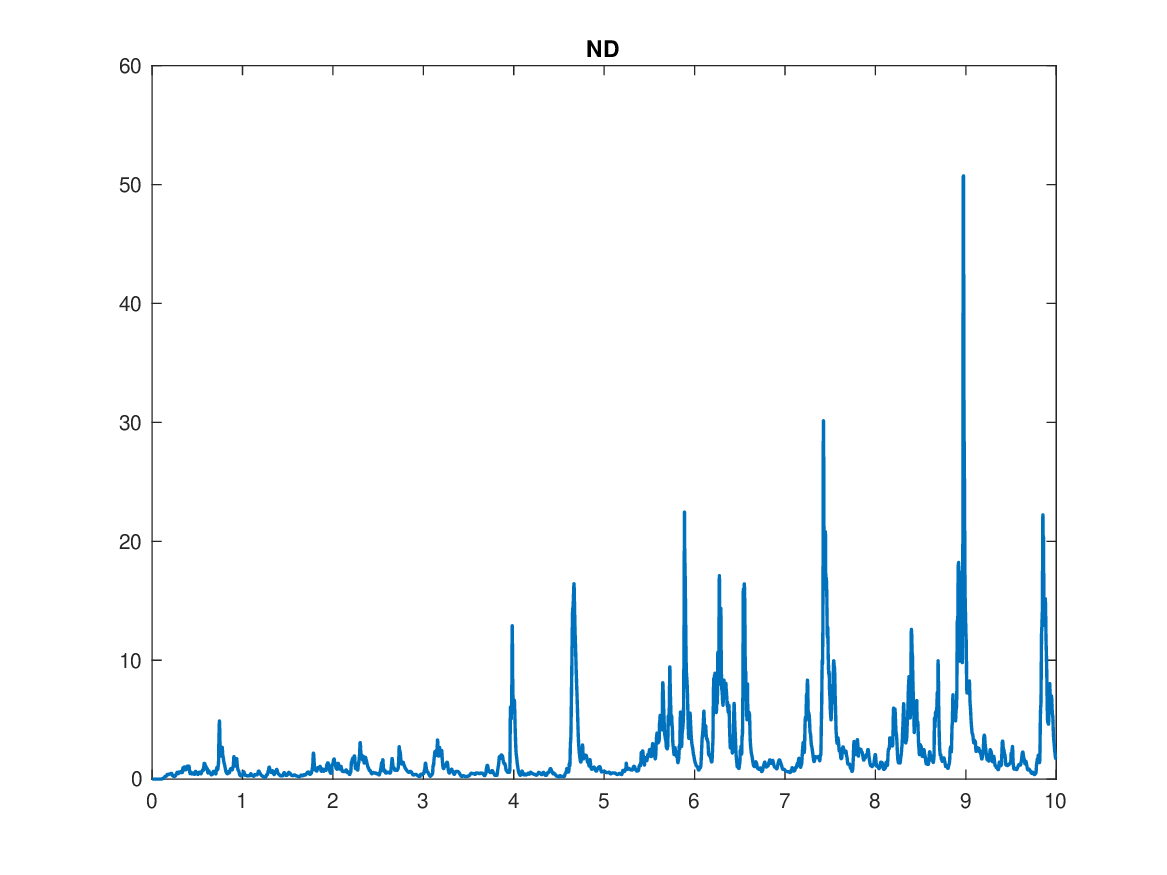}}
\subfloat[$\mathcal{E}_{N}^{\tt VD}$ ]{\includegraphics[width=8.2cm]{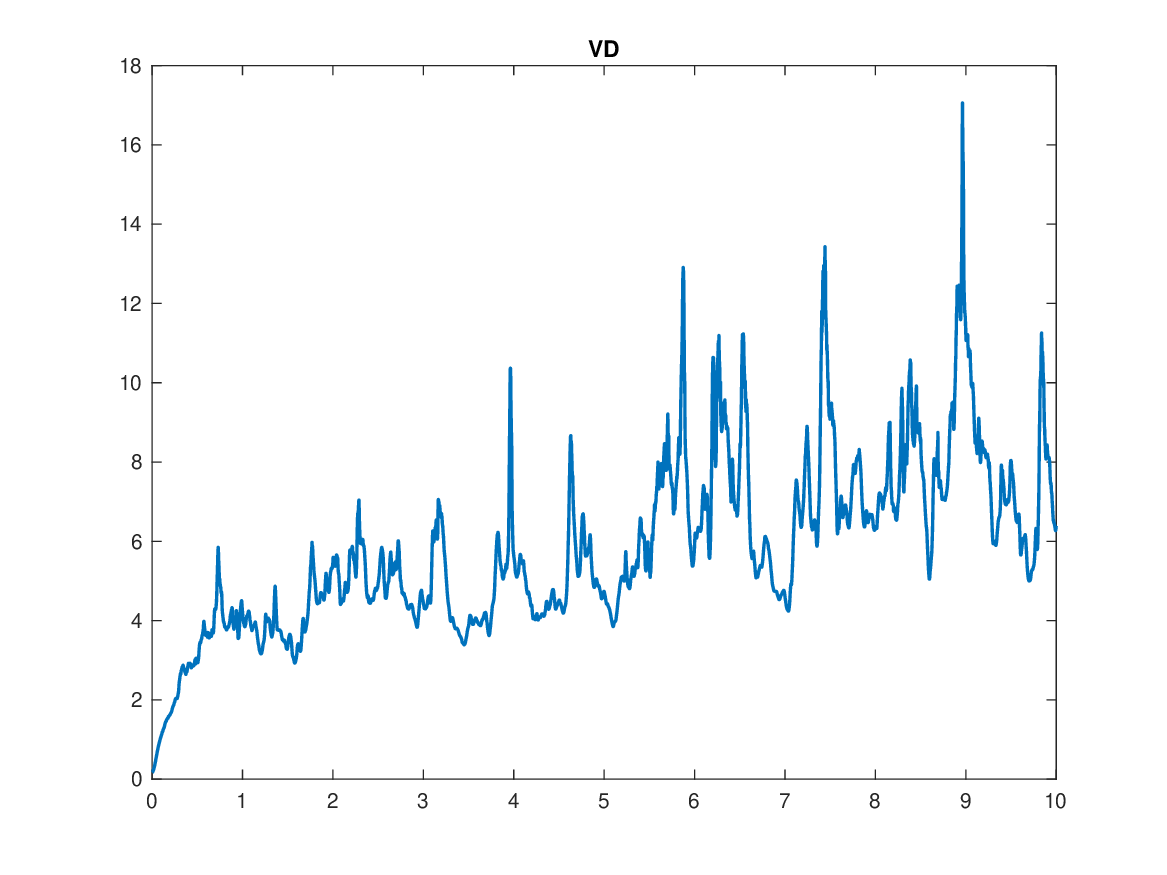}}
\caption{Constant time step DLN \eqref{v2} with $k =0.001,\ Re=10,000, \ \theta=0.95, \ C_s=0.1, \ \mu=0.4$. We do not see backscatter in $\mathcal{E}_{N}^{\tt MD}$ .}
\label{fig:plotdlnctheta0.95dt0.001}
\end{figure}
\begin{figure}[H]
\subfloat[$\mathcal{E}_{N}^{\tt MD}$ ]{\includegraphics[width=8.2cm]{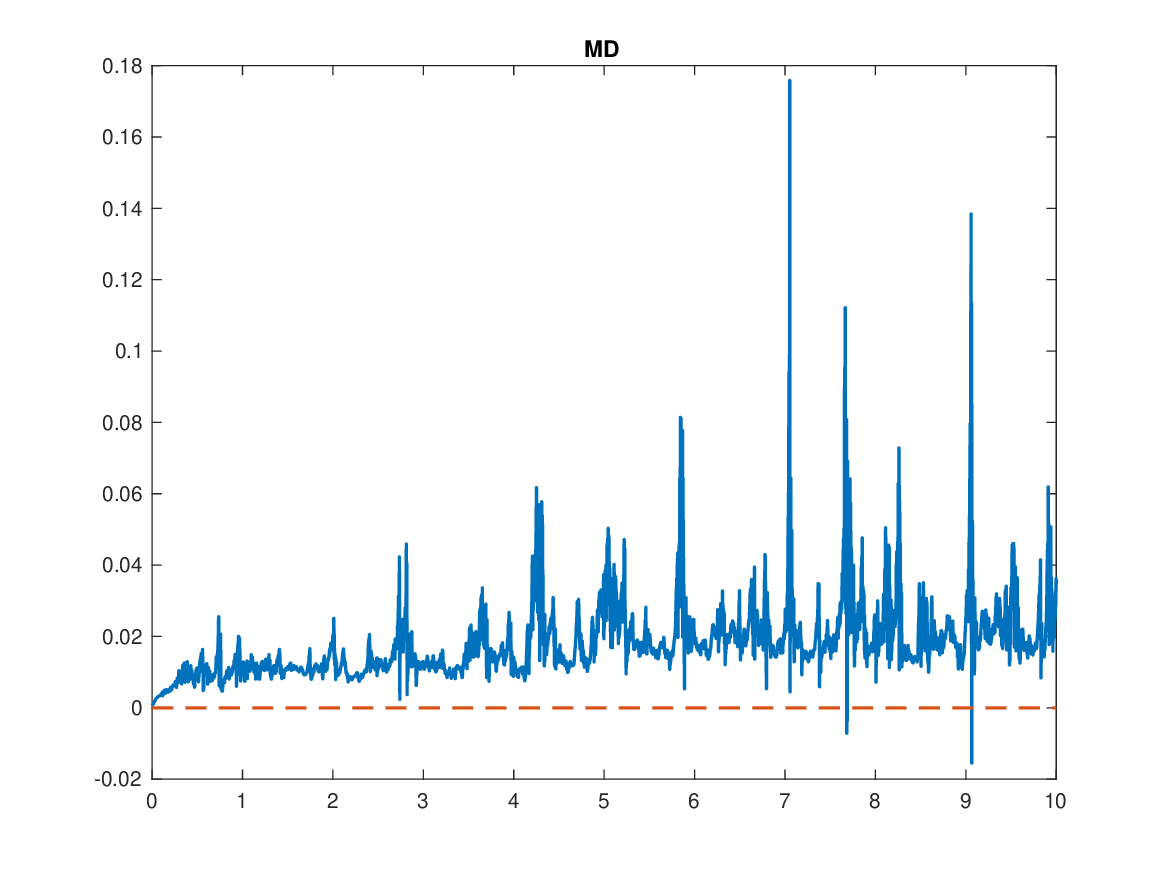}}
\subfloat[$\mathcal{E}_{N}^{\tt CSMD}$ ]{\includegraphics[width=8.2cm]{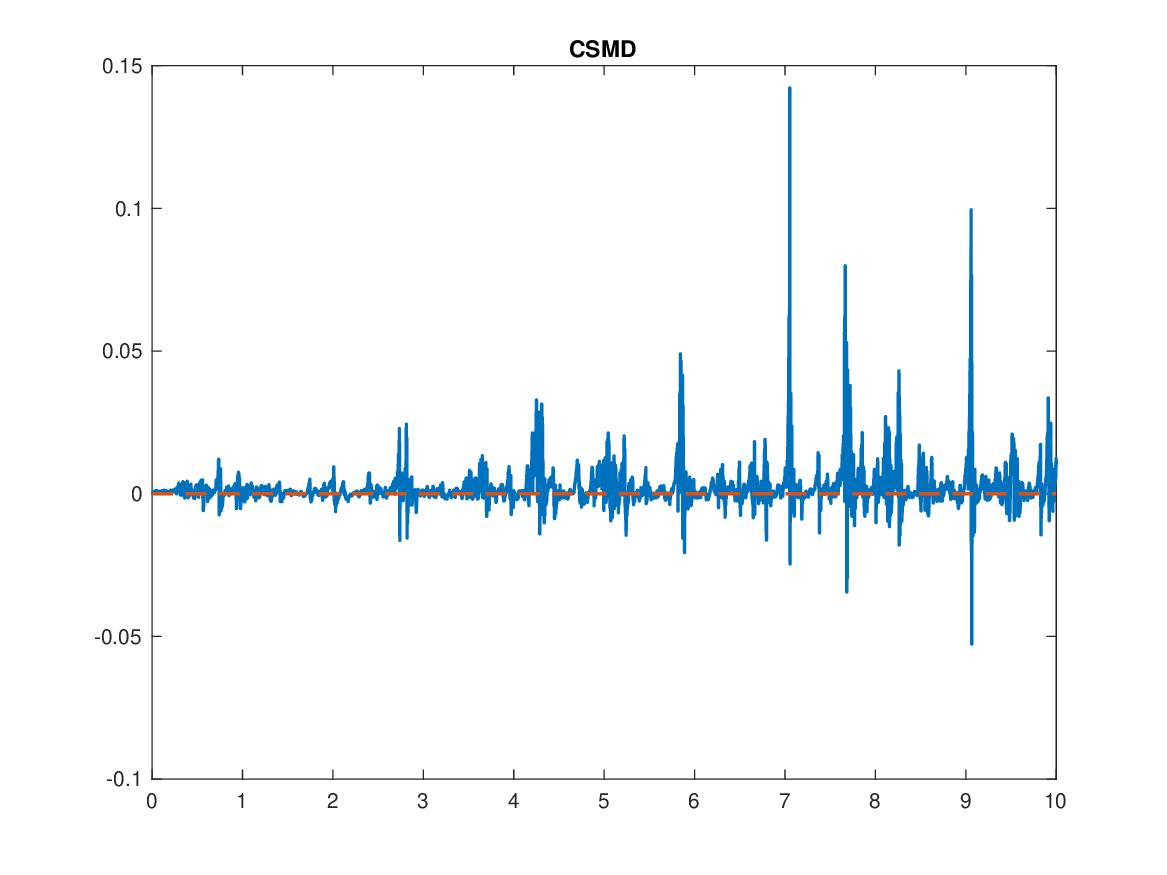}}\\
\subfloat[$\mathcal{E}_{N}^{\tt ND}$ ]{\includegraphics[width=8.2cm]{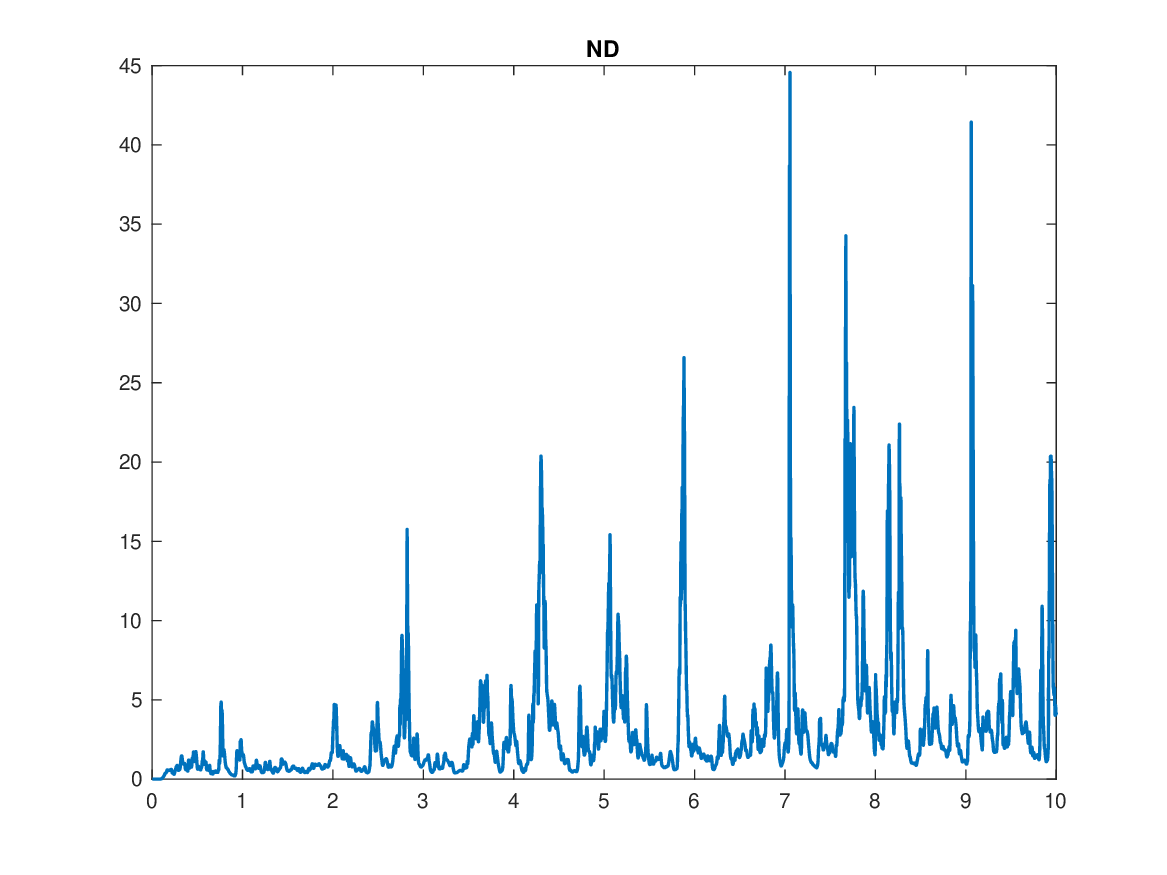}}
\subfloat[$\mathcal{E}_{N}^{\tt VD}$ ]{\includegraphics[width=8.2cm]{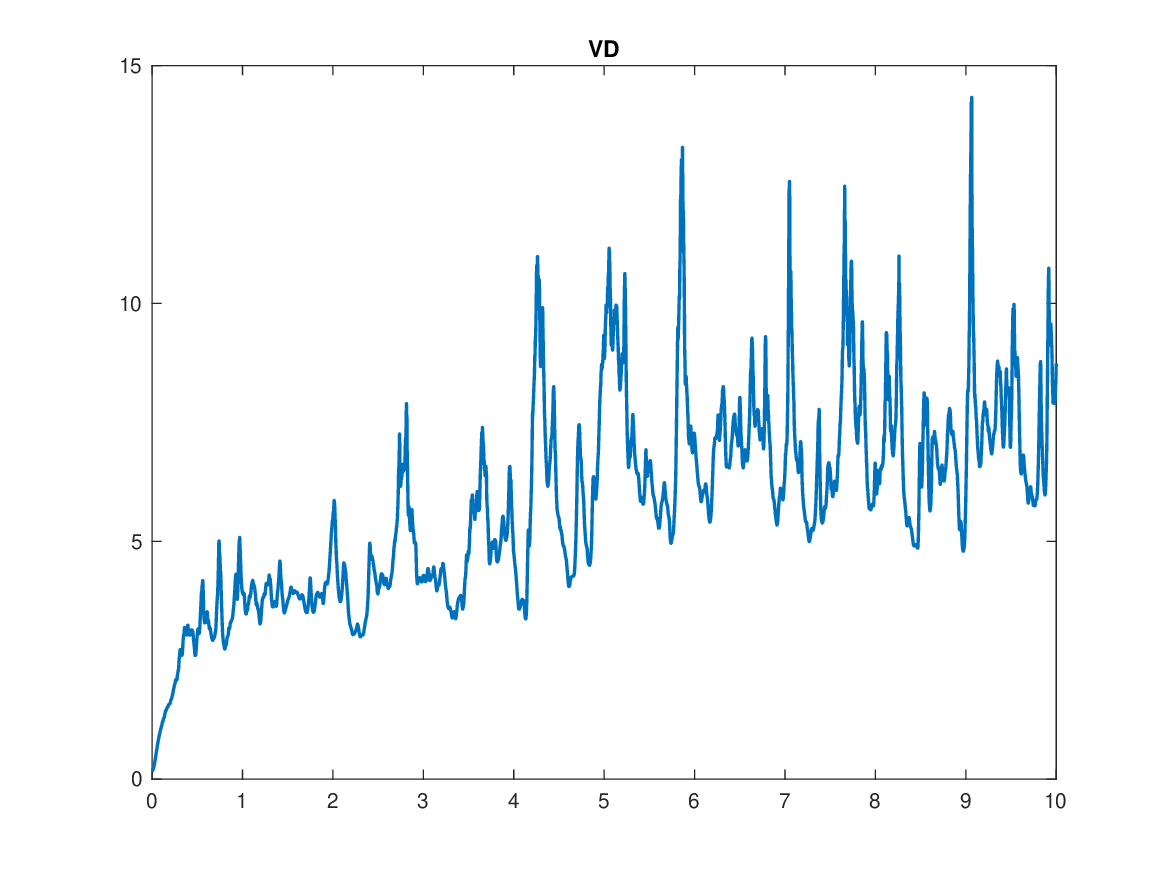}}
\caption{Constant time step DLN \eqref{v2} with $k =0.001,\ Re=10,000, \ \theta=2/\sqrt{5}, \ C_s=0.1, \ \mu=0.4$. We do not see backscatter in $\mathcal{E}_{N}^{\tt MD}$ .}
\label{fig:plotdlnctheta2bysqrt5dt0.001}
\end{figure}
\begin{figure}[H]
\subfloat[$\mathcal{E}_{N}^{\tt MD}$ ]{\includegraphics[width=8.2cm]{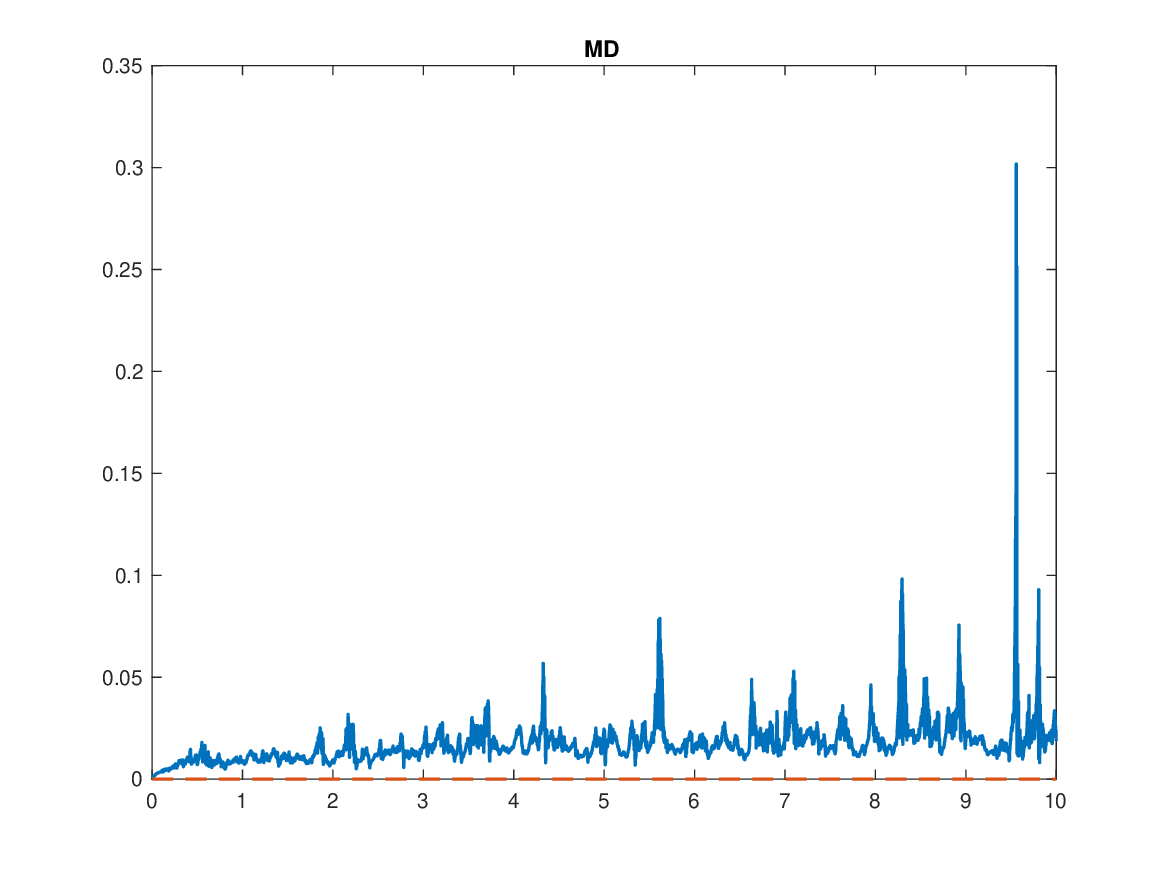}}
\subfloat[$\mathcal{E}_{N}^{\tt CSMD}$ ]{\includegraphics[width=8.2cm]{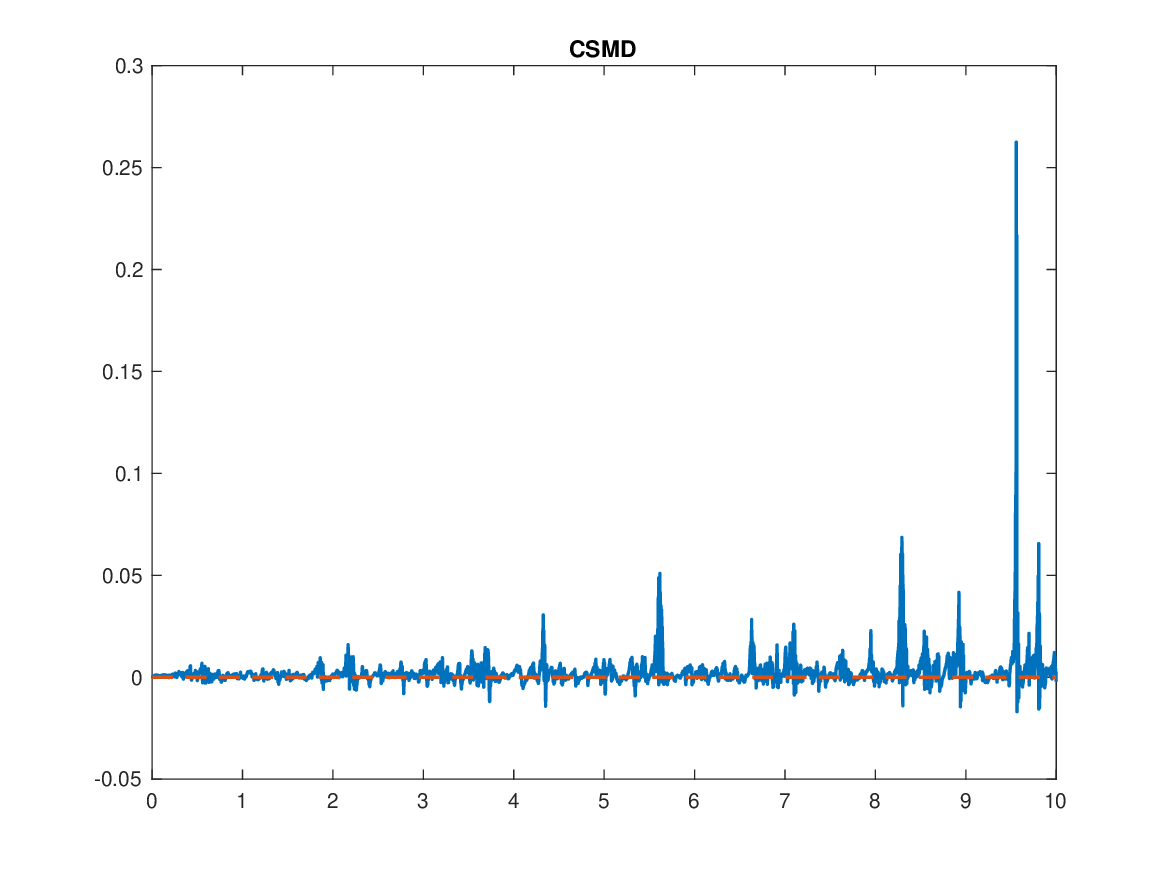}}\\
\subfloat[$\mathcal{E}_{N}^{\tt ND}$ ]{\includegraphics[width=8.2cm]{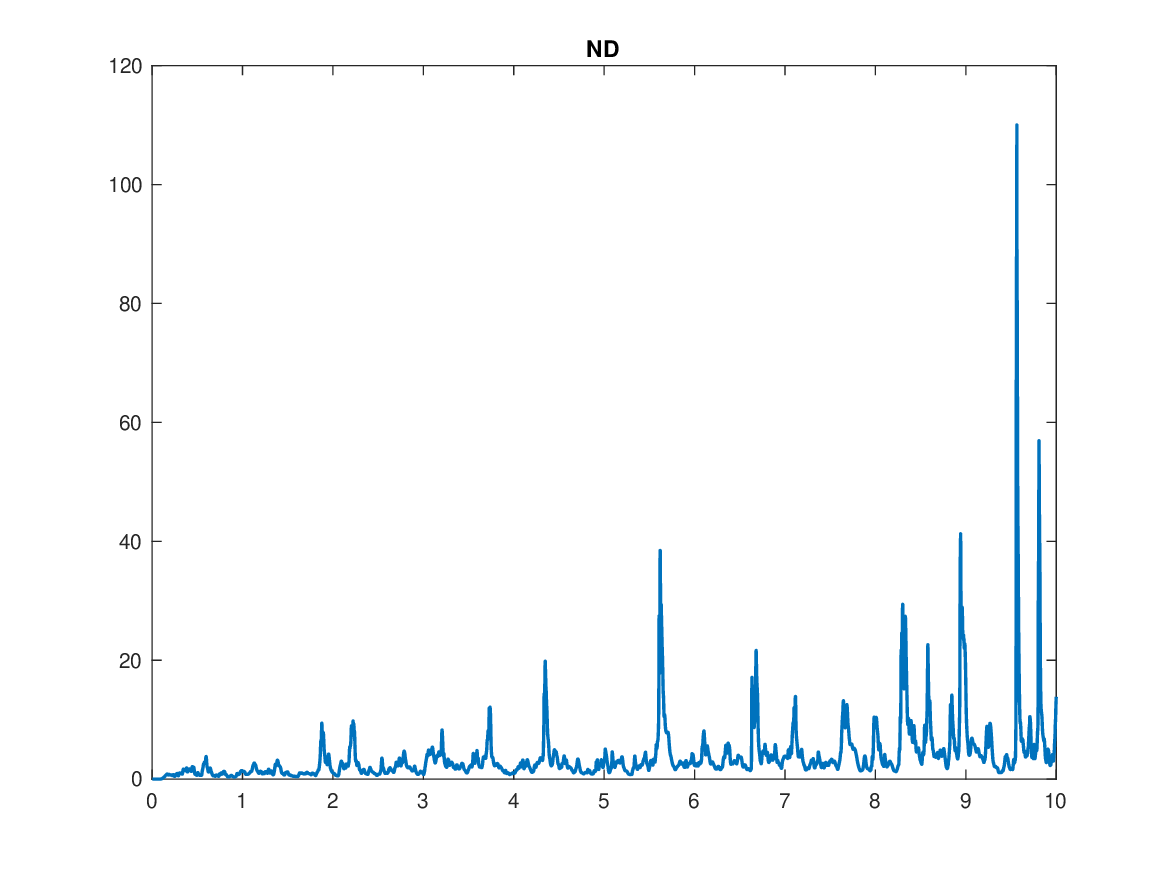}}
\subfloat[$\mathcal{E}_{N}^{\tt VD}$]{\includegraphics[width=8.2cm]{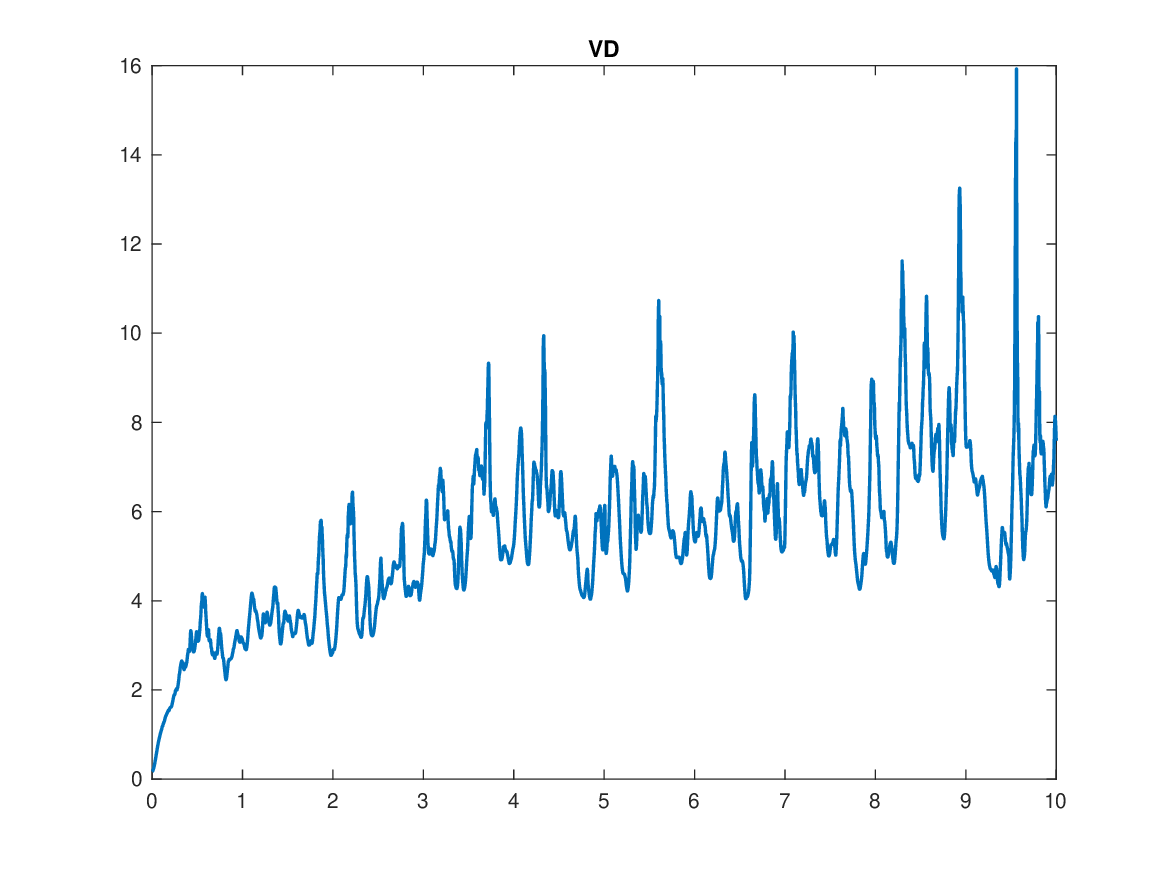}}
\caption{Constant time step DLN \eqref{v2} with $k =0.001,\ Re=10,000, \ \theta=2/3, \ C_s=0.1, \ \mu=0.4$. We do not see backscatter in $\mathcal{E}_{N}^{\tt MD}$.}
\label{fig:plotdlnctheta2by3dt0.001}
\end{figure}
\cref{fig:plotdlnctheta0.98dt0.001}, \cref{fig:plotdlnctheta0.95dt0.001}, and \cref{fig:plotdlnctheta2bysqrt5dt0.001} look plausible since the solution seems quasi-periodic even though there is very negligible backscatter in the model dissipation. We notice that numerical dissipation, $\mathcal{E}_{N}^{\tt ND}$ is overwhelming the effect of new backscatter term, $\mathcal{E}_{N}^{\tt CSMD}$. That's why we do not see much backscatter. In \cref{fig:plotdlnctheta2by3dt0.001}, we see $\mathcal{E}_{N}^{\tt ND}$ is still hiding the effect of small negative values in the $\mathcal{E}_{N}^{\tt CSMD}$ from the model dissipation, $\mathcal{E}_{N}^{\tt MD}$. Hence, there is no predicted backscatter which we believe is due to $\mathcal{E}_{N}^{\tt ND}$ being too big.
There is substantial difference between the DLN algorithm for the Corrected Smagorinksy when $\theta=1$ and other values of $\theta$. The only way to tell which is correct by doing the adapting the time step to control the $\mathcal{E}_{N}^{\tt ND}$. Since $\mathcal{E}_{N}^{\tt ND}$ is overwhelming the effect of new term, it emphasize that the adaptivity based on dissipation criteria is important. 
The minimum dissipation criteria by F. Capuano, B. Sanderse, E. M. De Angelis, and G. Coppola \cite{capuano2017minimum} is to keep the ratio of numerical dissipation, $\mathcal{E}_{N}^{\tt ND}$ and the viscous dissipation, $\mathcal{E}_{N}^{\tt VD}$ less than some tolerance, Tol for adapting the time step. Thus adapt for 
$\chi=\bigg|\frac{\mathcal{E}_{N}^{\tt ND}}{\mathcal{E}_{N}^{\tt ND}}\bigg|<\text{Tol}.$
The number of time steps to reach the final time is the indicator that how sensitive the DLN method is. We fix the final time, $T=10$ and report the number of time steps to reach there in \cref{tab:totaltime steps} for each values of $\theta$. However, the code has maximum time step $k_{max}=0.025$,  minimum time step $k_{min}=0.001$, and the initial time step $k_0=0.0001$. That means even if we use the adaptivity by doing dissipation criteria, there are very few times when the criteria is not satisfied and we see spiking which is clearly visible in \cref{fig:plotdlntheta0.98tol0.01T10}, \cref{fig:dlntheta0.95tol0.05T10}, \cref{fig:dlntheta2bysqrt5tol0.15T10}, \cref{fig:dlntheta2by3tol0.15T10}. This is the evidence of sensitivity of the CSM problem. To see the backscatter, we set the $\text{Tol}=0.01$ for $\theta=0.98$, we set the $\text{Tol}=0.05$ for $\theta=0.95$, we set the $\text{Tol}=0.15$ for $\theta=\frac{2}{\sqrt{5}}$. We started the test for all $\theta$ with $\text{Tol}=0.01$ and gradually increased $\text{Tol}$ when we failed to see the backscatter in $\mathcal{E}_{N}^{\tt MD}$. By increasing $\text{Tol}$, we are allowing more $\mathcal{E}_{N}^{\tt ND}$. Hence instead of seeing less backscatter, we see more backscatter. This is an anomaly for which we do not have an explanation. It could be due to the fact that increased tolerance led to larger time steps, see \cref{tab:totaltime steps}. But in the case of  $\theta=\frac{2}{3}$ for $\text{Tol}=0.15$, \cref{fig:dlntheta2by3tol0.15T10} looks plausible even though we fail to see backscatter in $\mathcal{E}_{N}^{\tt MD}$. Hence, when $\theta$ is smaller, there is more $\mathcal{E}_{N}^{\tt ND}$ and we see evidence  of backscatter in $\mathcal{E}_{N}^{\tt CSMD}$, but it's smaller than the $\mathcal{E}_{N}^{\tt ND}$ in the method. 
Next, we check the effect of the new term in the model on Kinetic Energy (KE) in the flow for each case. This new term may be highly fluctuating, but it has an impact on KE. Since we get one pattern for all the cases when backscatter in $\mathcal{E}_{N}^{\tt MD}$ is happening and another pattern for all the cases when backscatter in $\mathcal{E}_{N}^{\tt MD}$ is not happening, we show one example for $\theta=0.95$ in \cref{fig:kev0.95}. Bigger KE means less energy dissipation. When backscatter is happening, we notice from \cref{fig:kev0.95} that after a certain amount of time, the KE stabilizes or only varies within a small range.
\begin{figure}[H]
\subfloat[$\mathcal{E}_{N}^{\tt MD}$]{\includegraphics[width=8.2cm]{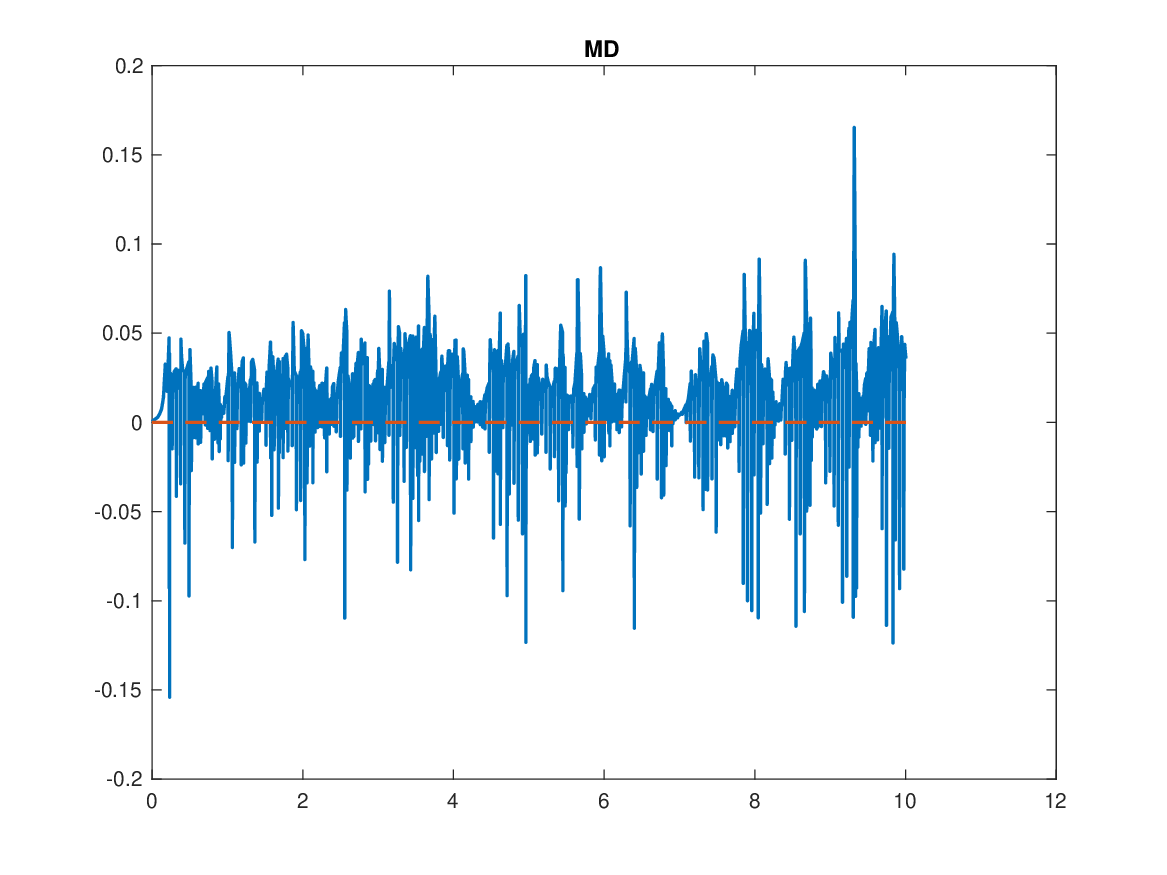}}
\subfloat[$\mathcal{E}_{N}^{\tt CSMD}$]{\includegraphics[width=8.2cm]{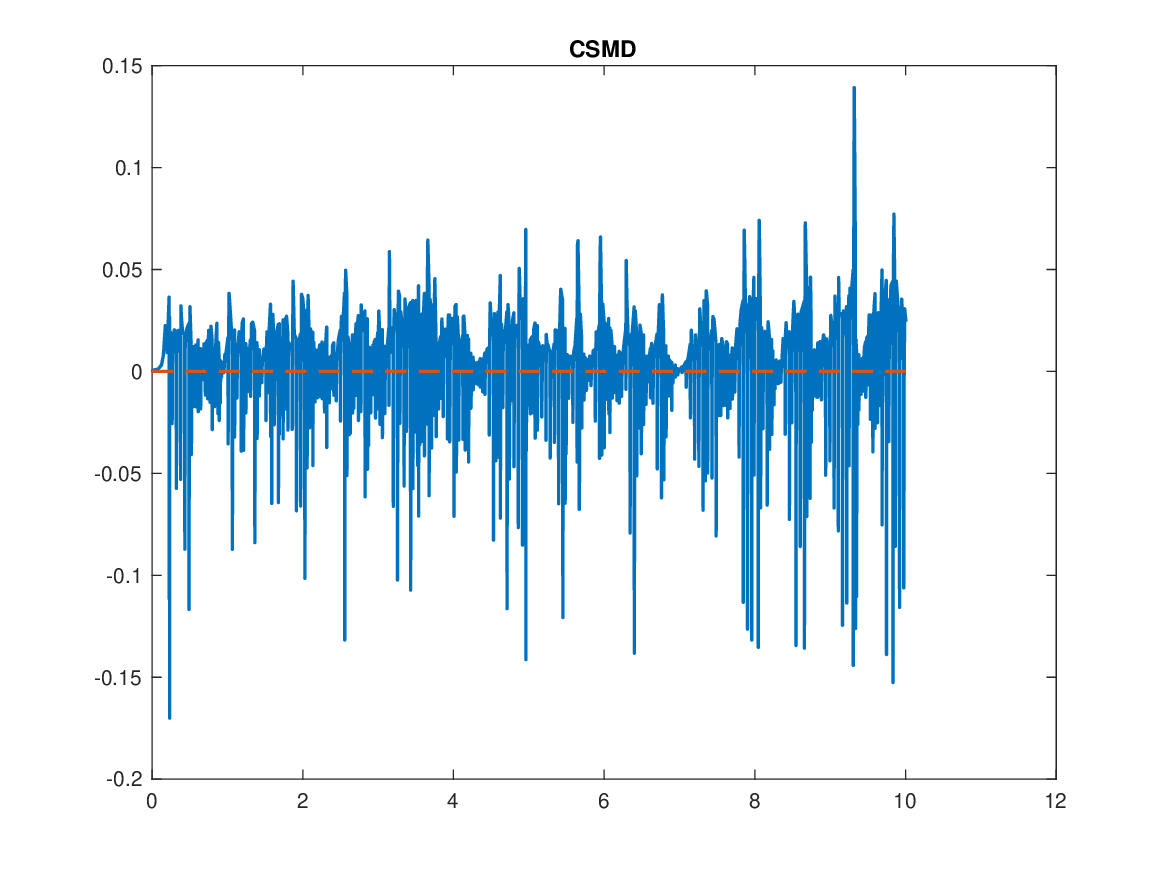}}\\
\subfloat[$\mathcal{E}_{N}^{\tt ND}$]{\includegraphics[width=8.2cm]{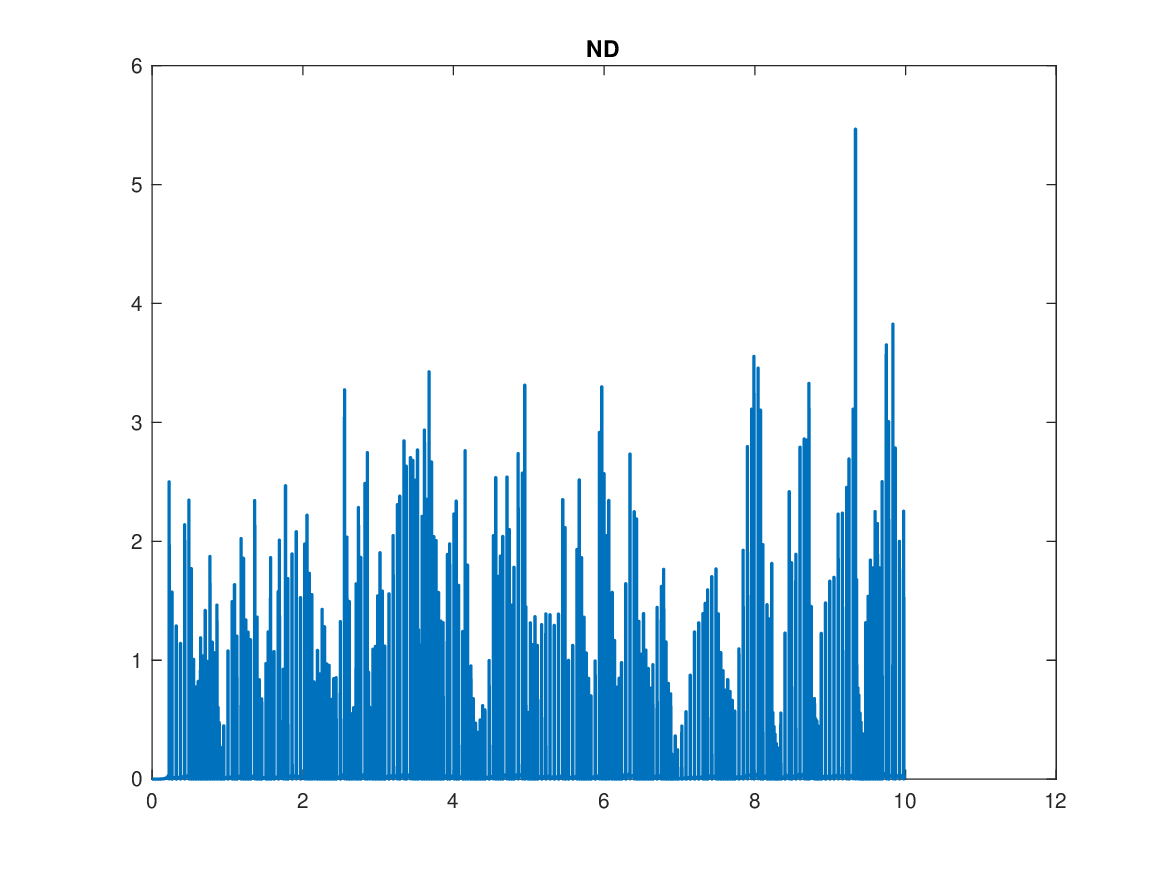}}
\subfloat[$\mathcal{E}_{N}^{\tt VD}$]{\includegraphics[width=8.2cm]{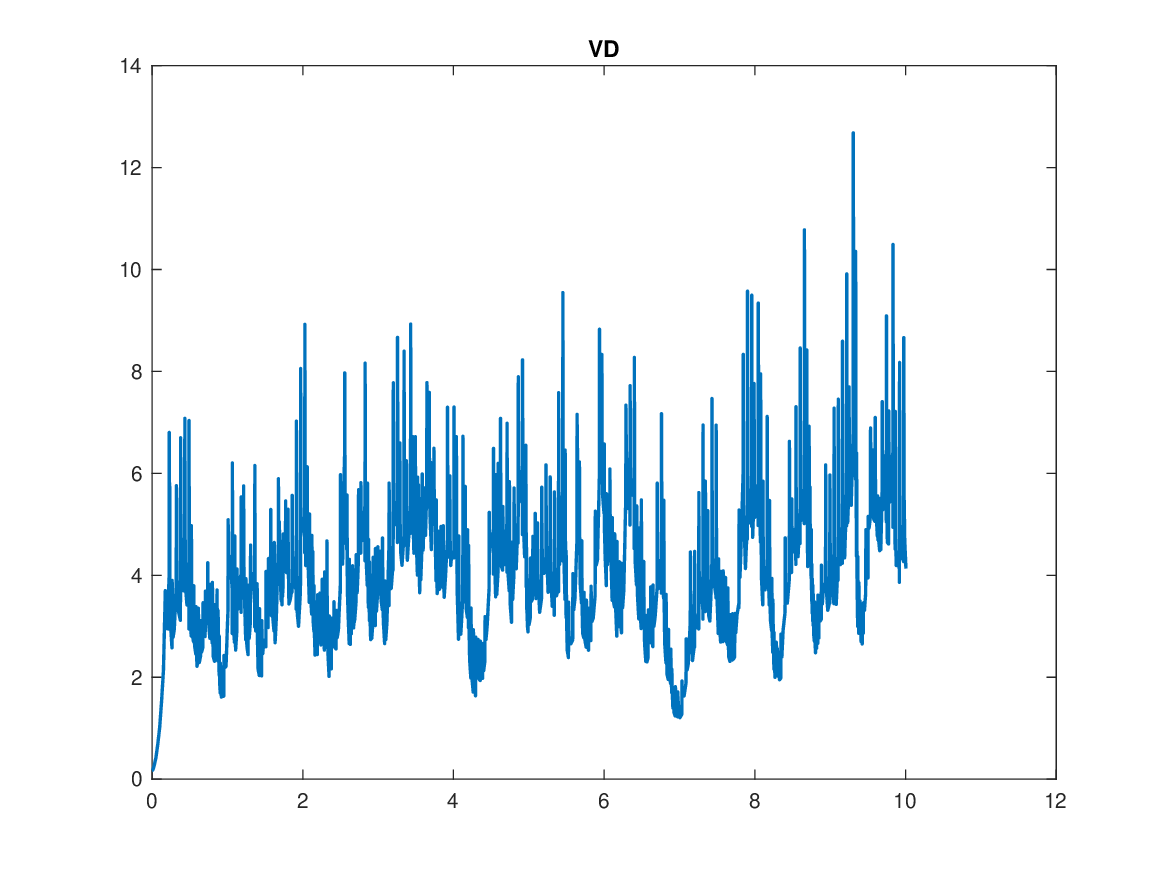}}
\caption{Variable Step DLN \eqref{v2} with $Tol =0.01,\ Re=10,000, \ \theta=0.98, \ C_s=0.1, \ \mu=0.4$. We see backscatter in $\mathcal{E}_{N}^{\tt MD}$.}
\label{fig:plotdlntheta0.98tol0.01T10}
\end{figure}
\begin{figure}[H]
\subfloat[$\mathcal{E}_{N}^{\tt MD}$]{\includegraphics[width=8.2cm]{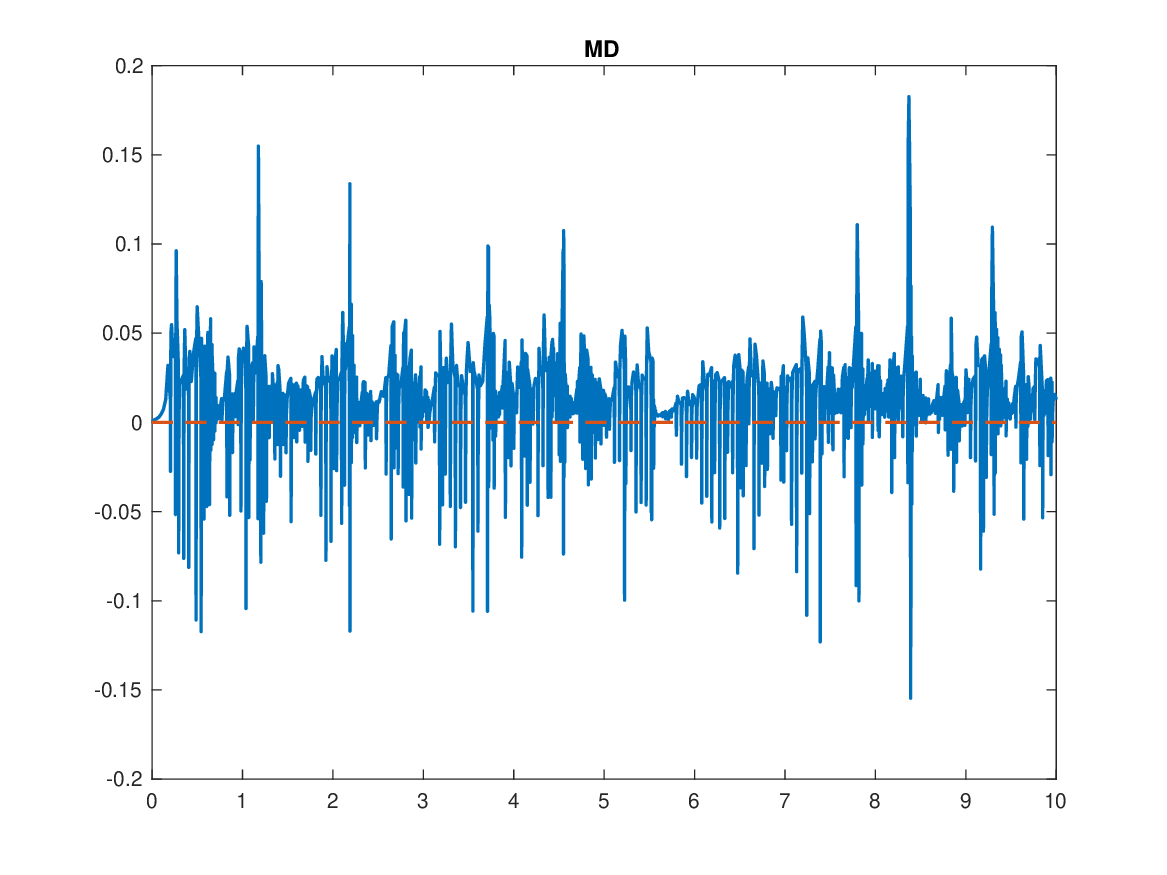}}
\subfloat[$\mathcal{E}_{N}^{\tt CSMD}$]{\includegraphics[width=8.2cm]{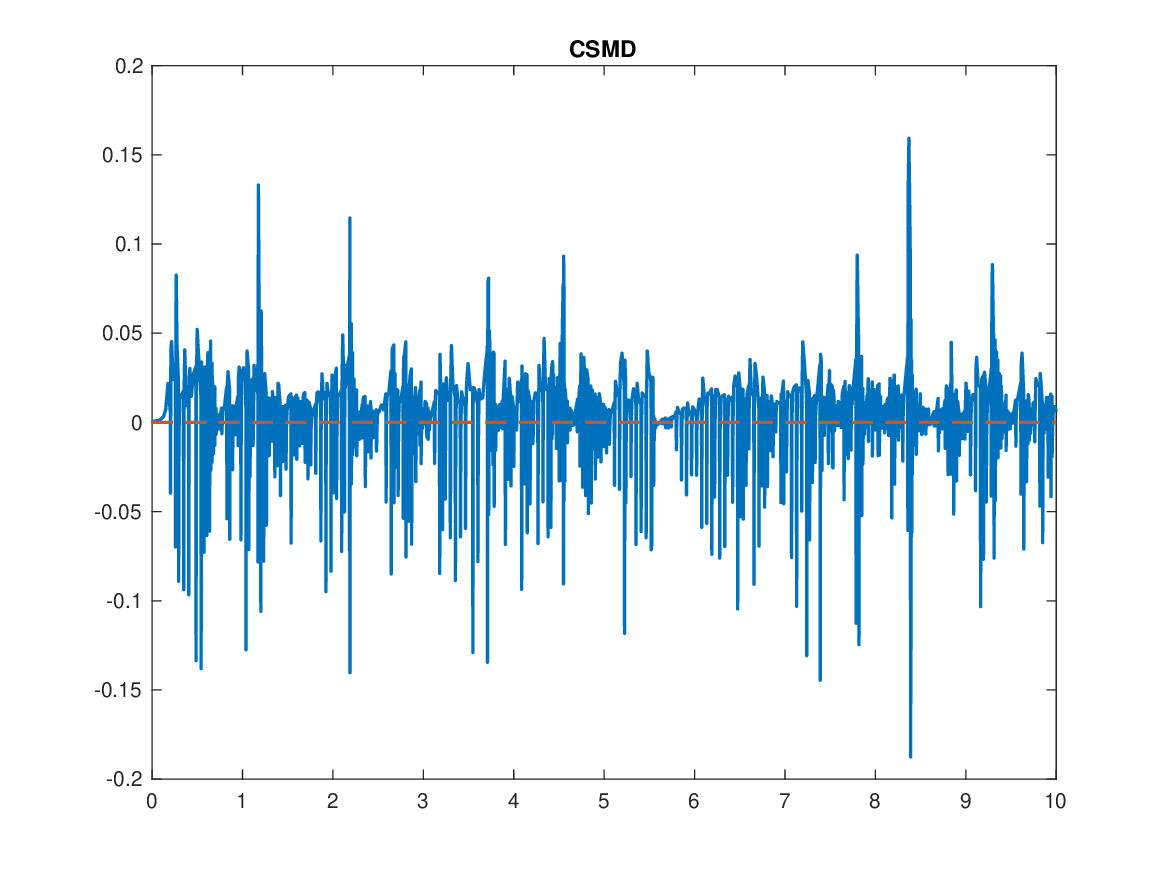}}\\
\subfloat[$\mathcal{E}_{N}^{\tt ND}$]{\includegraphics[width=8.2cm]{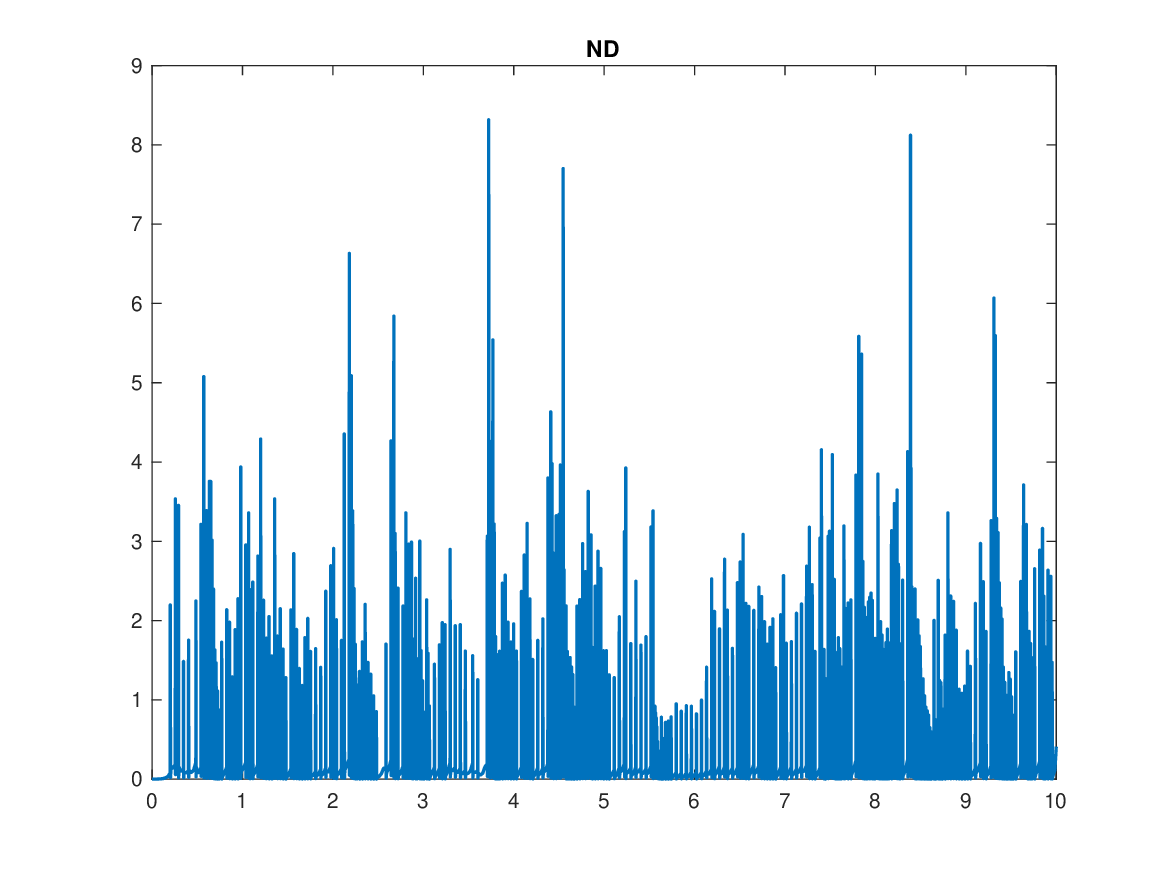}}
\subfloat[$\mathcal{E}_{N}^{\tt VD}$]{\includegraphics[width=8.2cm]{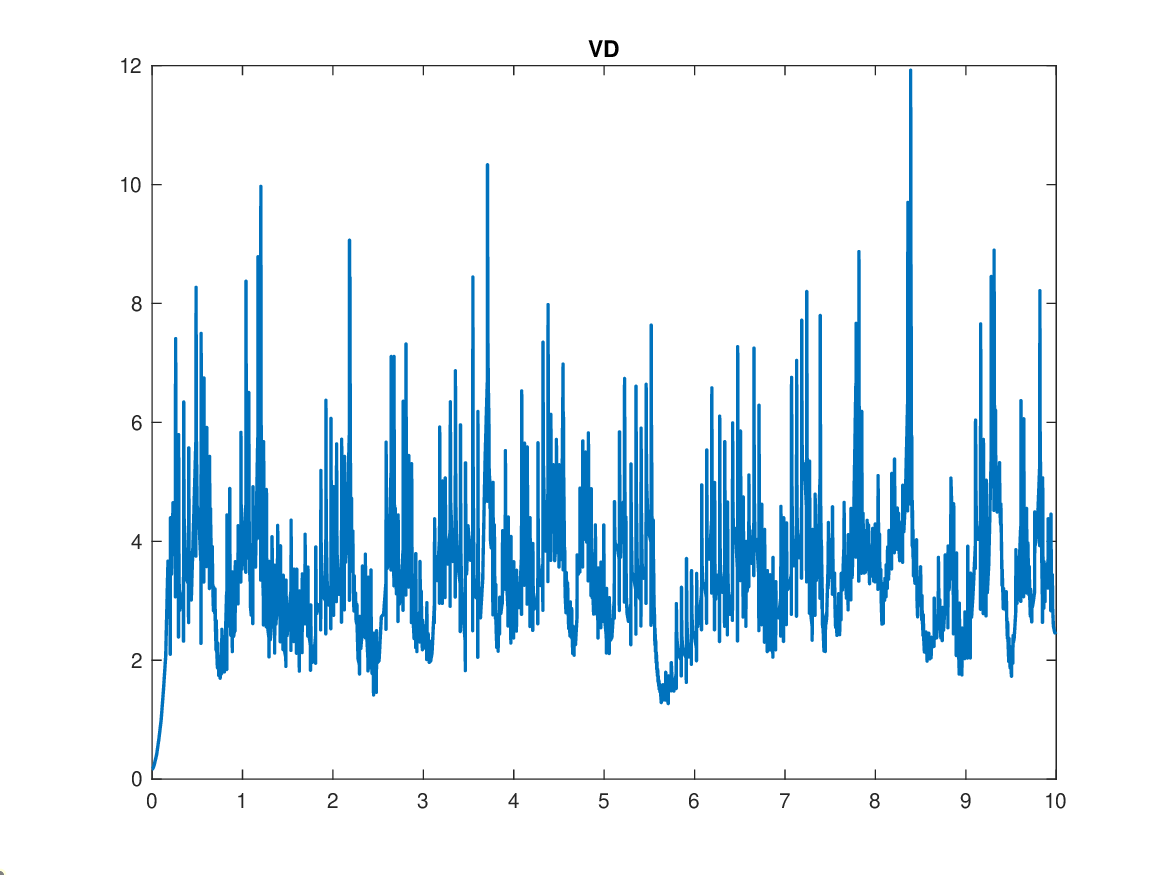}}
\caption{Variable Step DLN \eqref{v2} with $Tol =0.05,\ Re=10,000, \ \theta=0.95, \ C_s=0.1, \ \mu=0.4$. We see backscatter in $\mathcal{E}_{N}^{\tt MD}$.}
\label{fig:dlntheta0.95tol0.05T10}
\end{figure}
\begin{figure}[H]
\subfloat[$\mathcal{E}_{N}^{\tt MD}$]{\includegraphics[width=8.2cm]{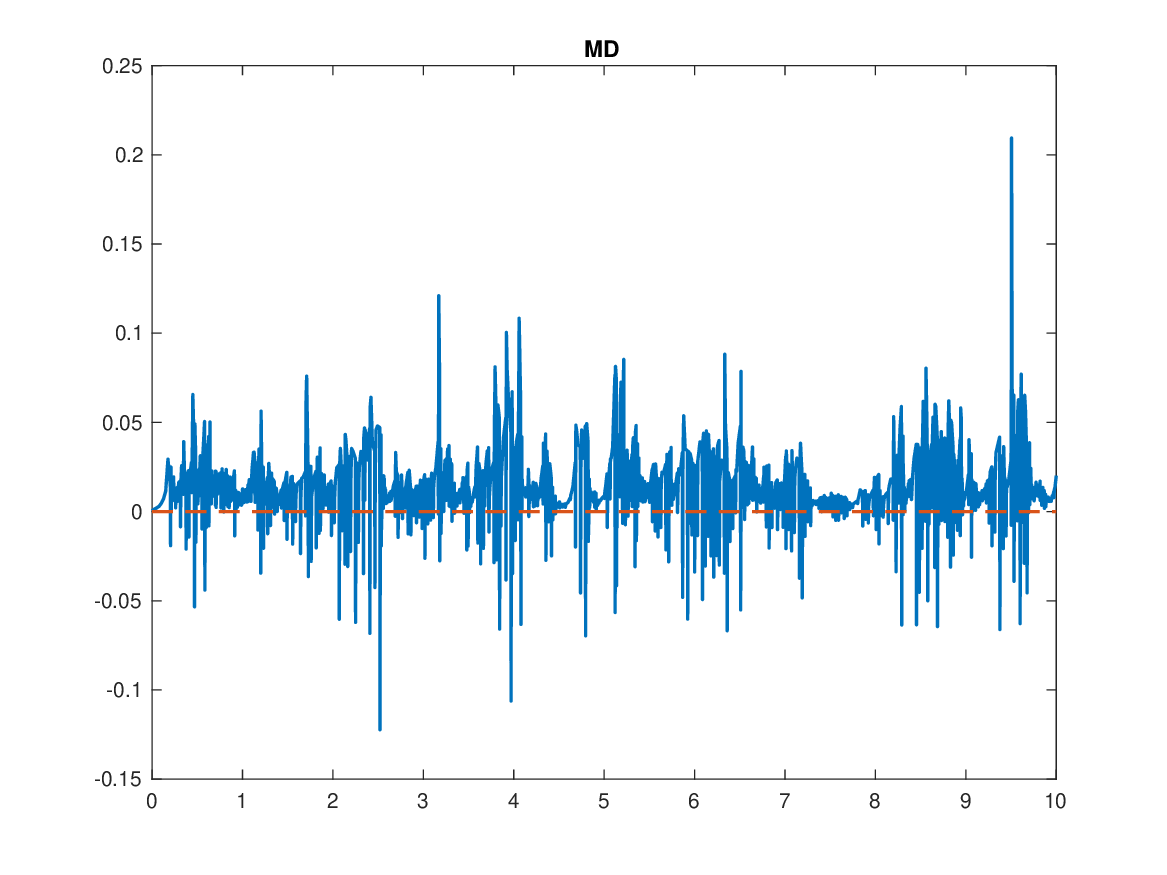}}
\subfloat[$\mathcal{E}_{N}^{\tt CSMD}$]{\includegraphics[width=8.2cm]{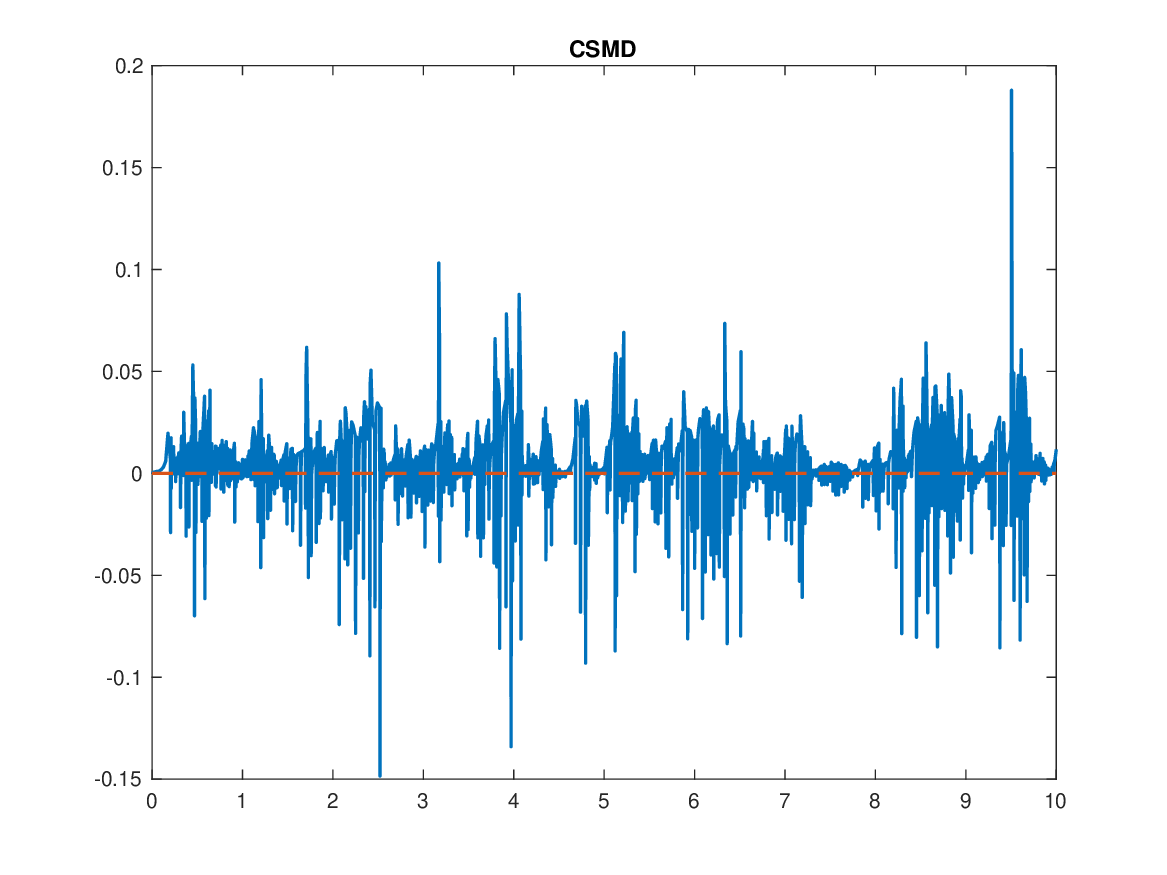}}\\
\subfloat[$\mathcal{E}_{N}^{\tt ND}$]{\includegraphics[width=8.2cm]{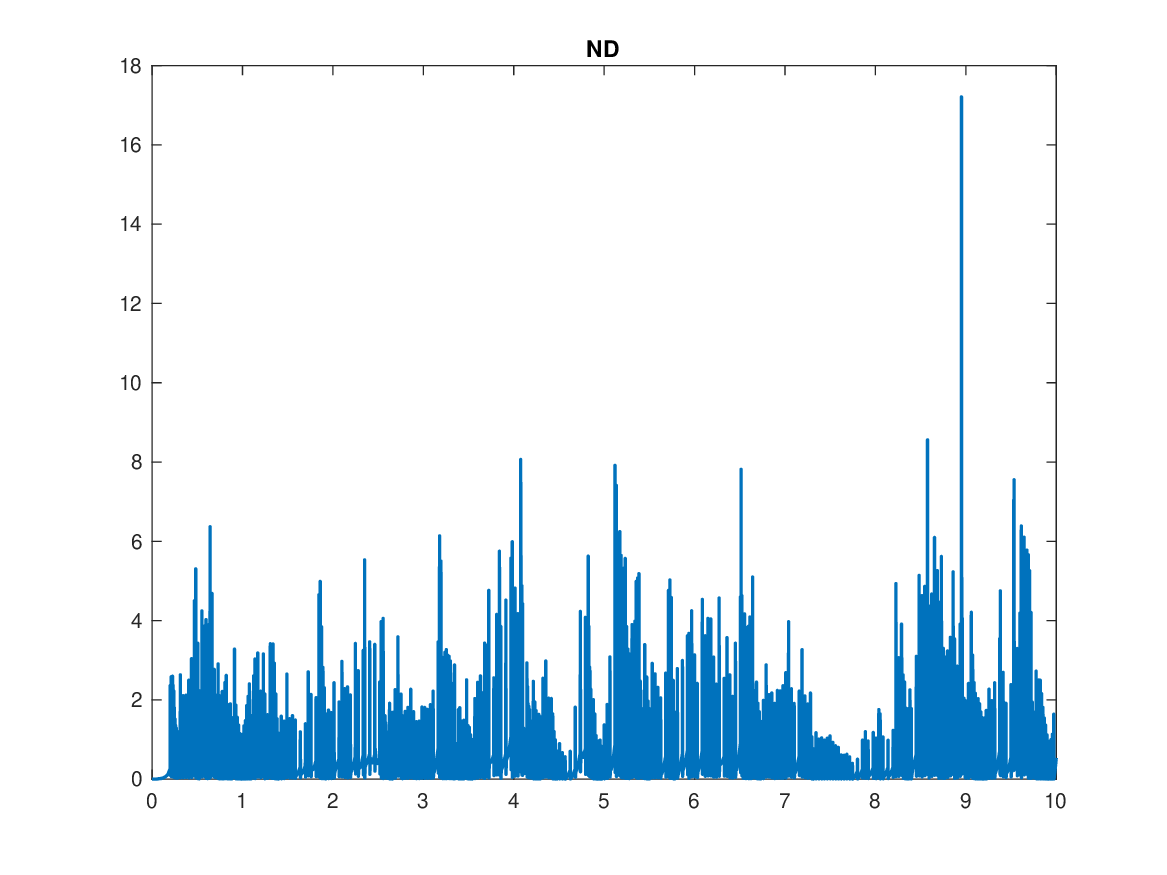}}
\subfloat[$\mathcal{E}_{N}^{\tt VD}$]{\includegraphics[width=8.2cm]{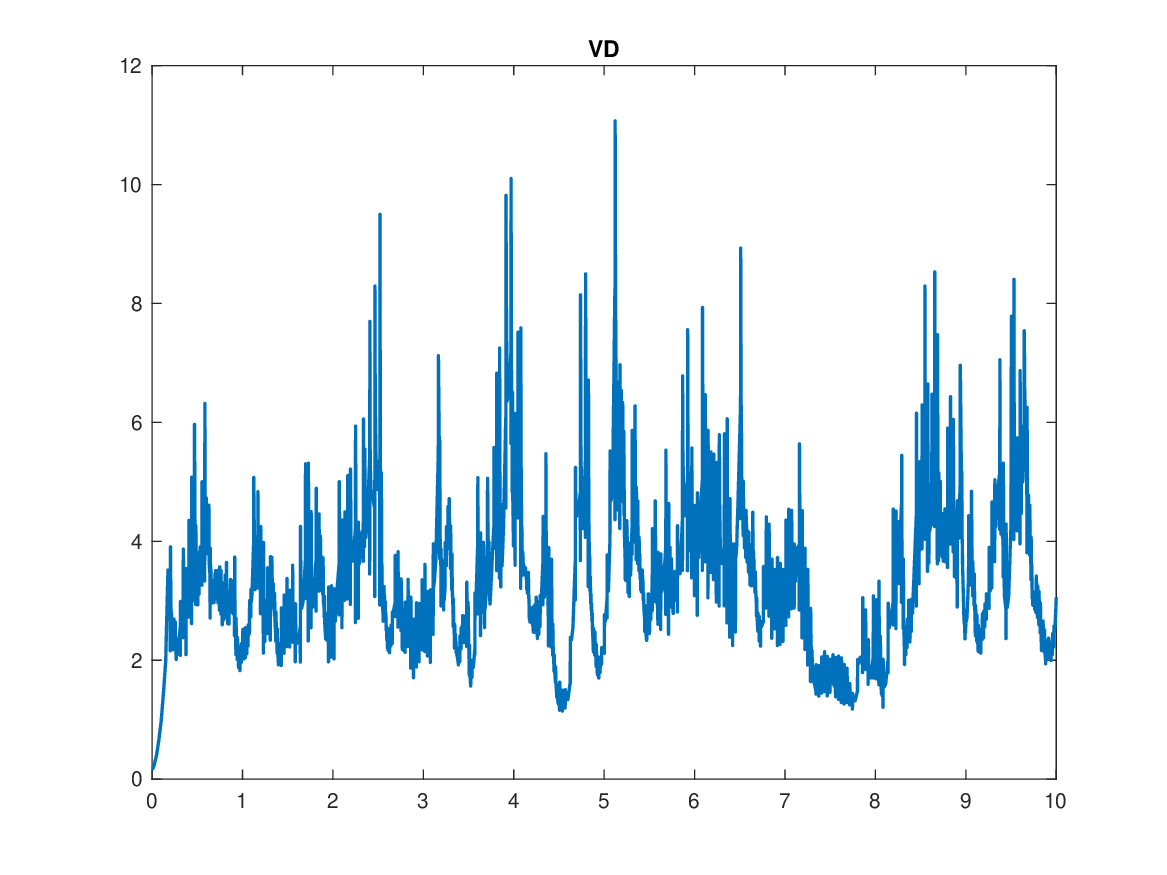}}
\caption{Variable Step DLN \eqref{v2} with $Tol =0.15,\ Re=10,000, \ \theta=\frac{2}{\sqrt{5}}, \ C_s=0.1, \ \mu=0.4$. We see backscatter in $\mathcal{E}_{N}^{\tt MD}$.}
\label{fig:dlntheta2bysqrt5tol0.15T10}
\end{figure}
\begin{figure}[H]
\subfloat[$\mathcal{E}_{N}^{\tt MD}$]{\includegraphics[width=8.2cm]{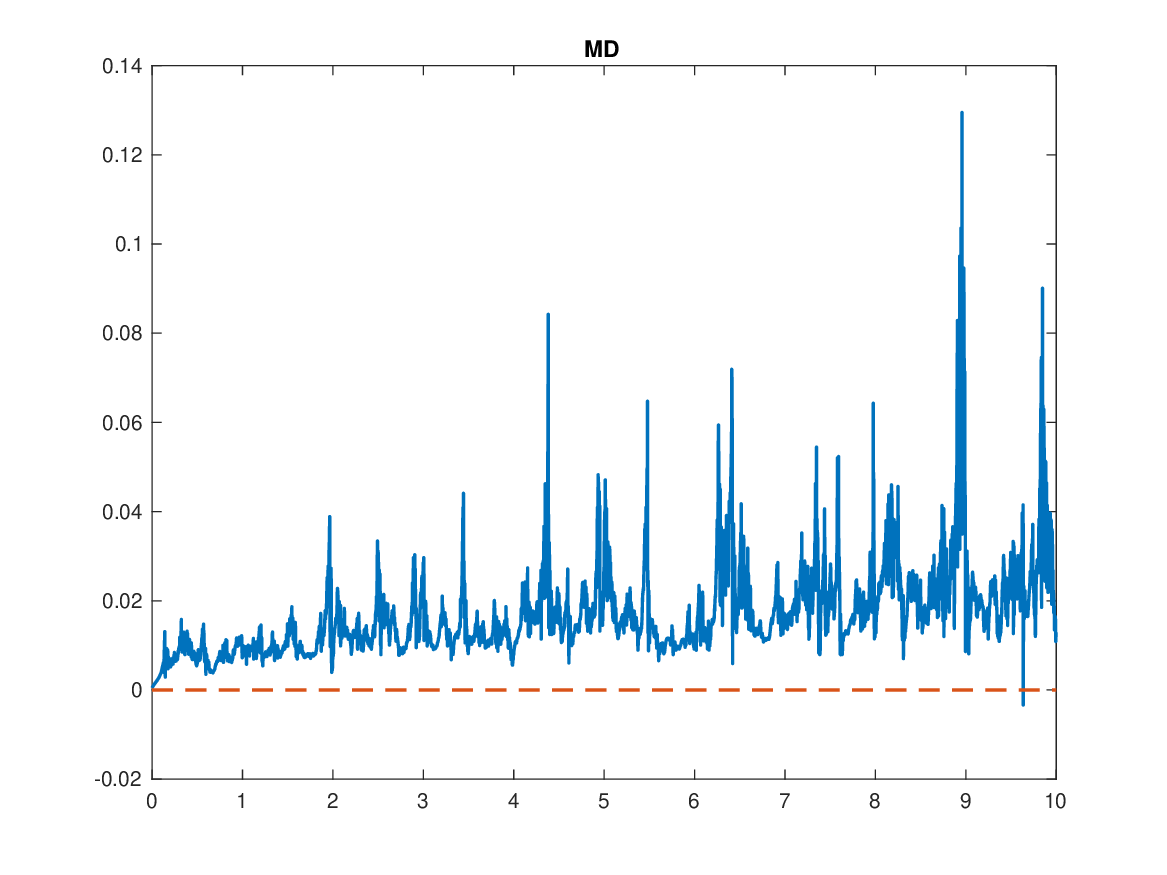}}
\subfloat[$\mathcal{E}_{N}^{\tt CSMD}$]{\includegraphics[width=8.2cm]{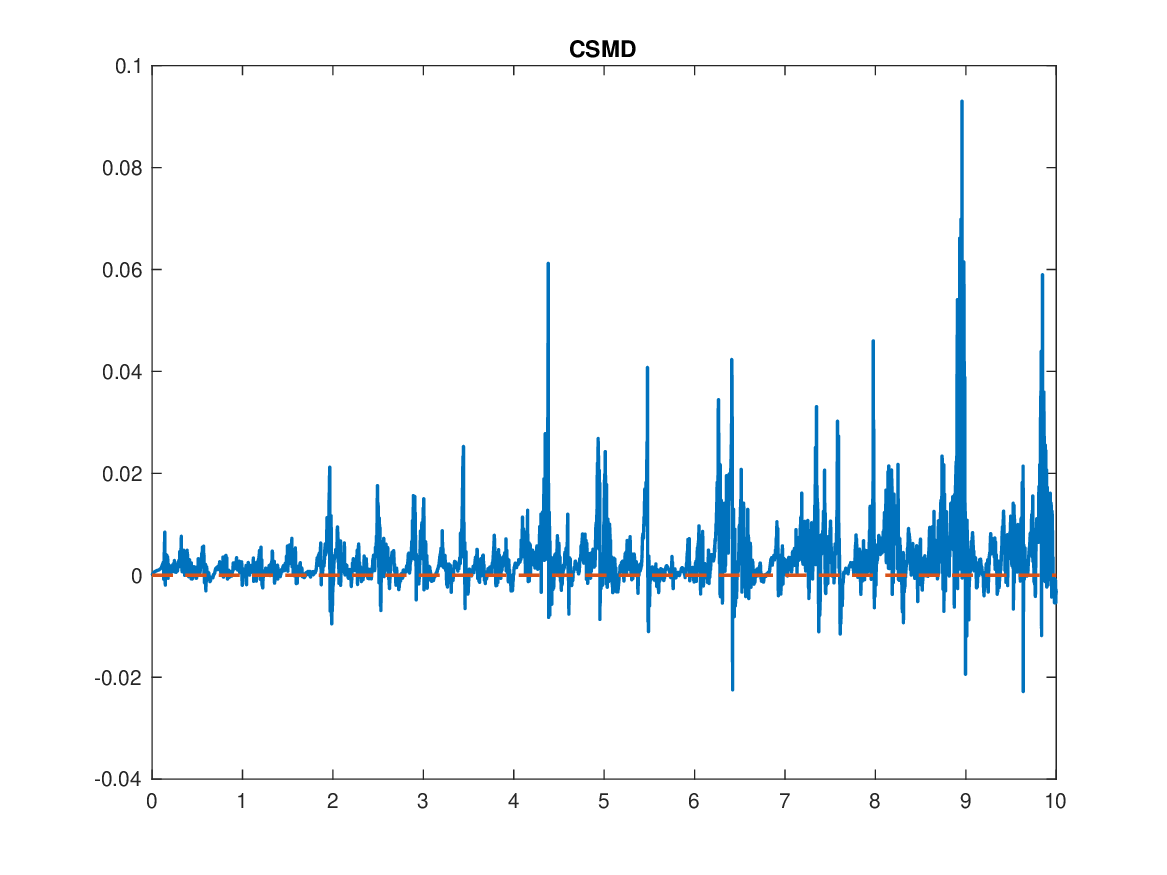}}\\
\subfloat[$\mathcal{E}_{N}^{\tt ND}$]{\includegraphics[width=8.2cm]{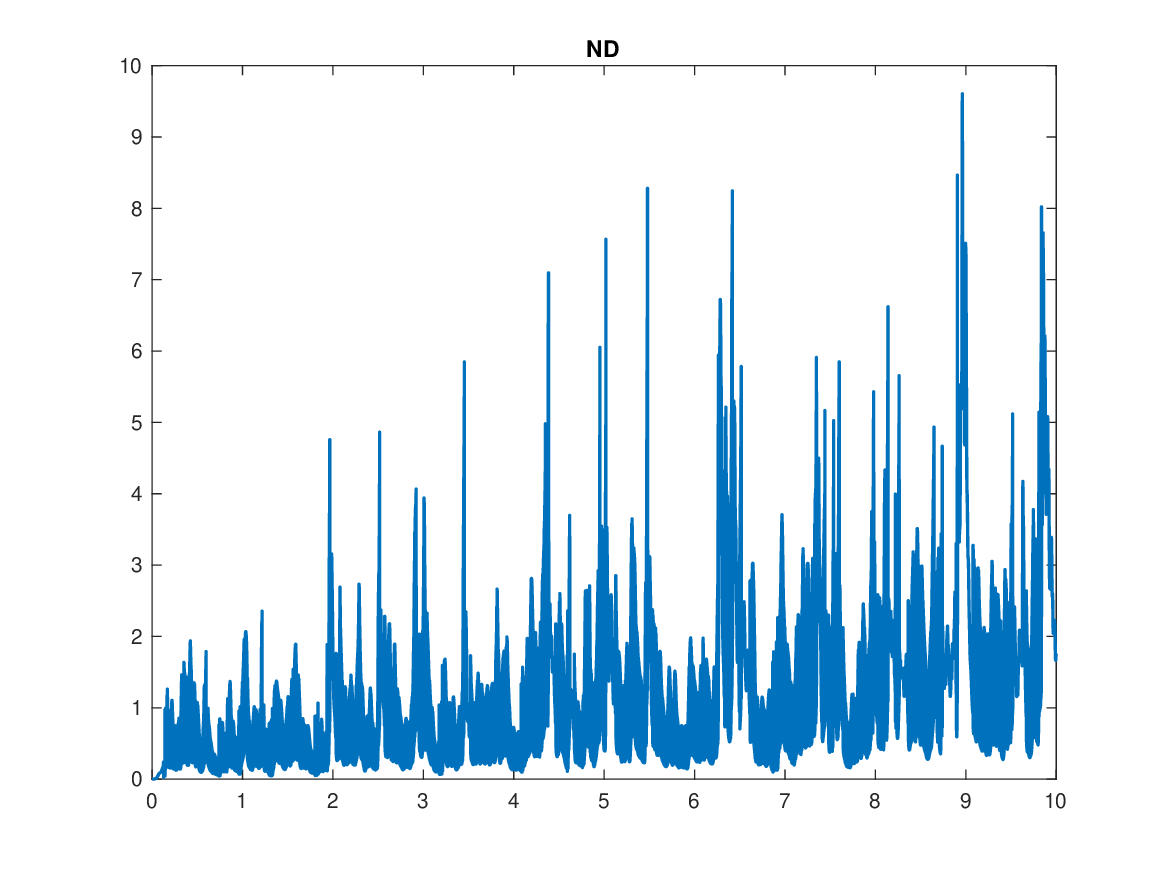}}
\subfloat[$\mathcal{E}_{N}^{\tt VD}$]{\includegraphics[width=8.2cm]{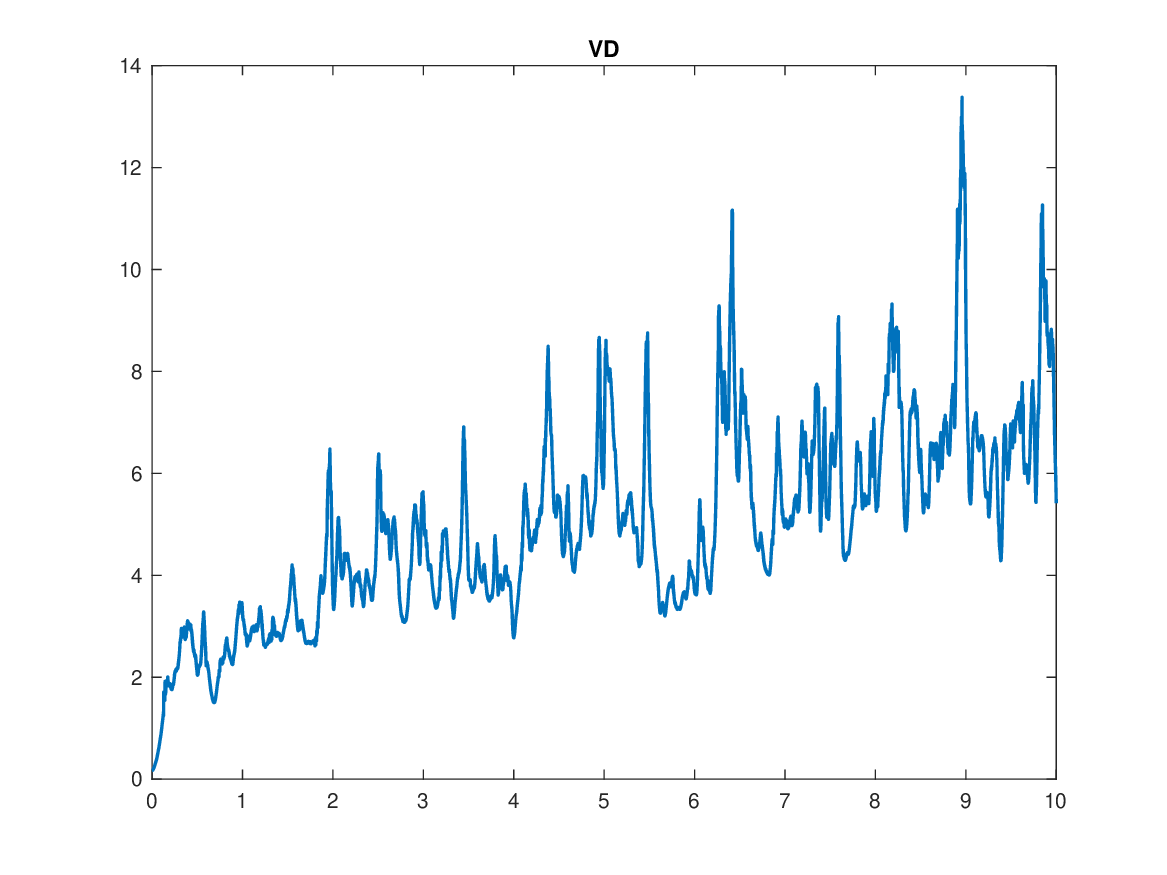}}
\caption{Variable Step DLN \eqref{v2} with $Tol =0.15,\ Re=10,000, \ \theta=\frac{2}{3}, \ C_s=0.1, \ \mu=0.4$. We do not see backscatter in $\mathcal{E}_{N}^{\tt MD}$.}
\label{fig:dlntheta2by3tol0.15T10}
\end{figure}
\begin{figure}[H]
\subfloat[KE, Tol=0.01]{\includegraphics[width=8.2cm]{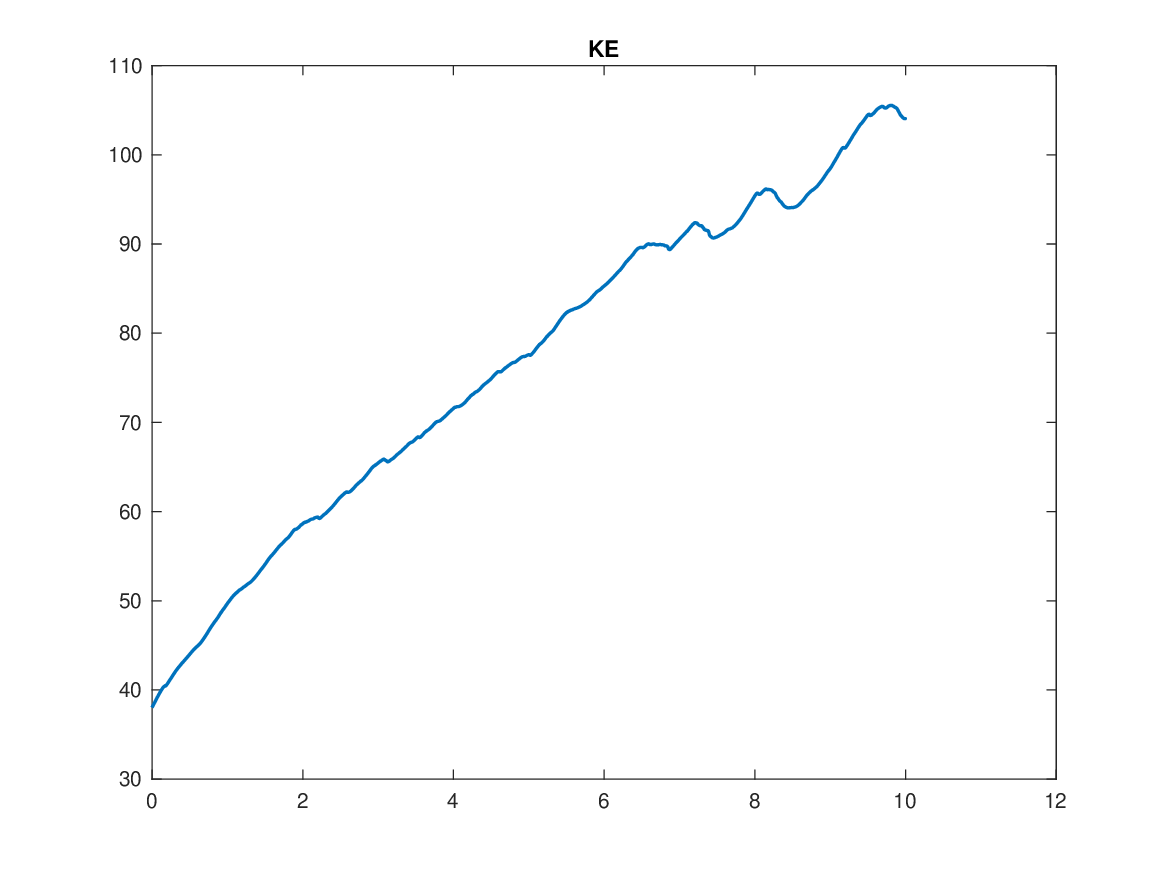}}
\subfloat[KE, Tol=0.05]{\includegraphics[width=8.2cm]{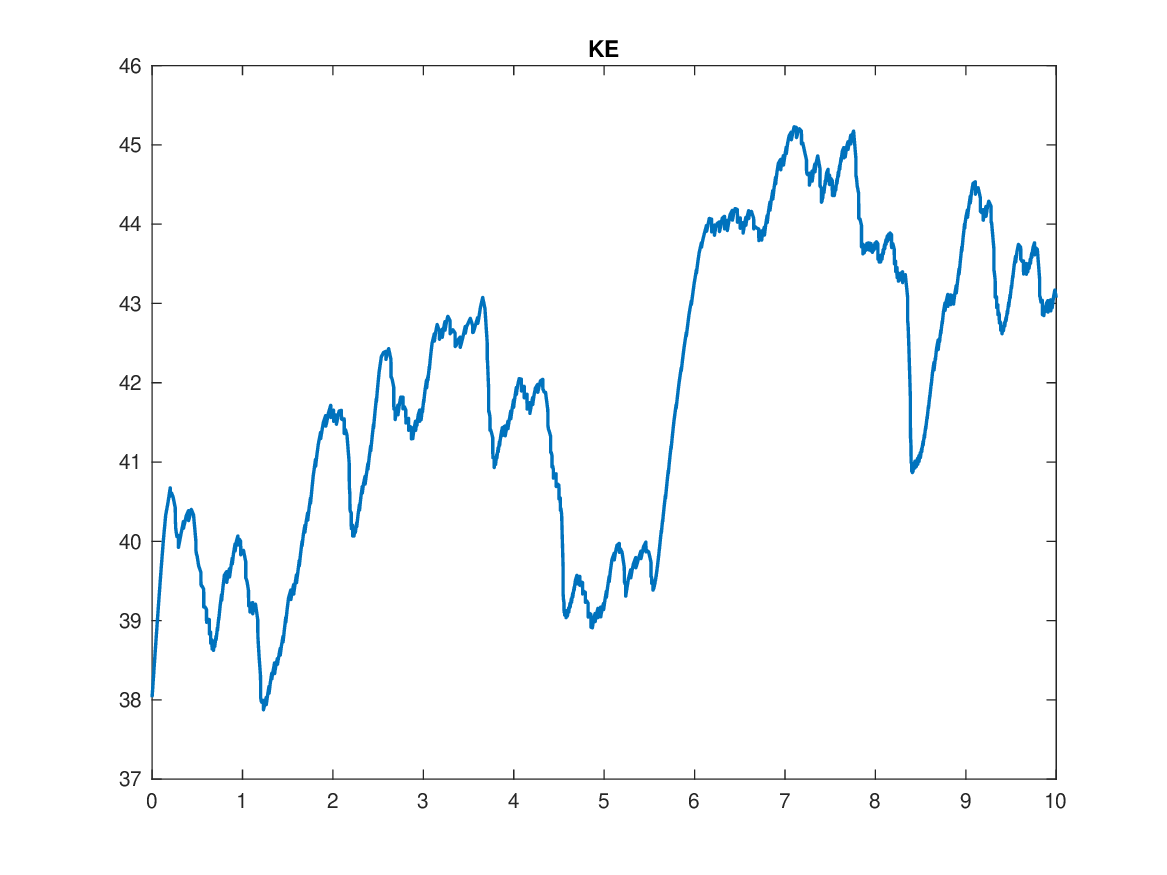}}
\caption{Variable Step DLN \eqref{v2} with $Re=10,000, \ \theta=0.95, \ C_s=0.1, \ \mu=0.4$. The left picture is for no backscatter and right picture is for backscatter.}
\label{fig:kev0.95}
\end{figure}
\begin{table}[H]
\centering
\caption{Total time steps taken to reach $T=10$ while using variable DLN for different values of $\theta$.}
\begin{tabular}{|l|l|l|}
\hline
$\theta$            & Tol           & Total time steps \\ \hline
\textbf{0.98}    & \textbf{0.01} & \textbf{9575}   \\ \hline
0.95             & 0.01          & 6505            \\ \hline
\textbf{0.95}    & \textbf{0.05} & \textbf{1604}   \\ \hline
2/$\sqrt{5}$          & 0.01          & 8988            \\ \hline
2/$\sqrt{5}$        & 0.05          & 5680            \\ \hline
\textbf{2/$\sqrt{5}$} & \textbf{0.15} & \textbf{1973}   \\ \hline
2/3              & 0.01          & 9944            \\ \hline
2/3              & 0.05          & 9575            \\ \hline
2/3              & 0.15          & 7149               \\ \hline
\end{tabular}
\label{tab:totaltime steps}
\end{table}
In \cref{tab:totaltime steps}, for the highlighted values, we notice significant backscatter in MD.
\section{Conclusion}\label{sec:conclusion}
In the report, we propose the variable time-stepping DLN algorithm for the CSM and present a complete numerical analysis for the algorithm. In the stability analysis, we have shown that the numerical solutions are unconditional stable in energy over long term. In the error analysis, we have proved that the numerical velocity converges at second order under mild time step limits if the highest polynomial degrees satisfy $r = 2$ and $s = 1$, which is verified by the first numerical test problem in Subsection \ref{sec:firstexample}. 
It's clear that to get backscatter phenomenon not from ringing property of the method, we need some dissipative methods and we need some control of numerical dissipation, $\mathcal{E}_{N}^{\tt ND}$. We therefore test in Subsection \ref{sec:secondexample} by adapting time step using minimum dissipation criteria.The closer $\theta=1$, the closer DLN method gets to be exactly conservative. If it is exactly conservative, we do not need tight control over $\mathcal{E}_{N}^{\tt ND}$. The further we go away from exactly conservative, the tighter control we need over $\mathcal{E}_{N}^{\tt ND}$ to see what seems to be true. In future, error analysis for semi-implicit DLN algorithm for CSM to avoid time restriction could be proven since it's an important open problem. Futhermore, in $3$D, storage can be an issue and hence analysis of reduced storage penalty method is also an interesting problem.
\section*{Acknowledgement}
We would like to thank Dr. W. J. Layton, for his insightful idea and guidance throughout the research.
\section{Appendix}\label{appendix}
In this appendix, we provide additional some additional tables and figures.
\begin{table}[H]
    	\centering
    	\caption{Errors by $\| \cdot \|_{\infty,0}$-norm and Convergence Rate for the constant DLN with $\theta = 2/\sqrt{5}$}
    	\begin{tabular}{cccccccc}
    		\hline
    		\hline
    		Time step $k$ & Mesh size $h$ & $\| | e^{w} | \|_{\infty,0}$ & Rate 
    		& $\| | \nabla e^{w} | \|_{\infty,0}$ & Rate  
    		& $\| | e^{p} | \|_{\infty,0}$   & Rate
    		\\
    		\hline
    		\hline
    		0.08  & 0.08571  & 6.1375       & -        & 59.5951   & -       & 10.2725  & -
    		\\
    		0.04  & 0.04221  & 0.0499412    & 6.9412   & 1.35769   & 5.4560  & 0.0803944   & 6.9975
    		\\
    		0.02  & 0.02095  & 0.0119888    & 2.0585   & 0.399817  & 1.7637  & 0.0195956   & 2.0366
    		\\
    		0.01  & 0.01048  & 0.00297839   & 2.0091   & 0.103952  & 1.9434  & 0.00502445  & 1.9635
    		\\
    		\hline
    	\end{tabular}
    \end{table}
    \begin{table}[H]
    	\centering
    	\caption{Errors by $\| \cdot \|_{0,0}$-norm and Convergence Rate for the constant DLN with $\theta = 2/\sqrt{5}$}
    	\begin{tabular}{cccccccc}
    		\hline
    		\hline
    		Time step $k$ & Mesh size $h$ & $\| | e^{w} | \|_{0,0}$ & Rate 
    		& $\| | \nabla e^{w} | \|_{0,0}$ & Rate  
    		& $\| | e^{p} | \|_{0,0}$   & Rate
    		\\
    		\hline
    		\hline
    		0.08  & 0.08571  & 8.05856      & -        & 86.5876   & -       & 11.9822  & -
    		\\
    		0.04  & 0.04221  & 0.107272     & 6.2312   & 3.05843   & 4.8233  & 0.143556   & 6.3831
    		\\
    		0.02  & 0.02095  & 0.0249452    & 2.1044   & 0.900625  & 1.7638  & 0.0346417  & 2.0510
    		\\
    		0.01  & 0.01048  & 0.00616932   & 2.0156   & 0.234285  & 1.9427  & 0.00880143 & 1.9767
    		\\
    		\hline
    	\end{tabular}
    \end{table}
    \begin{table}[H]
    	\centering
    	\caption{Errors by $\| \cdot \|_{\infty,0}$-norm and Convergence Rate for the constant DLN with $\theta = 1$}
    	\begin{tabular}{cccccccc}
    		\hline
    		\hline
    		Time step $k$ & Mesh size $h$ & $\| | e^{w} | \|_{\infty,0}$ & Rate 
    		& $\| | \nabla e^{w} | \|_{\infty,0}$ & Rate  
    		& $\| | e^{p} | \|_{\infty,0}$   & Rate
    		\\
    		\hline
    		\hline
    		0.08  & 0.08571  & 6.03148      & -        & 72.2845   & -       & 14.0717    & -
    		\\
    		0.04  & 0.04221  & 0.0499902    & 6.9147   & 1.35784   & 5.7343  & 0.0831369   & 7.4031
    		\\
    		0.02  & 0.02095  & 0.0120016    & 2.0584   & 0.399858  & 1.7638  & 0.0203057   & 2.0336
    		\\
    		0.01  & 0.01048  & 0.00298191   & 2.0089   & 0.103961  & 1.9434  & 0.00512713  & 1.9857
    		\\
    		\hline
    	\end{tabular}
    \end{table}
    \begin{table}[H]
    	\centering
    	\caption{Errors by $\| \cdot \|_{0,0}$-norm and Convergence Rate for the constant DLN with $\theta = 1$}
    	\begin{tabular}{cccccccc}
    		\hline
    		\hline
    		Time step $k$ & Mesh size $h$ & $\| | e^{w} | \|_{0,0}$ & Rate 
    		& $\| | \nabla e^{w} | \|_{0,0}$ & Rate  
    		& $\| | e^{p} | \|_{0,0}$   & Rate
    		\\
    		\hline
    		\hline
    		0.08  & 0.08571  & 8.50684      & -        & 105.23    & -       & 14.0354    & -
    		\\
    		0.04  & 0.04221  & 0.107277     & 6.3092   & 3.05802   & 5.1048  & 0.14397    & 6.6072
    		\\
    		0.02  & 0.02095  & 0.0249479    & 2.1044   & 0.90061   & 1.7636  & 0.0347625  & 2.0502
    		\\
    		0.01  & 0.01048  & 0.0061698    & 2.0156   & 0.234279  & 1.9427  & 0.00883384 & 1.9764
    		\\
    		\hline
    	\end{tabular}
    \end{table}
\begin{figure}[H]
\subfloat[$\mathcal{E}_{N}^{\tt MD}$]{\includegraphics[width=8.2cm]{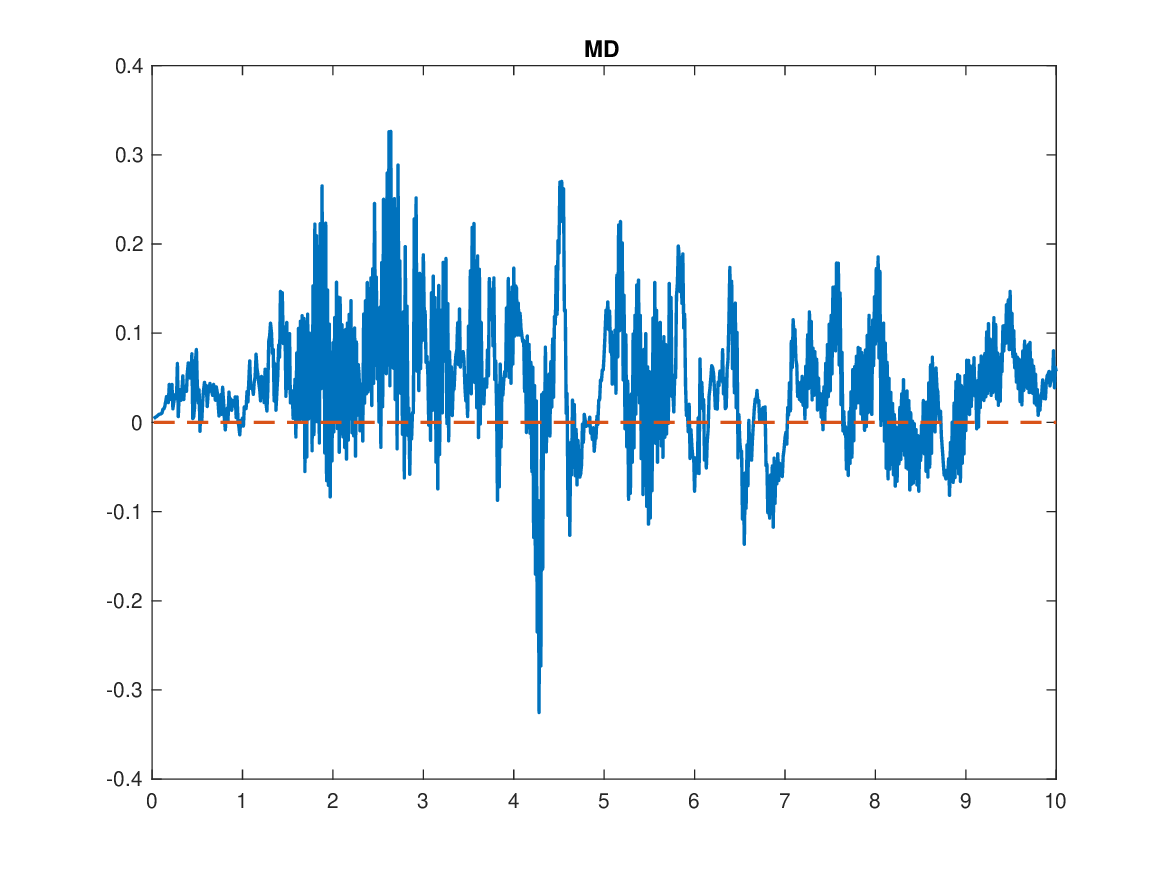}}
\subfloat[$\mathcal{E}_{N}^{\tt CSMD}$]{\includegraphics[width=8.2cm]{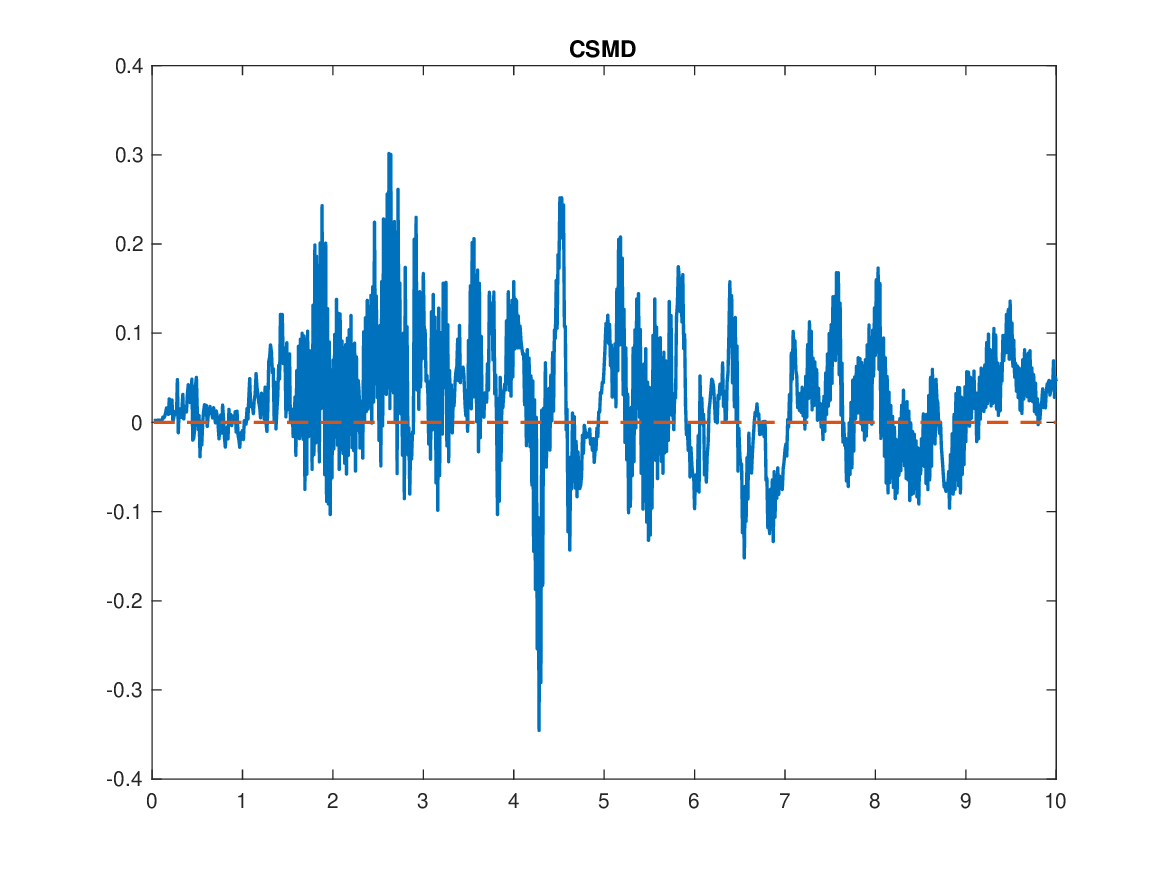}}\\
\subfloat[$\mathcal{E}_{N}^{\tt ND}$]{\includegraphics[width=8.2cm]{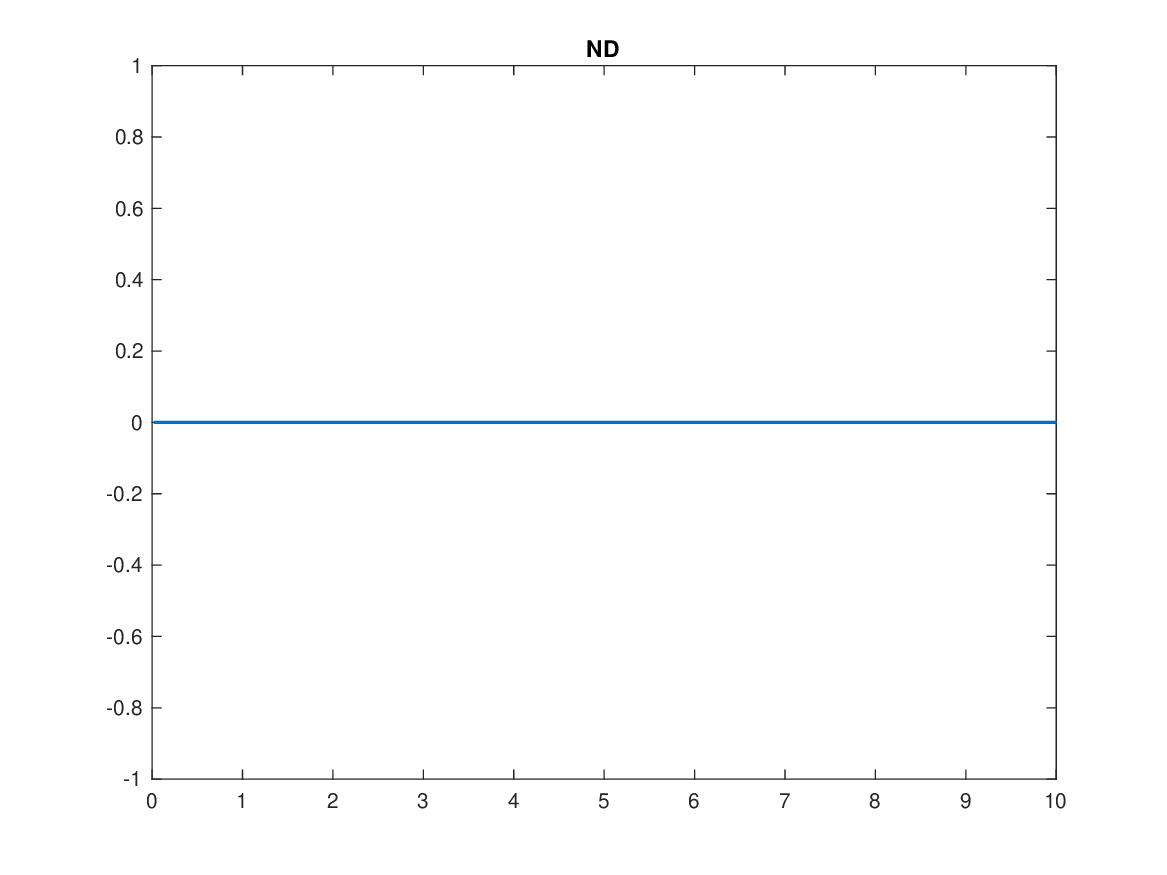}}
\subfloat[$\mathcal{E}_{N}^{\tt VD}$]{\includegraphics[width=8.2cm]{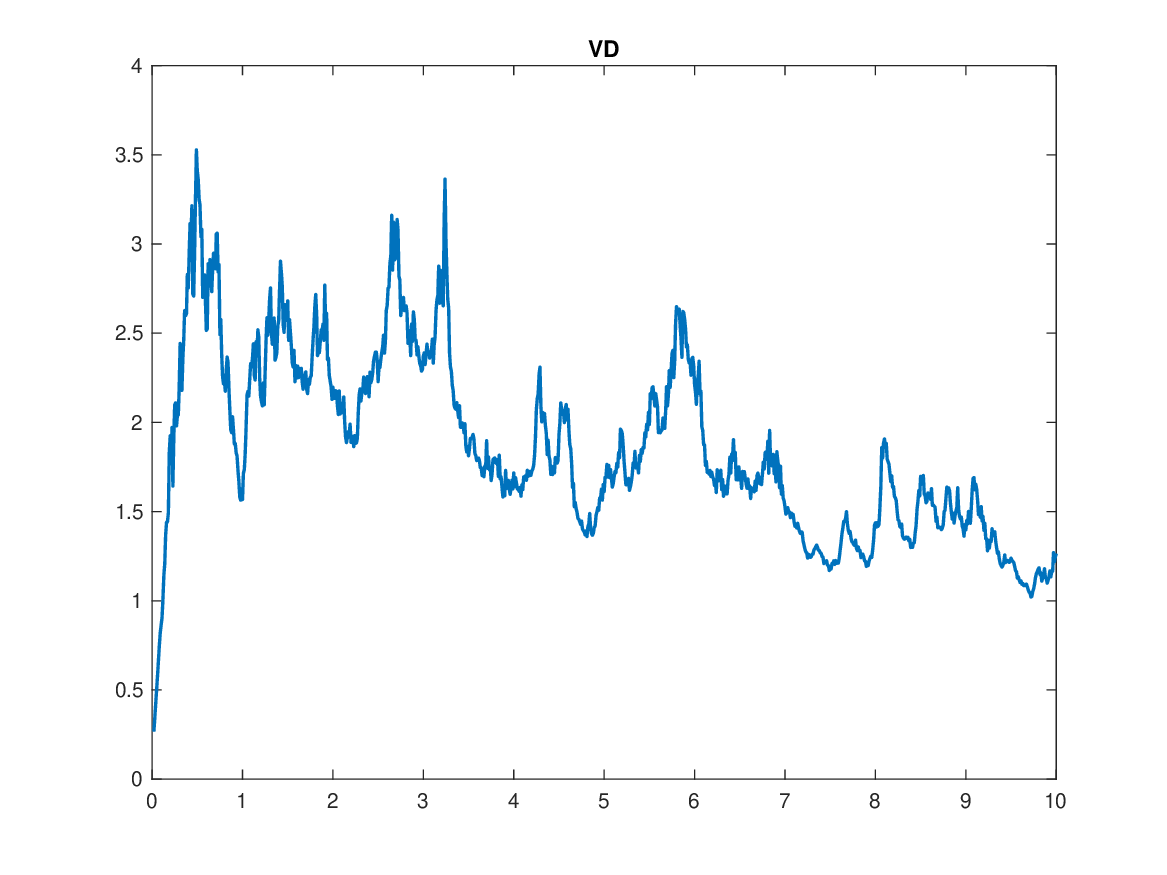}}
\caption{Constant time step DLN \eqref{v2} with $k =0.001,\ Re=10,000, \ \theta=1, \ C_s=0.1, \ \mu=0.4$.}
\label{fig:plotdlnctheta1dt0.001}
\end{figure}
\bibliographystyle{plain}
\bibliography{references}
\end{document}